\documentclass[reqno]{amsart}
\usepackage{amssymb}

 \newtheorem{Theorem}{Theorem}[section]
 \newtheorem{Corollary}[Theorem]{Corollary}
 \newtheorem{Lemma}[Theorem]{Lemma}
 \newtheorem{Proposition}[Theorem]{Proposition}

 \newtheorem{Definition}[Theorem]{Definition}

 \newtheorem{Remark}[Theorem]{Remark}

 \numberwithin{equation}{section}

\begin{document}

\title[concavity property of minimal $L^{2}$ integrals]
{concavity property of minimal $L^{2}$ integrals with Lebesgue measurable gain}

\author{Qi'an Guan}
\address{Qi'an Guan: School of
Mathematical Sciences, Peking University, Beijing 100871, China.}
\email{guanqian@math.pku.edu.cn}

\author{Zheng Yuan}
\address{Zheng Yuan: School of
Mathematical Sciences, Peking University, Beijing 100871, China.}
\email{zyuan@pku.edu.cn}

\thanks{}



\date{\today}

\dedicatory{}

\commby{}


\begin{abstract}
In this article, we present a concavity property of the minimal $L^{2}$ integrals related to multiplier ideal sheaves with Lebesgue measurable gain.
As applications, we give necessary conditions for our concavity degenerating to linearity,
characterizations for 1-dimensional case,
and a characterization for the holding of the equality in  optimal $L^2$ extension problem on open Riemann surfaces with weights may not be subharmonic.
\end{abstract}

\maketitle

\section{Introduction}
The multiplier ideal sheaves related to plurisubharmonic functions plays an important role in complex geometry and algebraic geometry
(see e.g. \cite{tian87,Nadel90,siu96,DEL00,D-K01,demailly-note2000,D-P03,Laz04,siu05,siu09,demailly2010}).
Recall the definition of the multiplier ideal sheaves as follows (see \cite{demailly2010}).

\emph{The multiplier ideal sheaf $\mathcal{I}(\varphi)$ was defined as the sheaf of germs of holomorphic functions $f$ such that
$|f|^{2}e^{-\varphi}$ is locally integrable, where $\varphi$ is a plurisubharmonic function on a complex manifold $M$.}

The strong openness conjecture is $\mathcal{I}(\varphi)=\mathcal{I}_{+}(\varphi):=\cup_{\varepsilon>0}\mathcal{I}((1+\varepsilon)\varphi)$, which was posed by Demailly (see \cite{demailly-note2000}) and be proved by Guan-Zhou \cite{GZopen-c} (the dimension two case was proved by Jonsson-Mustata \cite{JM12}). When $\mathcal{I}(\varphi)=\mathcal{O}$, this conjecture is called the openness conjecture, which was posed by
Demailly-Koll\'{a}r \cite{D-K01}, and was proved by Berndtsson \cite{berndtsson13} (the dimension two case was proved by Favre-Jonsson \cite{FM05j}) by establishing an effectiveness result of the openness conjecture.

Simulated by Berndtsson's effectiveness result,
continuing the solution of the strong openness conjecture \cite{GZopen-c},
Guan-Zhou \cite{GZopen-effect} established a non-sharp effectiveness result of the strong openness conjecture.
Recall that for the first time Guan-Zhou \cite{GZopen-effect} considered the minimal $L^{2}$ integral related to multiplier ideals on the sublevel set $\{\varphi<0\}$
i.e. the pseudoconvex domain $D$.

In \cite{guan_sharp}, by considering all the minimal $L^{2}$ integrals on the sublevels of the weights $\varphi$,
Guan presented a sharp version of the effectiveness result of the strong openness conjecture,
and obtained a concavity property of the minimal $L^{2}$ integrals without gain.
In \cite{guan_general concave}, Guan generalized the concavity property in \cite{guan_sharp} to minimal $L^{2}$ integrals with smooth gain.

In \cite{GM-concave}, Guan-Mi gave some applications of the concavity property in \cite{guan_general concave}:
a necessary condition for the concavity degenerating to linearity,
a characterization for 1-dimensional case,
and a characterization for the holding of the equality in  optimal $L^2$ extension problem on open Riemann surfaces with subharmonic weights.
Recall that if the subharmonic weights degenerate to $0$,
the characterization for the holding of the equality in  optimal $L^2$ extension problem on open Riemann surfaces is the solution of the equality part of the Suita conjecture in \cite{guan-zhou13ap};
if the subharmonic weights degenerate to harmonic,
the characterization for the holding of the equality in  optimal $L^2$ extension problem on open Riemann surfaces is the solution of the equality part of the extended Suita conjecture in \cite{guan-zhou13ap}.

In the present article, we point out that the smooth gain of the general concavity property in \cite{guan_general concave} (see also \cite{GM-concave})
can be replaced by Lebesgue measurable gain (Definition \ref{def:gain} and Theorem \ref{thm:general_concave}).
As applications, we give necessary conditions for our concavity degenerating to linearity (Section \ref{sec:ne}),
characterizations for 1-dimensional case (Section \ref{sec:1-d}),
and a characterization for the holding of the equality in  optimal $L^2$ extension problem on open Riemann surfaces with weights may not be subharmonic (Section \ref{sec:20210808}).

\subsection{Concavity property of minimal $L^{2}$ integrals with Lebesgue measurable gain}\label{sec:main}

Let $M$ be a complex manifold. We call $M$ that satisfies condition $(a)$, if there exists a closed subset $X\subset M$ satisfying the following two statements:

$(a1)$ $X$ is locally negligible with respect to $L^2$ holomorphic functions; i.e., for any local coordinate neighborhood $U\subset M$ and for any $L^2$ holomorphic function $f$ on $U\backslash X$, there exists an $L^2$ holomorphic function $\tilde f$ on $U$ such that $\tilde f|_{U\backslash X}=f$ with the same $L^2$ norm;

$(a2)$ $M\backslash X$ is a Stein manifold.

\

Let $M$ be an $n-$dimensional complex manifold satisfying condition $(a)$,  and let $K_{M}$ be the canonical (holomorphic) line bundle on $M$.
Let $\psi$ be a plurisubharmonic function on $M$,
and let  $\varphi$ be a Lebesgue measurable function on $M$,
such that $\varphi+\psi$ is a plurisubharmonic function on $M$. Take $T=-\sup_{M}\psi$ ($T$ maybe $-\infty$).

\begin{Definition}
\label{def:gain}
We call a positive measurable function $c$ (so-called "\textbf{gain}") on $(T,+\infty)$ in class $\mathcal{P}_T$ if the following two statements hold:

$(1)$ $c(t)e^{-t}$ is decreasing with respect to $t$;

$(2)$ there is a closed subset $E$ of $M$ such that $E\subset \{z\in Z:\psi(z)=-\infty\}$ and for any compact subset $K\subseteq M\backslash E$, $e^{-\varphi}c(-\psi)$ has a positive lower bound on $K$, where $Z$ is some analytic subset of $M$.	
\end{Definition}

\begin{Remark}
We recall a class $\mathcal{P}'_T$ of positive smooth functions in \cite{guan_general concave}. A positive smooth function $c$ on $(T,+\infty)$ in class $\mathcal{P}'_T$ if the following three statements hold:

$(1)'$ $\int_T^{+\infty}c(t)e^{-t}dt<+\infty$;

$(2)'$ $c(t)e^{-t}$ is decreasing with respect to $t$;

$(3)'$ for any compact subset $K\subseteq M$, $e^{-\varphi}c(-\psi)$ has a positive lower bound on $K$.	

We compare these two classes of functions $\mathcal{P}_T$ and $\mathcal{P}'_T$. When $c\in\mathcal{P}_T$, $c$ maybe non-smooth on $(T,+\infty)$ and $\int_T^{+\infty}c(t)e^{-t}dt$ maybe $+\infty$. When $\varphi$ is continuous on $M$, condition $(3)'$ is equivalent to $\liminf_{t\rightarrow+\infty}c(t)>0$.  When $\varphi$ is continuous on $M$ and $\psi\in A(S)$ (see Section \ref{sec:L2}), the decreasing property of $c(t)e^{-t}$ implies that $c\in\mathcal{P}_T$ and $\liminf_{t\rightarrow+\infty}c(t)$ may be equal to $0$.
\end{Remark}

Let $Z_{0}$ be a subset of $\{\psi=-\infty\}$ such that $Z_{0}\cap Supp(\{\mathcal{O}/\mathcal{I(\varphi+\psi)}\})\neq\emptyset$.
Let $U\supseteq Z_{0}$ be an open subset of $M$
and let $f$ be a holomorphic $(n,0)$ form on $U$.
Let $\mathcal{F}\supseteq\mathcal{I}(\varphi+\psi)|_{U}$ be a  analytic subsheaf of $\mathcal{O}$ on $U$.

Denote
\begin{equation*}
\begin{split}
\inf\Bigg\{\int_{\{\psi<-t\}}|\tilde{f}|^{2}e^{-\varphi}c(-\psi):(\tilde{f}-f)\in H^{0}(Z_0,&
(\mathcal{O}(K_{M})\otimes\mathcal{F})|_{Z_0})\\&\&{\,}\tilde{f}\in H^{0}(\{\psi<-t\},\mathcal{O}(K_{M}))\Bigg\},
\end{split}
\end{equation*}
by $G(t;\varphi,\psi,c)$ (so-called "\textbf{minimal $L^{2}$ integrals related to multiplier ideal sheaves}"),  where $t\in[T,+\infty)$, $c$ is a nonnegative function on $(T,+\infty)$,
$|f|^{2}:=\sqrt{-1}^{n^{2}}f\wedge\bar{f}$ for any $(n,0)$ form $f$ and $(\tilde{f}-f)\in H^{0}(Z_0,
(\mathcal{O}(K_{M})\otimes\mathcal{F})|_{Z_0})$ means $(\tilde{f}-f,z_0)\in(\mathcal{O}(K_{M})\otimes\mathcal{F})_{z_0}$ for all $z_0\in Z_0$. If there is no holomorphic holomorphic $(n,0)$ form $\tilde{f}$ on $\{\psi<-t\}$ satisfying $(\tilde{f}-f)\in H^{0}(Z_0,(\mathcal{O}(K_{M})\otimes\mathcal{F})|_{Z_0})$,
we set $G(t;\varphi,\psi,c)=+\infty$.  Without misunderstanding, we denote $G(t;\varphi,\psi,c)$ by $G(t)$, and when we  focus on different $\varphi$, $\psi$ or $c$, we denote it by $G(t;\varphi)$, $G(t;\psi)$ or $G(t;c)$, respectively.

In the present article, we obtain the following concavity for $G(t)$.
\begin{Theorem}
\label{thm:general_concave}
Let $c\in\mathcal{P}_T$. If there exists $t\in[T,+\infty)$ satisfying that $G(t)<+\infty$, then $G(h^{-1}(r))$ is concave with respect to $r\in(\int_{T_1}^{T}c(t)e^{-t}dt,\int_{T_1}^{+\infty}c(t)e^{-t}dt)$, $\lim_{t\rightarrow T+0}G(t)=G(T)$ and $\lim_{t\rightarrow +\infty}G(t)=0$, where $h(t)=\int_{T_1}^{t}c(t_{1})e^{-t_{1}}dt_{1}$ and $T_1\in(T,+\infty)$.
\end{Theorem}

When   $c(t)\in\mathcal{P}'_T$ and $M$ is a Stein manifold,  Theorem \ref{thm:general_concave} is the concavity property in \cite{guan_general concave} (see also \cite{GM-concave}).

Theorem \ref{thm:general_concave}  implies the following corollary.
\begin{Corollary}
\label{infty}
	If $\int_{T_1}^{+\infty}c(t)e^{-t}dt=+\infty$, where $c\in \mathcal{P}_T$, and $f\notin H^{0}(Z_0,
(\mathcal{O}(K_{M})\otimes\mathcal{F})|_{Z_0})$, then $G(t)=+\infty$ for any $t\geq T$, i.e., there is no holomorphic holomorphic $(n,0)$ form $\tilde{f}$ on $\{\psi<-t\}$ satisfying $(\tilde{f}-f)\in H^{0}(Z_0,
(\mathcal{O}(K_{M})\otimes\mathcal{F})|_{Z_0})$ and $\int_{\{\psi<-t\}}|\tilde{f}|^{2}e^{-\varphi}c(-\psi)<+\infty$.
\end{Corollary}

In the following, we give  two corollaries of Theorem \ref{thm:general_concave} when  concavity degenerates to linearity.
\begin{Corollary}
\label{thm:linear}
Let $c\in \mathcal{P}_T$, and let $G(t)\in(0,+\infty)$ for some $t\geq T$, then  $G(h^{-1}(r))$ is concave with respect to $r\in(\int_{T_1}^{T}c(t)e^{-t}dt,\int_{T_1}^{+\infty}c(t)e^{-t}dt]$ and
the following three statements are equivalent:

$(1) $
$G(t)=\frac{G(T_1)}{\int_{T_1}^{+\infty}c(t_1)e^{-t_1}dt_1}\int_{t}^{+\infty}c(t_1)e^{-t_1}dt_1$
holds for any $t\in[T,+\infty)$, i.e., $G(\hat{h}^{-1}(r))$ is linear with respect to $r\in[0,\int_{T}^{+\infty}c(s)e^{-s}ds)$, where $\hat{h}(t)=\int_{t}^{+\infty}c(s)e^{-s}ds$;

$(2)$ $\frac{G(h^{-1}(r_{0}))}{\int_{T_1}^{+\infty}c(t_1)e^{-t_1}dt_1-r_{0}}\leq\lim_{t\rightarrow T+0}\frac{G(t)}{\int_{t}^{+\infty}c(t_1)e^{-t_1}dt_1}$ holds
for some $r_{0}\in(\int_{T_1}^{T}c(t)e^{-t}dt,\int_{T_1}^{+\infty}c(t)e^{-t}dt)$, i.e., $$\frac{G(t_{0})}{\int_{t_{0}}^{+\infty}c(t_{1})e^{-t_{1}}dt_{1}}\leq\lim_{t\rightarrow T+0}\frac{G(t)}{\int_{t}^{+\infty}c(t_1)e^{-t_1}dt_1}$$
holds for some $t_{0}\in(T,+\infty)$;

$(3)$ $\lim_{r\to{\int_{T_1}^{+\infty}c(t_1)e^{-t_1}dt_1}-0}\frac{G(h^{-1}(r))}{\int_{T_1}^{+\infty}c(t)e^{-t}dt-r}\leq\lim_{t\rightarrow T+0}\frac{G(t)}{\int_{t}^{+\infty}c(t_1)e^{-t_1}dt_1}$ holds,
i.e.,
$$\lim_{t\to+\infty}\frac{G(t)}{\int_{t}^{+\infty}c(t_{1})e^{-t_{1}}dt_{1}}\leq\lim_{t\rightarrow T+0}\frac{G(t)}{\int_{t}^{+\infty}c(t_1)e^{-t_1}dt_1}$$
holds. 	
\end{Corollary}

\begin{Remark}
\label{r:1}
	Let $M=\Delta\subset \mathbb{C}$ and let $\psi=\psi+\varphi=2\log|z|$. Let $c(t)\equiv1$ and let $\mathcal{F}=\mathcal{I}(\varphi+\psi)$. Let $f\equiv dz$ and $Z_0=o$ the origin of $\mathbb{C}$. It is clear that $G(\hat{h}^{-1}(r))=2\pi r$ is linear with respect to $r$, where $\hat{h}(t)=\int_{t}^{+\infty}c(l)e^{-l}dl$.
\end{Remark}

Let $c(t)$ be a nonnegative measurable function  on $(T,+\infty)$. Set
\begin{displaymath}
\begin{split}
\mathcal H^2(c,t)=\{\tilde f:\int_{\{\psi<-t\}}|\tilde{f}|^{2}e^{-\varphi}c(-\psi)<+\infty,{\,}(\tilde{f}-f)&\in H^{0}(Z_0,
(\mathcal{O}(K_{M})\otimes\mathcal{F})|_{Z_0})\\&\&{\,}\tilde{f}\in H^{0}(\{\psi<-t\},\mathcal{O}(K_{M}))\},
\end{split}
\end{displaymath}
where $t\in[T,+\infty)$.

\begin{Corollary}
	\label{thm:1}
Let $c\in \mathcal{P}_T$, if $G(t)\in(0,+\infty)$ for some $t\geq T$ and $G(\hat{h}^{-1}(r))$ is linear with respect to $r\in[0,\int_{T}^{+\infty}c(s)e^{-s}ds)$, where $\hat{h}(t)=\int_{t}^{+\infty}c(l)e^{-l}dl$,
 then there is a unique holomorphic $(n,0)$ form $F$ on $M$ satisfying $(F-f)\in H^{0}(Z_0,(\mathcal{O}(K_{M})\otimes\mathcal F)|_{Z_0})$ and $G(t;c)=\int_{\{\psi<-t\}}|F|^2e^{-\varphi}c(-\psi)$ for any $t\geq T$.  Equality
\begin{equation}
	\label{eq:20210412b}
	\int_{\{-t_1\leq\psi<-t_2\}}|F|^2e^{-\varphi}a(-\psi)=\frac{G(T_1;c)}{\int_{T_1}^{+\infty}c(t)e^{-t}dt}\int_{t_2}^{t_1} a(t)e^{-t}dt
\end{equation}
holds for any nonnegative measurable function $a$ on $(T,+\infty)$, where $+\infty\geq t_1>t_2\geq T$.

Furthermore, if $\mathcal H^2(\tilde{c},t_0)\subset\mathcal H^2(c,t_0)$ for some $t_0\geq T$, where $\tilde{c}$ is a nonnegative measurable function on $(T,+\infty)$, we have
\begin{equation}
	\label{eq:20210412a}
	G(t_0;\tilde{c})=\int_{\{\psi<-t_0\}}|F|^2e^{-\varphi}\tilde{c}(-\psi)=\frac{G(T_1;c)}{\int_{T_1}^{+\infty}c(s)e^{-s}ds}\int_{t_0}^{+\infty} \tilde{c}(s)e^{-s}ds.
	\end{equation}	
\end{Corollary}

 When $c(t)\in\mathcal{P}'_T$ and $M$ is a Stein manifold, Corollary \ref{thm:linear} and Corollary \ref{thm:1} can be referred to \cite{GM-concave} (when $c\equiv1$, $M$ is a Stein manifold, $\varphi$ is a smooth plurisubharmonic function on $M$ and $\{\psi=-\infty\}$ is a closed subset of $M$, Xu-Zhou \cite{xu-phd} also get the existence of $F$ in Corollary \ref{thm:1} independently).

\begin{Remark}
\label{r:c} Let $\tilde{c}\in\mathcal{P}_T$, if $\mathcal H^2(\tilde{c},t_1)\subset\mathcal H^2(c,t_1)$, then $\mathcal H^2(\tilde{c},t_2)\subset\mathcal H^2(c,t_2)$, where $t_1>t_2>T$. In the following, we give some sufficient conditions of
	$\mathcal H^2(\tilde{c},t_0)\subset\mathcal H^2(c,t_0)$ for $t_0> T$:
	
	$(1)$ $\tilde{c}\in\mathcal{P}_T$ and $\lim_{t\rightarrow+\infty}\frac{\tilde{c}(t)}{c(t)}>0$. Especially,  $\tilde{c}\in\mathcal{P}_T$, $c$ and $\tilde c$ are smooth on $(T,+\infty)$ and $\frac{d}{dt}(\log (\tilde{c}(t))\geq \frac{d}{dt}(\log c(t))$;
	
	$(2)$ $\tilde{c}\in\mathcal{P}_T$, $\mathcal H^2(c,t_0)\not=\emptyset$ and there exists $t>t_0$, such that $\{\psi<-t\}\subset\subset\{\psi<-t_0\}$, $\{z\in\overline{\{\psi<-t\}}:\mathcal{I}(\varphi+\psi)_z\not=\mathcal{O}_z\}\subset Z_0$ and $\mathcal{F}|_{\overline{\{\psi<-t\}}}=\mathcal{I}(\varphi+\psi)|_{\overline{\{\psi<-t\}}}$.
	
	The sufficiency of condition $(1)$ is clear. For condition (2), assume that $\mathcal H^2(\tilde{c},t_0)\not=\emptyset$, then the following inequality gives the  sufficiency of condition $(2)$:
	\begin{displaymath}
	\begin{split}
		&\int_{\{\psi<-t_0\}}|\tilde F|^2e^{-\varphi}c(-\psi)\\
		\leq&2\int_{\{\psi<-t\}}|\tilde F-F|^2e^{-\varphi}c(-\psi)+2\int_{\{\psi<-t\}}|F|^2e^{-\varphi}c(-\psi)\\
		&+\int_{\{-t\leq\psi<-t_0\}}|\tilde F|^2e^{-\varphi}c(-\psi)\\
		\leq&2C\int_{\{\psi<-t\}}|\tilde F-F|^2e^{-\varphi-\psi}+2\int_{\{\psi<-t\}}|F|^2e^{-\varphi}c(-\psi)\\
		&+\frac{\sup_{s\in(t_0,t]}c(s)}{\inf_{s\in(t_0,t]}\tilde{c}(s)}\int_{\{\psi<-t_0\}}|\tilde F|^2e^{-\varphi}\tilde{c}(-\psi)\\
		<&+\infty,
	\end{split}			
	\end{displaymath}
	where $\tilde F\in\mathcal H^2(\tilde{c},t_0)$ and $F\in\mathcal H^2(c,t_0)$.
\end{Remark}

\subsection{Applications}\label{sec:app}

In this section, we give some applications of our concavity property.

\subsubsection{Applications in  optimal $L^2$ extension theorem}\label{sec:L2}
Let $M$ be an $n-$dimensional complex manifold, and let $S$ be an analytic subset of $M$. Let $dV_{M}$ be a continuous volume form on $M$.

We consider a class of plurisubharmonic functions $\Phi$ from $M$ to $[-\infty,+\infty)$, such that

$(1)$ $S\subset\Phi^{-1}(-\infty)$, and $\Phi^{-1}(-\infty)$ is a closed subset of some analytic subset of $M$ satisfying that $\Phi$ has locally lower bound on $M\backslash\Phi^{-1}(-\infty)$.

$(2)$ if $S$ is $l-$dimensional around a point $x\in S_{reg}$, there is a local coordinate $(z_1,...,z_n)$ on a neighborhood $U$ of $x$ such that $z_{l+1}=...=z_n=0$ on $S\cap U$ and
$$\sup_{U-S}|\Phi(z)-(n-l)\log\sum_{l+1}^{n}|z_j|^2|<+\infty.$$

The set of such polar functions $\Phi$ will be denoted by $A(S)$.  We call $\Phi$ is in class $A'(S)$, if the condition $(2)$ is replaced by $(2)'$:

$(2)'$ if $S$ is $l-$dimensional around a point $x\in S_{reg}$, there is a local coordinate $(z_1,...,z_n)$ on a neighborhood $U$ of $x$ such that $z_{l+1}=...=z_n=0$ on $S\cap U$ and
$\Phi(z)-(n-l)\log\sum_{l+1}^{n}|z_j|^2$ is continuous on $U$.

Let $\psi\in A(S)$. Following \cite{ohsawa5} (see also \cite{guan-zhou13ap}),  one can define a positive measure $dV_{M}[\psi]$ on $S_{reg}$ as the minimum element of the partially ordered set of positive measures $d\mu$ satisfying
$$\int_{S_l}fd\mu\geq\limsup_{t\rightarrow+\infty}\frac{2(n-l)}{\sigma_{2n-2l-1}}\int_{M}\mathbb I_{\{-t-1<\psi<-t\}}fe^{-\psi}dV_{M}$$
for any nonnegative continuous function $f$ with $suppf\subset\subset M$. Here denote by $\sigma_m$ the volume of the unit sphere in $\mathbb R^{m+1}$.
If $\psi\in A'(S)$, then $dV_{M}[\psi]|_{S_l}$ is a continuous volume form on $S_l$ and $dV_{M}[\psi+h]|_{S_l}=e^{-h}dV_{M}[\psi]|_{S_l}$ (see \cite{guan-zhou13ap}), where $h$ is a continuous function on $M$.

Let us recall a class of complex manifolds (see \cite{guan-zhou13ap}).
	Let $M$ be a complex manifold with the volume form $dV_{M}$, and let $S$ be an analytic subset of $M$.
	We say $(M,S)$ satisfies condition $(ab)$ if  there exists a closed subset $X\subset M$ satisfying the following statements:
	
	$(a)$ $X$ is locally negligible with respect to $L^2$ holomorphic functions;
	
	 $(b)$ $M\backslash X$ is a Stein manifold which intersects with every component of $S$, such that $S_{sing}\subset X$.
	
\

We give the following  $L^{2}$ extension theorem with an optimal estimate. When $c(t)$ is continuous, the theorem can be referred to \cite{guan-zhou13ap}.

\begin{Theorem}
	\label{thm:L2}
	Let $(M,S)$ satisfy condition $(ab)$. Let $\psi\in A(S)$ satisfying $\psi<-T$ on $M$. Let $\varphi$ be a continuous function on $M$, such that $\varphi+\psi$ is plurisubharmonic on $M$. Let $c(t)$ be a positive function on $(T,+\infty)$ such that $c(t)e^{-t}$ is decreasing and $\int_{T}^{+\infty}c(t)e^{-t}dt<+\infty$. Then for any  holomorphic section $f$ of $K_M|_S$ on $S$, such that
$$\sum_{k=1}^{n}\frac{\pi^k}{k!}\int_{S_{n-k}}\frac{|f|^2}{dV_{M}}e^{-\varphi}dV_{M}[\psi]<+\infty,$$
there exists a holomorphic $(n,0)$ form $F$ on $M$ such that $F|_S=f$ and
$$\int_{M}|F|^2e^{-\varphi}c(-\psi)\leq\left(\int_T^{+\infty}c(t)e^{-t}dt\right)\sum_{k=1}^{n}\frac{\pi^k}{k!}\int_{S_{n-k}}\frac{|f|^2}{dV_{M}}e^{-\varphi}dV_{M}[\psi].$$

\end{Theorem}

By the definition of $dV_{M}[\psi]$, we know $\frac{|f|^2}{dV_{M}}dV_{M}[\psi]$ is independent of the choice of $dV_M$ (see \cite{guan-zhou13ap}).

Denote that $\|f\|_S:=\sum_{k=1}^{n}\frac{\pi^k}{k!}\int_{S_{n-k}}\frac{|f|^2}{dV_{M}}e^{-\varphi}dV_{M}[\psi]$ and $\|F\|_M:=\int_{M}|F|^2e^{-\varphi}c(-\psi)$.
Let $\mathcal{F}|_{Z_0}=\mathcal{I}(\psi)|_{S_{reg}}$ and choose  the $f$ in the definition of $G(t)$ by any holomorphic extension of the $f$ in Theorem \ref{thm:L2}. Then $G(T)=\inf\{\|F\|_M:F$ is a holomorphic extension of $f$ from $S$ to $M$$\}$, and Theorem \ref{thm:L2} tells us that
\begin{equation}
	\label{eq:210628a}G(T)\leq\left(\int_T^{+\infty}c(t)e^{-t}dt\right)\|f\|_S
\end{equation}
(when $G(T)<+\infty$, Lemma \ref{lem:A} shows that there exists a holomorphic extension $F$ of $f$ such that $G(T)=\|F\|_M$).

Using Corollary \ref{thm:1} and Theorem \ref{thm:L2}, we obtain  a necessary condition of inequality \eqref{eq:210628a} becomes an equality.

\begin{Theorem}
	\label{thm:ll1}
Let $(M,S)$ satisfy condition $(ab)$. Let $\psi\in A(S)$, and let $\psi<-T$. Let $\varphi$ be a continuous function on $M$, such that $\varphi+\psi$ is plurisubharmonic on $M$. Let $c(t)$ be a positive function on $(T,+\infty)$ such that $c(t)e^{-t}$ is decreasing and $\int_{T}^{+\infty}c(t)e^{-t}dt<+\infty$. Let $f$ be a  holomorphic section  of $K_M|_S$ on $S$, such that
$$\sum_{k=1}^{n}\frac{\pi^k}{k!}\int_{S_{n-k}}\frac{|f|^2}{dV_{M}}e^{-\varphi}dV_{M}[\psi]<+\infty.$$
If $G(T)=\left(\int_T^{+\infty}c(t)e^{-t}dt\right)\|f\|_S,$ then $G(\hat{h}^{-1}(r))$ is linear with respect to $r$ and there exists a  unique holomorphic $(n,0)$ form $F$ on $M$ such that $F|_S=f$ and $G(T)=\|F\|_M$.
	
For any $t\geq T$, there exists a unique holomorphic $(n,0)$ form $F_t$ on $\{\psi<-t\}$ such that $F_t|_S=f$ and
$$\int_{\{\psi<-t\}}|F_t|^2e^{-\varphi}c(-\psi)\leq\left(\int_t^{+\infty}c(l)e^{-l}dl\right)\sum_{k=1}^{n}\frac{\pi^k}{k!}\int_{S_{n-k}}\frac{|f|^2}{dV_{M}}e^{-\varphi}dV_{M}[\psi].$$
In fact, $F_t=F$ on $\{\psi<-t\}$.

If $\mathcal H^2(\tilde{c},t)\subset\mathcal H^2(c,t)$ for some $t\geq T$, where $\tilde{c}$ is a nonnegative measurable function on $(T,+\infty)$, then there exists a unique holomorphic $(n,0)$ form $F_t$ on $\{\psi<-t\}$ such that $F_t|_S=f$ and
$$\int_{\{\psi<-t\}}|F_t|^2e^{-\varphi}\tilde{c}(-\psi)\leq\left(\int_t^{+\infty}\tilde{c}(l)e^{-l}dl\right)\sum_{k=1}^{n}\frac{\pi^k}{k!}\int_{S_{n-k}}\frac{|f|^2}{dV_{M}}e^{-\varphi}dV_{M}[\psi].$$
In fact, $F_t=F$ on $\{\psi<-t\}$.
\end{Theorem}

When $c(t)\in \mathcal{P}'_T$ and $M$ is a Stein manifold, Theorem \ref{thm:ll1} was obtained by Guan-Mi in \cite{GM-concave}.

Using Theorem \ref{thm:L2}, we obtain the following optimal $L^2$ extension theorem.
\begin{Corollary}
\label{c:L2b}Let $M$ be an $n-$dimensional Stein manifold, and let $S$ be an analytic subset of $M$.	Let $\psi_1\in A(S)$, and let $\psi_2$ be a plurisubharmonic function on $M$ such that $\psi=\psi_1+\psi_2<-T$ on $M$.  Let $\varphi$ be a Lebesgue measurable function on $M$ such that $\varphi+\psi_2$ is  plurisubharmonic on $M$.  Let $c(t)$ be a positive function on $(T,+\infty)$, such that $c(t)e^{-t}$ is decreasing, $\int_{T}^{+\infty}c(t)e^{-t}dt<+\infty$ and $e^{-\varphi}c(-\psi)$ has locally a positive lower bound on $M\backslash Z$, where $Z$ is some analytic subset of $M$. For any  holomorphic section $f$ of $K_M|_{S_{reg}}$ on $S_{reg}$ satisfying
$$\sum_{k=1}^{n}\frac{\pi^k}{k!}\int_{S_{n-k}}\frac{|f|^2}{dV_{M}}e^{-\varphi-\psi_2}dV_{M}[\psi_1]<+\infty,$$
there exists a holomorphic $(n,0)$ form $F$ on $M$ such that $F|_{S_{reg}}=f$ and
$$\int_{M}|F|^2e^{-\varphi}c(-\psi)\leq\left(\int_T^{+\infty}c(t)e^{-t}dt\right)\sum_{k=1}^{n}\frac{\pi^k}{k!}\int_{S_{n-k}}\frac{|f|^2}{dV_{M}}e^{-\varphi-\psi_2}dV_{M}[\psi_1].$$
	\end{Corollary}

 Especially, when $c\equiv1$ and $\psi_1=2\log|w|$, where $w$ is a holomorphic function on $M$, such that $dw$ does not vanish identically on any branch of $w^{-1}(0)$ and $S_{reg}=\{z\in M:w(z)=0\,\&\,dw(z)\not=0\}$, Corollary \ref{c:L2b} can be referred to \cite{guan-zhou12} (see also \cite{guan-zhou13ap}).

Denote that $\|f\|_S^{*}:=\sum_{k=1}^{n}\frac{\pi^k}{k!}\int_{S_{n-k}}\frac{|f|^2}{dV_{M}}e^{-\varphi-\psi_2}dV_{M}[\psi_1]$.
Let $\mathcal{F}|_{Z_0}=\mathcal{I}(\psi_1)|_{S_{reg}}$ and choose  the $f$ in the definition of $G(t)$ by any holomorphic extension of the $f$ in Corollary \ref{c:L2b}. Then $G(T)=\inf\{\|F\|_M:F$ is a holomorphic extension of $f$ from $S$ to $M$$\}$, and Corollary \ref{c:L2b} tells us that
\begin{equation}
	\label{eq:210628b}G(T)\leq\left(\int_T^{+\infty}c(t)e^{-t}dt\right)\|f\|_S^*.
\end{equation}

Similarly to Theorem \ref{thm:ll1}, we give  a necessary condition of  inequality \eqref{eq:210628b} becomes an equality.

\begin{Corollary}
\label{thm:ll2}
	Let $M$ be an $n-$dimensional Stein manifold, and let $S$ be an analytic subset of $M$.	Let $\psi_1\in A(S)$, and let $\psi_2$ be a plurisubharmonic function on $M$ such that $\psi=\psi_1+\psi_2<-T$ on $M$.  Let $\varphi$ be a Lebesgue measurable function on $M$ such that $\varphi+\psi_2$ is  plurisubharmonic on $M$.  Let $c(t)\in\mathcal{P}_T$  such that $\int_{T}^{+\infty}c(t)e^{-t}dt<+\infty$. Let $f$ be a holomorphic section of $K_M|_{S_{reg}}$ on $S_{reg}$ satisfying
$$\sum_{k=1}^{n}\frac{\pi^k}{k!}\int_{S_{n-k}}\frac{|f|^2}{dV_{M}}e^{-\varphi-\psi_2}dV_{M}[\psi_1]<+\infty.$$
	If $G(T)=\left(\int_T^{+\infty}c(t)e^{-t}dt\right)\|f\|_S^*,$ then $G(\hat{h}^{-1}(r))$ is linear with respect to $r$ and there exists a unique   holomorphic $(n,0)$ form $F$ on $M$ such that $F|_S=f$ and $G(T)=\|F\|_M$.

For any $t\geq T$, there exists a unique holomorphic $(n,0)$ form $F_t$ on $\{\psi<-t\}$ such that $F_t|_S=f$ and
$$\int_{\{\psi<-t\}}|F_t|^2e^{-\varphi}c(-\psi)\leq\left(\int_t^{+\infty}c(l)e^{-l}dl\right)\sum_{k=1}^{n}\frac{\pi^k}{k!}\int_{S_{n-k}}\frac{|f|^2}{dV_{M}}e^{-\varphi-\psi_2}dV_{M}[\psi_1].$$
In fact, $F_t=F$ on $\{\psi<-t\}$.

If $\mathcal H^2(\tilde{c},t)\subset\mathcal H^2(c,t)$ for some $t\geq T$, where $\tilde{c}$ is a nonnegative measurable function on $(T,+\infty)$, then there exists a unique holomorphic $(n,0)$ form $F_t$ on $\{\psi<-t\}$ such that $F_t|_S=f$ and
$$\int_{\{\psi<-t\}}|F_t|^2e^{-\varphi}\tilde{c}(-\psi)\leq\left(\int_t^{+\infty}\tilde{c}(l)e^{-l}dl\right)\sum_{k=1}^{n}\frac{\pi^k}{k!}\int_{S_{n-k}}\frac{|f|^2}{dV_{M}}e^{-\varphi-\psi_2}dV_{M}[\psi_1].$$
In fact, $F_t=F$ on $\{\psi<-t\}$.
\end{Corollary}

\subsubsection{Necessary conditions of $G(\hat{h}^{-1}(r))$ is linear}\label{sec:ne}

 In this section, we give some necessary conditions of $G(\hat{h}^{-1}(r))$ is linear.

\begin{Theorem}
\label{thm:n1}
	Let $M$ be an $n-$dimensional complex manifold satisfying condition $(a)$. Let $c\in \mathcal{P}_T$, and assume that there exists $t\geq T$ such that $G(t)\in(0,+\infty)$. If there exists a Lebesgue measurable function $\tilde \varphi\geq\varphi$ such that $\tilde\varphi+\psi$ is plurisubharmonic function on $M$ and satisfies that:
	
	$(1)$ $\tilde\varphi\not=\varphi$;
	
	$(2)$ $\lim_{t\rightarrow T+0}\sup_{\{\psi\geq-t\}}(\tilde\varphi-\varphi)=0$;
	
	$(3)$ $\tilde\varphi-\varphi$ is bounded on $M$.
	
	Then $G(\hat{h}^{-1}(r))$ is not linear with respect to $r\in(0,\int_{T}^{+\infty}c(s)e^{-s}ds)$.
	Especially, if $\varphi+\psi$ is strictly plurisubharmonic at $z_1\in M$, $G(\hat{h}^{-1}(r))$ is not linear with respect to $r\in(0,\int_{T}^{+\infty}c(s)e^{-s}ds)$.
\end{Theorem}
  In the following, we give two necessary conditions for $\psi$ when $G(\hat{h}^{-1}(r))$ is  linear.

\begin{Theorem}
	\label{thm:n4}Let $M$ be an $n-$dimensional complex manifold satisfying condition $(a)$. Let $c\in\mathcal{P}_T$, and assume that  $G(T)\in(0,+\infty)$.  If  there exists a plurisubharmonic function $\tilde\psi\geq\psi$ on $M$ satisfying that:
	
	$(1)$ $\tilde\psi<-T$ on $M$;
	
	$(2)$ $\tilde\psi\not=\psi$;
	
	$(3)$ $\lim_{t\rightarrow+\infty}\sup_{\{\psi<-t\}}(\tilde\psi-\psi)=0$.
	
	Then $G(\hat{h}^{-1}(r))$ is not linear with respect to $r\in(0,\int_{T}^{+\infty}c(s)e^{-s}ds)$.
	Especially, if $\psi$ is strictly plurisubharmonic at $z_1\in M\backslash(\cap_{t}\overline{\{\psi<-t\}})$, $G(\hat{h}^{-1}(r))$ is not linear with respect to $r\in(0,\int_{T}^{+\infty}c(s)e^{-s}ds)$.
\end{Theorem}

Let $M$ be an $n-$dimensional Stein manifold, and let $S$ be an analytic subset of $M$.	Let  $\psi$ be a plurisubharmonic function on $M$, and let $\varphi$ be a Lebesgue measurable function on $M$ such that $\varphi+\psi$ is  plurisubharmonic on $M$.

We call $(\varphi,\psi)$ in class $W$ if there exist two plurisubharmonic functions $\psi_1\in A'(S)$ and $\psi_2$, such that $\varphi+\psi_2$ is plurisubharmonic function on $M$ and $\psi=\psi_1+\psi_2$.

 The following theorem gives a necessary condition of $G(\hat{h}^{-1}(r))$ is linear when $(\varphi,\psi)\in W$.
 \begin{Theorem}
 	\label{thm:n2}Let $c\in P_T$, and let $(\varphi,\psi)\in W$. Let $\mathcal{F}|_{Z_0}=\mathcal{I}(\psi_1)|_{S_{reg}}$. Assume that  $G(T)\in(0,+\infty)$ and  $\psi_2(z)>-\infty$ for almost every $z\in S_{reg}$. If $G(\hat{h}^{-1}(r))$ is  linear with respect to $r\in(0,\int_{T}^{+\infty}c(s)e^{-s}ds)$, then we have
 	\begin{equation}
 		\label{eq:2106i}
 		\frac{G(T)}{\int_{T}^{+\infty}c(t)e^{-t}dt}=\sum_{k=1}^{n}\frac{\pi^k}{k!}\int_{S_{n-k}}\frac{|f|^2}{dV_{M}}e^{-\varphi-\psi_2}dV_{M}[\psi_1], 	
 		\end{equation}	
 		and there is no plurisubharmonic function $\tilde\psi\geq\psi$ on $M$ satisfying that:
 	
 	$(1)$ $\tilde\psi<-T$;
 	
 	$(2)$ $\tilde\psi\not=\psi$;
 	
 	$(3)$ $(\varphi+\psi-\tilde\psi,\tilde\psi)\in W$.
 	   \end{Theorem}

\subsubsection{Characterizations for the linearity of $G(\hat{h}(r))$ in $1-$dimensional case}\label{sec:1-d}

 In this section, we consider the $1-$dimensional case. Let $M=\Omega$ be an open Riemann surface admitted a nontrivial Green function $G_{\Omega}$, we give characterizations of the linearity.

 We recall some notations (see \cite{guan-zhou13ap}). Let $p:\Delta\rightarrow\Omega$ be the universal covering from unit disc $\Delta$ to $\Omega$. we call the holomorphic function $f$ (resp. holomorphic $(1,0)$ form $F$) on $\Delta$ a multiplicative function (resp. multiplicative differential (Prym differential)), if there is a character $\chi$, which is the representation of the fundamental group of $\Omega$, such that $g^{*}f=\chi(g)f$ (resp. $g^{*}(F)=\chi(g)F$), where $|\chi|=1$ and $g$ is an element of the fundamental group of $\Omega$. Denote the set of such kinds of $f$ (resp. $F$) by $\mathcal{O}^{\chi}(\Omega)$ (resp. $\Gamma^{\chi}(\Omega)$).

 For any harmonic function $u$ on $\Omega$, there exists a $\chi_{u}$ and a multiplicative function $f_u\in\mathcal{O}^{\chi_u}(\Omega)$, such that $|f_u|=p^{*}\left(e^{u}\right)$. If $u_1-u_2=\log|f|$, where $u_1$ and $u_2$ are harmonic function on $\Omega$ and $f$ is holomorphic function on $\Omega$, then $\chi_{u_1}=\chi_{u_2}$.

 For the Green function $G_{\Omega}(z,z_0)$, one can also find a $\chi_{z_0}$ and a multiplicative function $f_{z_0}\in\mathcal{O}^{\chi_{z_0}}(\Omega)$, such that $|f_{z_0}|=p^{*}\left(e^{G_{\Omega}(z,z_0)}\right)$.

 Let $M=\Omega$ be an open Riemann surface admitted a nontrivial Green function $G_{\Omega}$.  Let $\psi$ be a subharmonic function on $\Omega$ satisfying $T=-\sup_{\Omega}\psi=0$, and let $\varphi$ be a Lebesgue measurable function on $\Omega$, such that $\varphi+\psi$ is subharmonic on $\Omega$. Let $Z_0=z_0$ be a point in $\Omega$.

 Let $w$ be a local coordinate on a neighborhood $V_{z_0}$ of $z_0\in\Omega$ satisfying $w(z_0)=0$. Set $f=f_1(w)dw$ on $V_{z_0}$, where $f$ is the holomorphic $(1,0)$ form in the definition of $G(t)$ (see Section \ref{sec:main}) and $f_1$ is a holomorphic function on $V_{z_0}$.

   The following two theorems give characterizations of $G(\hat{h}^{-1}(r))$ is linear with respect to $r\in(0,\int_{0}^{+\infty}c(l)e^{-l}dl)$ for some kinds of $(\varphi,\psi)$. Set $d^c=\frac{1}{2\pi i}(\partial-\bar\partial)$.

\begin{Theorem}
	\label{thm:e2}
	Let $c\in\mathcal{P}_0$.   Assume that $\varphi+a\psi$ is a subharmonic function on a neighborhood of $z_0$ for some $a\in[0,1)$, and $G(0)\in(0,+\infty)$. Then $G(\hat{h}^{-1}(r))$ is linear with respect to $r$ if and only if the following statements hold:
	
	$(1)$ $\varphi+\psi=2\log|g|+2G_{\Omega}(z,z_0)+2u$, $ord_{z_0}(g)=ord_{z_0}(f_1)$ and $\mathcal{F}_{z_0}=\mathcal{I}(\varphi+\psi)_{z_0}$, where $g$ is a holomorphic function on $\Omega$ and $u$ is a harmonic function on $\Omega$;
	
	$(2)$ $\psi=2pG_{\Omega}(z,z_0)$ on $\Omega$ for some $p>0$;
	
	$(3)$ $\chi_{-u}=\chi_{z_0}$, where $\chi_{-u}$ and $\chi_{z_0}$ are the  characters associated to the functions $-u$ and $G_{\Omega}(z,z_0)$ respectively.	
\end{Theorem}
When $\psi=2G_{\Omega}(z,z_0)$, $\mathcal{F}_{z_0}=\mathcal{I}(\varphi+\psi)_{z_0}$ and $c(t)\in\mathcal{P}'_0$, Theorems \ref{thm:e2} can be referred to \cite{GM-concave}.

\begin{Theorem}
	\label{thm:e1}
		Let $c\in\mathcal{P}_0$, and let $Z_0=z_0$ be a point in $\Omega$. Assume that $(\psi-2pG_{\Omega}(z,z_0))(z_0)>-\infty$, where $p=\frac{1}{2}v(dd^{c}\psi,z_0)$, and  $G(0)\in(0,+\infty)$. Then $G(\hat{h}^{-1}(r))$ is linear with respect to $r$ if and only if the following statements hold:

	$(1)$ $\varphi+\psi=2\log|g|+2G_{\Omega}(z,z_0)+2u$, $ord_{z_0}(g)=ord_{z_0}(f_1)$ and $\mathcal{F}_{z_0}=\mathcal{I}(\varphi+\psi)_{z_0}$, where $g$ is a holomorphic function on $\Omega$ and $u$ is a harmonic function on $\Omega$;
	
	$(2)$ $p>0$ and $\psi=2pG_{\Omega}(z,z_0)$ on $\Omega$;
	
	$(3)$ $\chi_{-u}=\chi_{z_0}$, where $\chi_{-u}$ and $\chi_{z_0}$ are the  characters associated to the functions $-u$ and $G_{\Omega}(z,z_0)$ respectively.
\end{Theorem}

\subsubsection{Characterizations for the holding of the equality in  optimal $L^2$ extension problem on open Riemann surfaces with weights may not be subharmonic.}\label{sec:20210808}

 Let $M=\Omega$ be an open Riemann surface admitted a nontrivial Green function $G_{\Omega}$.  Let $\psi$ be a subharmonic function on $\Omega$ satisfying $T=-\sup_{\Omega}\psi=0$, and let $\varphi$ be a Lebesgue measurable function on $\Omega$, such that $\varphi+\psi$ is subharmonic on $\Omega$. Let $Z_0=z_0$ be a point in $\Omega$.

 Let $w$ be a local coordinate on a neighborhood $V_{z_0}$ of $z_0\in\Omega$ satisfying $w(z_0)=0$. Let  $f\equiv dw$ be a holomorphic $(1,0)$ form  on $V_{z_0}$.
Following the notations in Section \ref{sec:L2}.   Now, we give characterizations for the holding of the equality in  optimal $L^2$ extension problem on open Riemann surfaces with weights may not be subharmonic.

\begin{Corollary}
	\label{c:ll3}Let $M=\Omega$, $S=z_0$, and $T=0$. Let $\varphi(z_0)>-\infty$. Assume that $\psi\in A(z_0)$, $e^{-\varphi-\psi}$ is not $L^1$ on any neighborhood of $z_0$ and  $c(t)\in\mathcal{P}_0$ satisfying $\int_{0}^{+\infty}c(t)e^{-t}dt<+\infty$.

Then there exists a holomorphic $(1,0)$ form $F$ on $\Omega$ such that $F(z_0)=f(z_0)$ and
\begin{equation}
	\label{eq:210809}
	\int_{\Omega}|F|^2e^{-\varphi}c(-\psi)\leq\left(\int_{0}^{+\infty}c(t)e^{-t}dt\right)\|f\|_{z_0}.\end{equation}
	Moreover,
	  equality $\left(\int_{0}^{+\infty}c(t)e^{-t}dt\right)\|f\|_{z_0}=\inf\{\|\tilde F\|_{\Omega}$: $\tilde F$ is a holomorphic extension of $f$ from $z_0$ to $\Omega$$\}$ holds if and only if
	the following statements hold:
	
	$(1)$ $\varphi=2\log|g|+2u$, where $u$ is a harmonic function on $\Omega$ and $g$ is a holomorphic function on $\Omega$ such that $g(z_0)\not=0$;
	
	$(2)$ $\psi=2G_{\Omega}(z,z_0)$ on $\Omega$;
	
	$(3)$ $\chi_{-u}=\chi_{z_0}$, where $\chi_{-u}$ and $\chi_{z_0}$ are the  characters associated to the functions $-u$ and $G_{\Omega}(z,z_0)$ respectively.
\end{Corollary}
When $\psi=2G_{\Omega}(z,z_0)$ and $c(t)\in\mathcal{P}'_0$, Corollary \ref{c:ll3} can be referred to \cite{GM-concave}.

\begin{Corollary}\label{c:ll4}
Let $M=\Omega$, $S=z_0$, and $T=0$. Let $(\varphi,\psi)\in W$ and let $\|f\|_{z_0}^{*}\in(0,+\infty)$.  Let $c(t)\in\mathcal{P}_0$  such that $\int_{0}^{+\infty}c(t)e^{-t}dt<+\infty$.
 Then equality $\left(\int_{0}^{+\infty}c(t)e^{-t}dt\right)\|f\|_{z_0}^{*}=\inf\{\|F\|_{\Omega}:F$ is a holomorphic extension of $f$ from $z_0$ to $\Omega$$\}$ holds if and only if
	the following statements hold:
	
	$(1)$ $\varphi=2\log|g|+2u$, where $u$ is a harmonic function on $\Omega$ and $g$ is a holomorphic function on $\Omega$ such that $g(z_0)\not=0$;
	
	$(2)$ $\psi=2G_{\Omega}(z,z_0)$ on $\Omega$;
	
	$(3)$ $\chi_{-u}=\chi_{z_0}$, where $\chi_{-u}$ and $\chi_{z_0}$ are the  characters associated to the functions $-u$ and $G_{\Omega}(z,z_0)$ respectively.
\end{Corollary}

\section{Preparation}

\subsection{$L^{2}$ methods}

We call a positive measurable function $c$ on $(S,+\infty)$ in class $\tilde{\mathcal{P}}_{S}$ if $\int_{S}^{s}c(l)e^{-l}dl<+\infty$ for some $s>S$ and $c(t)e^{-t}$ is decreasing with respect to $t$. Note that $\mathcal{P}_T\subset\tilde{\mathcal{P}}_S$ when $S>T$.

In this section,
we present the following Lemma (proof can be referred to Section \ref{sec:p-GZsharp}),
whose various forms already appear in \cite{guan-zhou13p,guan-zhou13ap,guan_sharp,GM-concave} etc.:

\begin{Lemma} \label{lem:GZ_sharp}Let $B\in(0,+\infty)$ and $t_{0}\geq S$ be arbitrarily given.
Let $M$ be an $n-$dimensional Stein manifold.
Let $\psi<-S$ be a plurisubharmonic function
on $M$.
Let $\varphi$ be a plurisubharmonic function on $M$.
Let $F$ be a holomorphic $(n,0)$ form on $\{\psi<-t_{0}\}$,
such that
\begin{equation}
\label{equ:20171124a}
\int_{K\cap\{\psi<-t_{0}\}}|F|^{2}<+\infty
\end{equation}
for any compact subset $K$ of $M$,
and
\begin{equation}
\label{equ:20171122a}
\int_{M}\frac{1}{B}\mathbb{I}_{\{-t_{0}-B<\psi<-t_{0}\}}|F|^{2}e^{-\varphi}\leq C<+\infty.
\end{equation}
Then there exists a
holomorphic $(n,0)$ form $\tilde{F}$ on $M$, such that
\begin{equation}
\label{equ:3.4}
\begin{split}
\int_{M}&|\tilde{F}-(1-b_{t_0,B}(\psi))F|^{2}e^{-\varphi+v_{t_0,B}(\psi)}c(-v_{t_0,B}(\psi))\leq C\int_{S}^{t_{0}+B}c(t)e^{-t}dt
\end{split}
\end{equation}
where
$b_{t_0,B}(t)=\int_{-\infty}^{t}\frac{1}{B}\mathbb{I}_{\{-t_{0}-B< s<-t_{0}\}}ds$,
$v_{t_0,B}(t)=\int_{-t_0}^{t}b_{t_0,B}(s)ds-t_0$,
and $c(t)\in \tilde{\mathcal{P}}_{S}$.
\end{Lemma}

We give the proof of Lemma \ref{lem:GZ_sharp} in  Section \ref{sec:p-GZsharp}.
It is clear that $\mathbb{I}_{(-t_{0},+\infty)}(t)\leq b_{t_0,B}(t)\leq\mathbb{I}_{(-t_{0}-B,+\infty)}(t)$ and $\max\{t,-t_{0}-B\}\leq v_{t_0,B}(t) \leq\max\{t,-t_{0}\}$.

 Lemma \ref{lem:GZ_sharp} implies the following lemma, which will be used in the proof of Theorem \ref{thm:general_concave}.

\begin{Lemma} \label{lem:GZ_sharp2}Let $M$ be an $n-$dimensional complex manifold satisfying condition $(a)$, and let $c(t)\in\mathcal{P}_T$.
Let $B\in(0,+\infty)$ and $t_{0}>t_1>T$ be arbitrarily given.
Let $\psi<-T$ be a plurisubharmonic function
on $M$.
Let $\varphi$ be a Lebesgue measurable function on $M$, such that $\varphi+\psi$ is plurisubharmonic on $M$.
Let $F$ be a holomorphic $(n,0)$ form on $\{\psi<-t_{0}\}$,
such that
\begin{equation}
\label{eq:210627b}
\int_{\{\psi<-t_{0}\}}|F|^{2}e^{-\varphi}c(-\psi)<+\infty.
\end{equation}

Then there exists a
holomorphic $(n,0)$ form $\tilde{F}$ on $\{\psi<-t_1\}$, such that
\begin{equation}
\label{eq:210627d}
\begin{split}
\int_{\{\psi<-t_1\}}&|\tilde{F}-(1-b_{t_0,B}(\psi))F|^{2}e^{-\varphi-\psi+v_{t_0,B}(\psi)}c(-v_{t_0,B}(\psi))\leq C\int_{t_1}^{t_{0}+B}c(t)e^{-t}dt
\end{split}
\end{equation}
where
$C=\int_{M}\frac{1}{B}\mathbb{I}_{\{-t_{0}-B<\psi<-t_{0}\}}|F|^{2}e^{-\varphi-\psi}<+\infty$, $b_{t_0,B}(t)=\int_{-\infty}^{t}\frac{1}{B}\mathbb{I}_{\{-t_{0}-B< s<-t_{0}\}}ds$ and
$v_{t_0,B}(t)=\int_{0}^{t}b_{t_0,B}(s)ds$.
\end{Lemma}
\begin{proof}
	As $M$ is an $n-$dimensional complex manifold satisfying condition $(a)$ and  $c(t)\in\mathcal{P}_T$, there exist  a closed subset $X\subset M$ and a closed subset $E\subset X\cap\{\psi=-\infty\}$ satisfying that $X$ is locally negligible with respect to $L^2$ holomorphic functions, $M\backslash X$ is a Stein manifold, $e^{-\varphi}c(-\psi)$ has locally a positive lower bound on $M\backslash E$ and there exists an analytic subset $Z$ of $M$ such that $E\subset Z$.
	
	Combining inequality \eqref{eq:210627b} and $e^{-\varphi}c(-\psi)$ has locally a positive lower bound on $M\backslash E$, we obtain that
	$$\int_{K\cap\{\psi<-t_{0}\}}|F|^{2}<+\infty$$	
holds for any compact subset $K$ of $M\backslash X$. Then Lemma \ref{lem:GZ_sharp} shows that there exists a
holomorphic $(n,0)$ form $\tilde{F}_X$ on $\{\psi<-t_1\}\backslash X$, such that
\begin{equation}
\label{eq:210627f}
\begin{split}
\int_{\{\psi<-t_1\}\backslash X}|\tilde{F}_X-(1-b_{t_0,B}(\psi))F|^{2}e^{-\varphi-\psi+v_{t_0,B}(\psi)}c(-v_{t_0,B}(\psi))\leq C\int_{t_1}^{t_{0}+B}c(t)e^{-t}dt.
\end{split}
\end{equation}	

For any $z\in \{\psi<-t_1\}\cap(X\backslash E)$, there exists an open neighborhood $V_z$ of $z$, such that $V_z\subset\subset \{\psi<-t_1\}\backslash E$. Note that $c(t)e^{-t}$ is decreasing on $(T,+\infty)$ and $v_{t_0,B}(\psi)\geq\psi$, then we have
\begin{equation}\begin{split}
	&\int_{V_z\backslash X}|\tilde{F}_X-(1-b_{t_0,B}(\psi))F|^2e^{-\varphi}c(-\psi)\\
	\leq&\int_{V_z\backslash X}|\tilde{F}_X-(1-b_{t_0,B}(\psi))F|^{2}e^{-\varphi-\psi+v_{t_0,B}(\psi)}c(-v_{t_0,B}(\psi))\\
	<&+\infty
\end{split}\end{equation}
Note that there exists a positive number $C_1>0$ such that $e^{-\varphi}c(-\psi)>C_1$ on $V_{z}$ and $\int_{V_z\backslash X}|(1-b_{t_0,B}(\psi))F|^2e^{-\varphi}c(-\psi)\leq\int_{\{\psi<-t_{0}\}}|F|^{2}e^{-\varphi}c(-\psi)<+\infty$, then we have
\begin{equation}
	\label{eq:210627g}\begin{split}
		&\int_{V_z\backslash X}|\tilde{F}_X|^2\\
	\leq &C_1\int_{V_z\backslash X}|\tilde{F}_X|^2e^{-\varphi}c(-\psi)	\\
	\leq &2C_1\left(\int_{V_z\backslash X}|(1-b_{t_0,B}(\psi))F|^2e^{-\varphi}c(-\psi)+\int_{V_z\backslash X}|\tilde{F}_X-(1-b_{t_0,B}(\psi))F|^2e^{-\varphi}c(-\psi)\right)\\
	<&+\infty.\end{split}
\end{equation}
As $X$ is locally negligible with respect to $L^2$ holomorphic functions, we can find a holomorphic extension $\tilde{F}_E$ of $\tilde{F}_X$  from $\{\psi<-t_1\}\backslash X$ to $\{\psi<-t_1\}\backslash E$ such that
\begin{equation}
	\label{eq:210627h}
	\int_{\{\psi<-t_1\}\backslash E}|\tilde{F}_E-(1-b_{t_0,B}(\psi))F|^{2}e^{-\varphi-\psi+v_{t_0,B}(\psi)}c(-v_{t_0,B}(\psi))\leq C\int_{t_1}^{t_{0}+B}c(t)e^{-t}dt.\end{equation}
	
Note that $E\subset \{\psi<-t_0\}\subset\{\psi<-t_1\}$, for any $z\in E$, there exists an open neighborhood $U_z$ of $z$, such that $U_z\subset\subset\{\psi<-t_0\}$. As $\varphi+\psi$ is plurisubharmonic on $M$ and $e^{v_{t_0,B}(\psi)}c(-v_{t_0,B}(\psi))$ has a positive lower bound on $\{\psi<-t_1\}$, then we have
\begin{equation}
	\label{eq:210627i}
	\begin{split}
		&\int_{U_z\backslash E}|\tilde{F}_E-(1-b_{t_0,B}(\psi))F|^{2}\\
		\leq&C_2\int_{\{\psi<-t_1\}\backslash E}|\tilde{F}_E-(1-b_{t_0,B}(\psi))F|^{2}e^{-\varphi-\psi+v_{t_0,B}(\psi)}c(-v_{t_0,B}(\psi))\\
		<&+\infty,
	\end{split}
\end{equation}
where $C_2$ is some positive number. As $U_z\subset\subset\{\psi<-t_0\}$, we have
\begin{equation}
	\label{eq:210627j}
	\int_{U_z}|(1-b_{t_0,B}(\psi))F|^2\le\int_{U_z}|F|^2<+\infty.
\end{equation}
Combining inequality \eqref{eq:210627i} and \eqref{eq:210627j}, we obtain that $\int_{U_z\backslash E}|\tilde{F}_E|^2<+\infty$.

 As $E$ is contained in some analytic subset of $M$, we can find a holomorphic extension  $\tilde{F}$ of $\tilde{F}_E$  from $\{\psi<-t_1\}\backslash E$ to $\{\psi<-t_1\}$ such that
\begin{equation}
	\label{eq:210627h}
	\int_{\{\psi<-t_1\}}|\tilde{F}-(1-b_{t_0,B}(\psi))F|^{2}e^{-\varphi-\psi+v_{t_0,B}(\psi)}c(-v_{t_0,B}(\psi))\leq C\int_{t_1}^{t_{0}+B}c(t)e^{-t}dt.\end{equation}
This proves Lemma \ref{lem:GZ_sharp2}.
\end{proof}

\subsection{Some properties of $G(t)$}\label{sec:minimal}

\

 We present some properties related to $G(t)$
in the this section.

\begin{Lemma}(see \cite{G-R})
\label{closedness}
Let $N$ be a submodule of $\mathcal O_{\mathbb C^n,o}^q$, $1\leq q<+\infty$, let $f_j\in\mathcal O_{\mathbb C^n}(U)^q$ be a sequence of $q-$tuples holomorphic in an open neighborhood $U$ of the origin $o$. Assume that the $f_j$ converge uniformly in $U$ towards  a $q-$tuples $f\in\mathcal O_{\mathbb C^n}(U)^q$, assume furthermore that all germs $(f_{j},o)$ belong to $N$. Then $(f,o)\in N$.	
\end{Lemma}

The closedness of submodules will be used in the following discussion.

\begin{Lemma}
	\label{l:converge}
	Let $M$ be a complex manifold. Let $S$ be an analytic subset of $M$.  	
	Let $\{g_j\}_{j=1,2,...}$ be a sequence of nonnegative Lebesgue measurable functions on $M$, which satisfies that $g_j$ are almost everywhere convergent to $g$ on  $M$ when $j\rightarrow+\infty$,  where $g$ is a nonnegative Lebesgue measurable function on $M$. Assume that for any compact subset $K$ of $M\backslash S$, there exist $s_K\in(0,+\infty)$ and $C_K\in(0,+\infty)$ such that
	$$\int_{K}{g_j}^{-s_K}dV_M\leq C_K$$
	 for any $j$, where $dV_M$ is a continuous volume form on $M$.
	
 Let $\{F_j\}_{j=1,2,...}$ be a sequence of holomorphic $(n,0)$ form on $M$. Assume that $\liminf_{j\rightarrow+\infty}\int_{M}|F_j|^2g_j\leq C$, where $C$ is a positive constant. Then there exists a subsequence $\{F_{j_l}\}_{l=1,2,...}$, which satisfies that $\{F_{j_l}\}$ is uniformly convergent to a holomorphic $(n,0)$ form $F$ on $M$ on any compact subset of $M$ when $l\rightarrow+\infty$, such that
 $$\int_{M}|F|^2g\leq C.$$
\end{Lemma}
\begin{proof}
	As $S$ is a analytic subset of $M$, by Local Parametrization Theorem (see \cite{demailly-book}) and Maximum Principle, for any compact set $K\subset\subset M$ there exists  $ K_1\subset\subset M\backslash S$ satisfying 	
	\begin{equation}
		\label{eq:210809a}\sup_{z\in K}\frac{|F_j(z)|^2}{dV_M}\leq C_1\sup_{z\in K_1}\frac{|F_j(z)|^2}{dV_M}	\end{equation}
	for  any $j$, where $C_1$ is a constant.
Then there exists a compact set $K_2\subset\subset M\backslash S$ satisfying $K_1\subset K_2$ and
\begin{equation}
	\label{eq:210809b}\begin{split}
	\left(\frac{|F_j(z)|^{2}}{dV_M}\right)^r&\leq C_2\int_{K_2}\left(\frac{|F_j(z)|^{2}}{dV_M}\right)^r\\
		&\leq C_2 \left(\int_{K_2}|F_j|^{2}g_j\right)^{r}\left(\int_{K_2}g_j^{-\frac{r}{1-r}}\right)^{1-r}		
	\end{split}\end{equation}
for any $j$ and any $z\in K_1$, where $r\in(0,1)$ and $C_2$ is a constant. Let $\frac{r}{1-r}=s_{K_2}$, inequality \eqref{eq:210809b} implies
\begin{equation}
	\label{eq:210809c}\sup_{z\in K_1}\frac{|F_j(z)|^2}{dV_M}\leq C_3\int_{K_2}|F_j|^{2}g_j,
\end{equation}
where $C_3$ is a constant. As $\liminf_{j\rightarrow+\infty}\int_{M}|F_j|^2g_j<C$, combining inequality \eqref{eq:210809a}, \eqref{eq:210809c} and the diagonal method,
 we obtain a subsequence of $\{F_j\}$, denoted still by $\{F_j\}$, which is uniformly convergent to a holomorphic $(n,0)$ form $F$ on $M$ on any compact subset of $M$.
	
	It follows from the Fatou's Lemma and $\lim_{j\rightarrow+\infty}\int_{M}|F_j|^2g_j\leq C$ that
	\begin{displaymath}
		\begin{split}
			\int_{M}|F|^2g=&\int_{M}\lim_{j\rightarrow+\infty}|F_j|^2g_j\\
			\leq&\liminf_{j\rightarrow+\infty}\int_{M}|F_j|^2g_j\\
			\leq&C.
		\end{split}
	\end{displaymath}
Thus Lemma \ref{l:converge} holds.	
\end{proof}

Let $M$ be an $n-$dimensional complex manifold satisfying condition $(a)$.
Let $\psi$ be a plurisubharmonic function on $M$,
and let  $\varphi$ be a Lebesgue measurable function on $M$,
such that $\varphi+\psi$ is a plurisubharmonic function on $M$. Let $c\in\mathcal{P}_T$.
The following lemma is a characterization of $G(t)=0$ for any $t\geq T$, where $T=-\sup_M{\psi}$ and the meaning of  $G(t)$ can be referred to Section \ref{sec:main}.

\begin{Lemma}
\label{lem:0}
$f\in H^0(Z_0,(\mathcal{O}(K_{M})\otimes\mathcal{F})|_{Z_0})\Leftrightarrow G(t)=0$.
\end{Lemma}

\begin{proof}
It is clear that $f\in H^0(Z_0,(\mathcal{O}(K_{M})\otimes\mathcal{F})|_{Z_0})\Rightarrow G(t)=0$.

In the following part, we prove that $G(t)=0\Rightarrow f\in H^0(Z_0,(\mathcal{O}(K_{M})\otimes\mathcal{F})|_{Z_0})$.
As $G(t)=0$, then there exists holomorphic $(n,0)$ forms $\{\tilde{f}_{j}\}_{j\in\mathbb{N}^{+}}$ on $\{\psi<-t\}$
such that $\lim_{j\to+\infty}\int_{\{\psi<-t\}}|\tilde{f}_{j}|^{2}e^{-\varphi}c(-\psi)=0$ and
$(f_{j}-f)\in H^{0}(Z_0,(\mathcal{O}(K_{M})\otimes\mathcal{F})|_{Z_0})$ for any $j$.
As $e^{-\varphi}c(-\psi)$ has a positive lower bound on any compact subset of $M\backslash Z$, where $Z$ is some analytic subset of $M$, it follows from Lemma \ref{l:converge} that there exists a subsequence of $\{\tilde{f}_{j}\}_{j\in\mathbb{N}^{+}}$
denoted by $\{\tilde{f}_{j_{k}}\}_{k\in\mathbb{N}^{+}}$ that compactly convergent to $0$.
It is clear that $\tilde{f}_{j_{k}}-f$ is compactly convergent to $0-f=f$ on $U\cap\{\psi<-t\}$.
It follows from Lemma \ref{closedness}
that $f\in H^{0}(Z_0,(\mathcal{O}(K_{M})\otimes\mathcal{F})|_{Z_0})$.
This proves Lemma \ref{lem:0}.
\end{proof}

The following lemma shows the existence and uniqueness of the holomorphic $(n,0)$ form related to $G(t)$.

\begin{Lemma}
\label{lem:A}
Assume that $G(t)<+\infty$ for some $t\in[T,+\infty)$.
Then there exists a unique holomorphic $(n,0)$ form $F_{t}$ on
$\{\psi<-t\}$ satisfying $(F_{t}-f)\in H^{0}(Z_0,(\mathcal{O}(K_{M})\otimes\mathcal{F})|_{Z_0})$ and $\int_{\{\psi<-t\}}|F_{t}|^{2}e^{-\varphi}c(-\psi)=G(t)$.
Furthermore,
for any holomorphic $(n,0)$ form $\hat{F}$ on $\{\psi<-t\}$ satisfying $(\hat{F}-f)\in H^{0}(Z_0,(\mathcal{O}(K_{M})\otimes\mathcal{F})|_{Z_0})$ and
$\int_{\{\psi<-t\}}|\hat{F}|^{2}e^{-\varphi}c(-\psi)<+\infty$,
we have the following equality
\begin{equation}
\label{equ:20170913e}
\begin{split}
&\int_{\{\psi<-t\}}|F_{t}|^{2}e^{-\varphi}c(-\psi)+\int_{\{\psi<-t\}}|\hat{F}-F_{t}|^{2}e^{-\varphi}c(-\psi)
\\=&
\int_{\{\psi<-t\}}|\hat{F}|^{2}e^{-\varphi}c(-\psi).
\end{split}
\end{equation}
\end{Lemma}

\begin{proof}
Firstly, we prove the existence of $F_{t}$.
As $G(t)<+\infty$
then there exists holomorphic (n,0) forms $\{f_{j}\}_{j\in\mathbb{N}^{+}}$ on $\{\psi<-t\}$ such that
$\lim_{j\rightarrow+\infty}\int_{\{\psi<-t\}}|f_{j}|^{2}e^{-\varphi}c(-\psi)=G(t)$,
and $(f_{j}-f)\in H^{0}(Z_0,(\mathcal{O}(K_{M})\otimes\mathcal{F})|_{Z_0})$.
As $e^{-\varphi}c(-\psi)$ has a positive lower bound on any compact subset of $M\backslash Z$, where $Z$ is some analytic subset of $M$, it follows from Lemma \ref{l:converge} that there exists a subsequence of $\{f_{j}\}$ compact convergence to a holomorphic $(n,0)$ form $F$ on $\{\psi<-t\}$
satisfying  $ \int_{\{\psi<-t\}}|F|^{2}e^{-\varphi}c(-\psi)\leq G(t)$.
It follows from
Lemma \ref{closedness}
that $(F-f)\in H^{0}(Z_0,(\mathcal{O}(K_{M})\otimes\mathcal{F})|_{Z_0})$.
Then we obtain the existence of $F_{t}(=F)$.

Secondly, we prove the uniqueness of $F_{t}$ by contradiction:
if not, there exist two different holomorphic (n,0) forms $f_{1}$ and $f_{2}$ on on $\{\psi<-t\}$
satisfying $\int_{\{\psi<-t\}}|f_{1}|^{2}e^{-\varphi}c(-\psi)=\int_{\{\psi<-t\}}|f_{2}|^{2}=G(t)$,
$(f_{1}-f)\in H^{0}(Z_0,(\mathcal{O}(K_{M})\otimes\mathcal{F})|_{Z_0})$ and $(f_{2}-f)\in H^{0}(Z_0,(\mathcal{O}(K_{M})\otimes\mathcal{F})|_{Z_0})$.
Note that
\begin{equation}
\begin{split}
&\int_{\{\psi<-t\}}\left|\frac{f_{1}+f_{2}}{2}\right|^{2}e^{-\varphi}c(-\psi)+\int_{\{\psi<-t\}}\left|\frac{f_{1}-f_{2}}{2}\right|^{2}e^{-\varphi}c(-\psi)
\\=&
\frac{\int_{\{\psi<-t\}}|f_{1}|^{2}e^{-\varphi}c(-\psi)+\int_{\{\psi<-t\}}|f_{2}|^{2}e^{-\varphi}c(-\psi)}{2}=G(t),
\end{split}
\end{equation}
then we obtain that
$$\int_{\{\psi<-t\}}\left|\frac{f_{1}+f_{2}}{2}\right|^{2}e^{-\varphi}c(-\psi)<G(t),$$
and $(\frac{f_{1}+f_{2}}{2}-f)\in H^{0}(Z_0,
(\mathcal{O}(K_{M})\otimes\mathcal{F})|_{Z_0})$, which contradicts the definition of $G(t)$.

Finally, we prove equality \eqref{equ:20170913e}.
For any holomorphic $h$ on $\{\psi<-t\}$ satisfying $\int_{\{\psi<-t\}}|h|^{2}e^{-\varphi}c(-\psi)<+\infty$
and $h\in H^{0}(Z_0,
(\mathcal{O}(K_{M})\otimes\mathcal{F})|_{Z_0})$,
it is clear that
for any complex number $\alpha$,
$F_{t}+\alpha h$ satisfying $((F_{t}+\alpha h)-f)\in H^{0}(Z_0,
(\mathcal{O}(K_{M})\otimes\mathcal{F})|_{Z_0})$,
and $\int_{\{\psi<-t\}}|F_{t}|^{2}e^{-\varphi}c(-\psi)\leq\int_{\{\psi<-t\}}|F_{t}+\alpha h|^{2}e^{-\varphi}c(-\psi)<+\infty$.
Note that
$$\int_{\{\psi<-t\}}|F_{t}+\alpha h|^{2}e^{-\varphi}c(-\psi)-\int_{\{\psi<-t\}}|F_{t}|^{2}e^{-\varphi}c(-\psi)\geq 0$$
implies
$$\Re\int_{\{\psi<-t\}}F_{t}\bar{h}e^{-\varphi}c(-\psi)=0$$
then
$$\int_{\{\psi<-t\}}|F_{t}+h|^{2}e^{-\varphi}c(-\psi)=\int_{\{\psi<-t\}}(|F_{t}|^{2}+|h|^{2})e^{-\varphi}c(-\psi).$$
Choosing $h=\hat{F}-F_{t}$, we obtain equality \eqref{equ:20170913e}.
\end{proof}

The following Lemma shows the monotonicity and lower semicontinuity property of $G(t)$.

\begin{Lemma}
\label{lem:B}
$G(t)$ is decreasing with respect to $t\in[T,+\infty)$,
such that
$\lim_{t\to t_{0}+0}G(t)=G(t_{0})$ for any $t_0\in[T,+\infty)$,
and if $G(t)<+\infty$ for some $t\geq T$, then $\lim_{t\to +\infty}G(t)=0$.
Especially $G(t)$ is lower semicontinuous on $[T,+\infty)$.
\end{Lemma}

\begin{proof}
By the definition of $G(t)$,
it is clear that $G(t)$ is decreasing on $[T,+\infty)$. And it follows from the dominated convergence theorem that if $G(t)<+\infty$ for some $t\geq T$, then $\lim_{t\to +\infty}G(t)=0$.
Then it suffices to prove $\lim_{t\to t_{0}+0}G(t)=G(t_{0}).$
We prove it by contradiction:
if not,
then $\lim_{t\to t_{0}+0}G(t)<G(t_{0})$.

By Lemma \ref{lem:A}, there exists a unique holomorphic (n,0) form $F_{t}$ on
$\{\psi<-t\}$ satisfying $(F_{t}-f)\in H^{0}(Z_0,
(\mathcal{O}(K_{M})\otimes\mathcal{F})|_{Z_0})$ and
$\int_{\{\psi<-t\}}|F_{t}|^{2}e^{-\varphi}c(-\psi)=G(t)$.
Note that $G(t)$ is decreasing implies that
$\int_{\{\psi<-t\}}|F_{t}|^{2}e^{-\varphi}c(-\psi)\leq\lim_{t\to t_{0}+0}G(t)$ for any $t>t_{0}$. If $\lim_{t\to t_{0}+0}G(t)=+\infty$, the equality $\lim_{t\to t_{0}+0}G(t)=G(t_{0})$ is clear, thus it suffices to prove the case $\lim_{t\to t_{0}+0}G(t)<+\infty$.
As $e^{-\varphi}c(-\psi)$ has a positive lower bound on any compact subset of $M\backslash Z$, where $Z$ is some analytic subset of $M$,
and $\int_{\{\psi<-t_1\}}|F_t|e^{-\varphi}c(-\psi)\leq\lim_{t\to t_{0}+0}G(t)<+\infty$ holds for any $t\in(t_0,t_1]$, where $t_1>t_0$ is a  fixed number,
it follows from Lemma \ref{l:converge} that there exists $\{F_{t_{j}}\}$ $(t_{j}\to t_{0}+0,$ as $j\to+\infty)$
uniformly convergent on any compact subset of $\{\psi<-t_1\}$. Using the diagonal method, we obtain
a subsequence of $\{F_{t}\}$ (also denoted by $\{F_{t_j}\}$), which is convergent on
any compact subset of $\{\psi<-t_{0}\}$.

Let $\hat{F}_{t_{0}}=\lim_{j\to+\infty}F_{t_{j}}$, which is a holomorphic $(n,0)$ form on $\{\psi<-t_{0}\}$.
Then it follows from the decreasing property of $G(t)$ that
$$\int_{K}|\hat{F}_{t_{0}}|^{2}e^{-\varphi}c(-\psi)\leq \lim_{j\to+\infty}\int_{K}|F_{t_{j}}|^{2}e^{-\varphi}c(-\psi)\leq
\lim_{j\to+\infty}G(t_{j})\leq\lim_{t\to t_{0}+0}G(t)$$
for any
compact set $K\subset \{\psi<-t_{0}\}$.
It follows from Levi's Theorem that
$$\int_{\{\psi<-t_0\}}|\hat{F}_{t_{0}}|^{2}e^{-\varphi}c(-\psi)\leq \lim_{t\to t_{0}+0}G(t).$$
It follows from Lemma \ref{closedness} that $\hat{F}_{t_0}\in H^{0}(Z_0,
(\mathcal{O}(K_{M})\otimes\mathcal{F})|_{Z_0})$.
Then we obtain that $G(t_0)\leq\int_{\{\psi<-t_0\}}|\hat{F}_{t_{0}}|^{2}e^{-\varphi}c(-\psi)\leq \lim_{t\to t_{0}+0}G(t)$,
which contradicts $\lim_{t\to t_{0}+0}G(t)<G(t_{0})$.
\end{proof}

We consider the derivatives of $G(t)$ in the following lemma.

\begin{Lemma}
\label{lem:C}
Assume that $G(t_1)<\infty$, where $t_1\in(T,+\infty)$,
then for any $t_0>t_1$ we have
$$\frac{G(t_1)-G(t_0)}{\int_{t_1}^{t_0}c(t)e^{-t}dt}\leq
\liminf_{B\to0+0}\frac{G(t_0)-G(t_0+B)}{\int_{t_0}^{t_0+B}c(t)e^{-t}dt},$$
i.e.
\begin{displaymath}
	\frac{G(t_0)-G(t_1)}{\int_{T_1}^{t_0}c(t)e^{-t}dt-\int_{T_1}^{t_1}c(t)e^{-t}dt}
		\geq\limsup_{B\to0+0}\frac{G(t_0+B)-G(t_0)}{\int_{T_1}^{t_0+B}c(t)e^{-t}dt-\int_{T_1}^{t_0+B}c(t)e^{-t}dt}.
\end{displaymath}
\end{Lemma}

\begin{proof}
It follows from Lemma \ref{lem:B} that $G(t)<+\infty$ for any $t\geq t_1$.
By Lemma \ref{lem:A},
there exists a
holomorphic $(n,0)$ form $F_{t_0}$ on $\{\psi<-t_0\}$, such that
$(F_{t_0}-f)\in H^{0}(Z_0,(\mathcal{O}(K_{M})\otimes\mathcal{F})|_{Z_0})$ and
$\int_{\{\psi<-t_0\}}|F_{t_0}|^{2}e^{-\varphi}c(-\psi)=G(t_0)$.

It suffices to consider that $\liminf_{B\to0+0}\frac{G(t_0)-G(t_0+B)}{\int_{t_0}^{t_0+B}c(t)e^{-t}dt}\in[0.+\infty)$
because of the decreasing property of $G(t)$.
Then there exists $B_{j}\to 0+0$ $(j\to+\infty)$ such that
\begin{equation}
	\label{eq:211106d}\lim_{j\to+\infty}\frac{G(t_0)-G(t_0+B_j)}{\int_{t_0}^{t_0+B_j}c(t)e^{-t}dt}=\liminf_{B\to0+0}\frac{G(t_0)-G(t_0+B)}{\int_{t_0}^{t_0+B}c(t)e^{-t}dt}
\end{equation}
and $\{\frac{G(t_0)-G(t_0+B_j)}{\int_{t_0}^{t_0+B_j}c(t)e^{-t}dt}\}_{j\in\mathbb{N}^{+}}$ is bounded. As $c(t)e^{-t}$ is decreasing  and positive on $(T,+\infty)$, then
\begin{equation}
\label{eq:202141a}
\begin{split}
\lim_{j\to+\infty}\frac{G(t_0)-G(t_0+B_j)}{\int_{t_0}^{t_0+B_j}c(t)e^{-t}dt}=&\left(\lim_{j\to+\infty}\frac{G(t_0)-G(t_0+B_j)}{B_j}\right)\left(\frac{1}{\lim_{t\rightarrow t_0+0}c(t)e^{-t}}\right)\\
=&\left(\lim_{j\to+\infty}\frac{G(t_{0})-G(t_{0}+B_j)}{B_j}\right)\left(\frac{e^{t_0}}{\lim_{t\rightarrow t_0+0}c(t)}\right).	
\end{split}	
\end{equation}
Hence $\left\{\frac{G(t_0)-G(t_0+B_j)}{B_j}\right\}_{j\in\mathbb{N}^{+}}$ is bounded with respect to $j$.

As $t\leq v_{t_0,B_j}(t)$, the decreasing property of $c(t)e^{-t}$ shows that
$$e^{-\psi+v_{t_0,B_j}(\psi)}c(-v_{t_0,B_j}(\psi))\geq c(-\psi).$$
Lemma \ref{lem:GZ_sharp2}
shows that for any $B_{j}$,
there exists
holomorphic $(n,0)$ form $\tilde{F}_{j}$ on $\{\psi<-t_1\}$, such that
$(\tilde{F}_{j}-F_{t_{0}})\in H^{0}(Z_0,(\mathcal{O}(K_M)\otimes\mathcal{I}(\varphi+\psi))|_{Z_0})
\subseteq H^{0}(Z_0,(\mathcal{O}(K_{M})\otimes\mathcal{F})|_{Z_0})$
and
\begin{equation}
\label{equ:GZc}
\begin{split}
&\int_{\{\psi<-t_1\}}|\tilde{F}_{j}-(1-b_{t_{0},B_{j}}(\psi))F_{t_{0}}|^{2}e^{-\varphi}c(-\psi)
\\\leq&\int_{\{\psi<-t_1\}}|\tilde{F}_{j}-(1-b_{t_{0},B_{j}}(\psi))F_{t_{0}}|^{2}e^{-\varphi}e^{-\psi+v_{t_0,B_j}(\psi)}c(-v_{t_0,B_j}(\psi))
\\\leq&
\int_{t_1}^{t_{0}+B_{j}}c(t)e^{-t}dt
\int_{\{\psi<-t_1\}}\frac{1}{B_{j}}(\mathbb{I}_{\{-t_{0}-B_{j}<\psi<-t_{0}\}})|F_{t_{0}}|^{2}e^{-\varphi-\psi}
\\\leq&
\frac{e^{t_{0}+B_{j}}\int_{t_1}^{t_{0}+B_{j}}c(t)e^{-t}dt}{\inf_{t\in(t_{0},t_{0}+B_{j})}c(t)}
\int_{\{\psi<-t_1\}}\frac{1}{B_{j}}(\mathbb{I}_{\{-t_{0}-B_{j}<\psi<-t_{0}\}})|F_{t_{0}}|^{2}e^{-\varphi}c(-\psi)
\\\leq&
\frac{e^{t_{0}+B_{j}}\int_{t_1}^{t_{0}+B_{j}}c(t)e^{-t}dt}{\inf_{t\in(t_{0},t_{0}+B_{j})}c(t)}
\times\Bigg(\int_{\{\psi<-t_1\}}\frac{1}{B_{j}}\mathbb{I}_{\{\psi<-t_{0}\}}|F_{t_{0}}|^{2}e^{-\varphi}c(-\psi)
\\&-\int_{\{\psi<-t_1\}}\frac{1}{B_{j}}\mathbb{I}_{\{\psi<-t_{0}-B_{j}\}}|F_{t_{0}}|^{2}e^{-\varphi}c(-\psi)\Bigg)
\\\leq&
\frac{e^{t_{0}+B_{j}}\int_{t_1}^{t_{0}+B_{j}}c(t)e^{-t}dt}{\inf_{t\in(t_{0},t_{0}+B_{j})}c(t)}
\times\frac{G(t_{0})-G(t_{0}+B_{j})}{B_{j}}.
\end{split}
\end{equation}

Firstly, we will prove that $\int_{\{\psi<-t_1\}}|\tilde{F}_{j}|^{2}e^{-\varphi}c(-\psi)$ is bounded with respect to $j$.

Note that
\begin{equation}
\label{equ:GZd}
\begin{split}
&\left(\int_{\{\psi<-t_1\}}|\tilde{F}_{j}-(1-b_{t_{0},B_{j}}(\psi))F_{t_{0}}|^{2}e^{-\varphi}c(-\psi)\right)^{\frac{1}{2}}
\\\geq&\left(\int_{\{\psi<-t_1\}}|\tilde{F}_{j}|^{2}e^{-\varphi}c(-\psi)\right)^{\frac{1}{2}}-\left(\int_{\{\psi<-t_1\}}|(1-b_{t_{0},B_{j}}(\psi))F_{t_{0}}|^{2}e^{-\varphi}c(-\psi)\right)^{\frac{1}{2}},
\end{split}
\end{equation}
then it follows from inequality \eqref{equ:GZc} that
\begin{equation}
\label{equ:GZe}
\begin{split}
&\left(\int_{\{\psi<-t_1\}}|\tilde{F}_{j}|^{2}e^{-\varphi}c(-\psi)\right)^{\frac{1}{2}}
\\\leq&\left(\frac{e^{t_{0}+B_{j}}\int_{t_1}^{t_{0}+B_{j}}c(t)e^{-t}dt}{\inf_{t\in(t_{0},t_{0}+B_{j})}c(t)}\right)^{\frac{1}{2}}
\left(\frac{G(t_{0})-G(t_{0}+B_{j})}{B_{j}}\right)^{\frac{1}{2}}
\\&+\left(\int_{\{\psi<-t_1\}}|(1-b_{t_{0},B_{j}}(\psi))F_{t_{0}}|^{2}e^{-\varphi}c(-\psi)\right)^{\frac{1}{2}}.
\end{split}
\end{equation}
Since $\left\{\frac{G(t_{0}+B_{j})-G(t_{0})}{B_{j}}\right\}_{j\in\mathbb{N}^{+}}$ is bounded, $\lim_{j\rightarrow+\infty}\inf_{t\in(t_{0},t_{0}+B_{j})}c(t)=\lim_{t\rightarrow t_0+0}c(t)\in(0,+\infty)$ and
$$\int_{\{\psi<-t_1\}}|(1-b_{t_{0},B_{j}}(\psi))F_{t_{0}}|^{2}e^{-\varphi}c(-\psi)\leq\int_{\{\psi<-t_0\}}|F_{t_{0}}|^{2}e^{-\varphi}c(-\psi)<+\infty,$$
then $\int_{\{\psi<-t_1\}}|\tilde{F}_{j}|^{2}e^{-\varphi}c(-\psi)$ is bounded with respect to $j$.

\

Secondly, we will prove the main result.

It follows from $\int_{\{\psi<-t_1\}}|\tilde{F}_{j}|^{2}e^{-\varphi}c(-\psi)$ is bounded with respect to $j$ and Lemma \ref{l:converge} that there exists a subsequence of $\{\tilde{F}_j\}$, denoted by $\{\tilde{F}_{j_k}\}_{k\in\mathbb{N}^+}$, which is uniformly convergent to a holomorphic $(n,0)$ form $F_1$ on $\{\psi<-t_1\}$ on any compact subset of $\{\psi<-t_1\}$ when $k\rightarrow+\infty$,  such that
$$\int_{\{\psi<-t_1\}}|F_1|^2e^{-\varphi}c(-\psi)\le\liminf_{j\rightarrow+\infty}\int_{\{\psi<-t_1\}}|\tilde{F}_{j}|^{2}e^{-\varphi}c(-\psi)<+\infty.$$  As $(\tilde{F}_j-F_{t_{0}})\in H^{0}(Z_0,(\mathcal{O}(K_{M})\otimes\mathcal{F})|_{Z_0})$ for any $j$, we have $(F_1-F_{t_{0}})\in H^{0}(Z_0,(\mathcal{O}(K_{M})\otimes\mathcal{F})|_{Z_0})$.
Note that 
\begin{equation*}
	\lim_{j\rightarrow+\infty}b_{t_0,B_j}(t)=\lim_{j\rightarrow+\infty}\int_{-\infty}^{t}\frac{1}{B_j}\mathbb{I}_{\{-t_0-B_j<s<-t_0\}}ds=\left\{ \begin{array}{lcl}
	0 & \mbox{if}& x\in(-\infty,-t_0)\\
	1 & \mbox{if}& x\in[-t_0,+\infty)
    \end{array} \right.
\end{equation*}
and
\begin{equation*}
	\lim_{j\rightarrow+\infty}v_{t_0,B_j}(t)=\lim_{j\rightarrow+\infty}\int_{-t_0}^{t}b_{t_0,B_j}ds-t_0=\left\{ \begin{array}{lcl}
	-t_0 & \mbox{if}& x\in(-\infty,-t_0)\\
	t & \mbox{if}& x\in[-t_0,+\infty)
    \end{array}. \right.
\end{equation*}
Following from equality \eqref{eq:202141a}, inequality \eqref{equ:GZc} and the Fatou's Lemma, we have 
\begin{equation}
	\label{eq:211106b}\begin{split}
		&\int_{\{\psi<-t_0\}}|F_1-F_{t_0}|e^{-\varphi-\psi-t_0}c(t_0)+\int_{\{-t_0\le\psi<-t_1\}}|F_1|^2e^{-\varphi}c(-\psi)\\
		=&\int_{\{\psi<-t_1\}}\lim_{k\rightarrow+\infty}|\tilde{F}_{j_k}-(1-b_{t_{0},B_{j_k}}(\psi))F_{t_{0}}|^{2}e^{-\varphi}e^{-\psi+v_{t_0,B_{j_k}}(\psi)}c(-v_{t_0,B_{j_k}}(\psi))
\\
\le&\liminf_{k\rightarrow+\infty}\int_{\{\psi<-t_1\}}|\tilde{F}_{j_k}-(1-b_{t_{0},B_{j_k}}(\psi))F_{t_{0}}|^{2}e^{-\varphi}e^{-\psi+v_{t_0,B_{j_k}}(\psi)}c(-v_{t_0,B_{j_k}}(\psi))
\\
\leq&
\liminf_{k\rightarrow+\infty}\left(\frac{e^{t_{0}+B_{j_k}}\int_{t_1}^{t_{0}+B_{j_k}}c(t)e^{-t}dt}{\inf_{t\in(t_{0},t_{0}+B_{j_k})}c(t)}
\times\frac{G(t_{0})-G(t_{0}+B_{j_k})}{B_{j_k}}\right)\\
=&\frac{e^{t_0}\int_{t_1}^{t_0}c(t)e^{-t}dt}{\lim_{t\rightarrow t_0+0}c(t)}\lim_{j\rightarrow +\infty}\frac{G(t_0)-G(t_0+B_j)}{B_j}\\
=&\int_{t_1}^{t_0}c(t)e^{-t}dt\lim_{j\rightarrow +\infty}\frac{G(t_0)-G(t_0+B_j)}{\int_{t_0}^{t_0+B_j}c(t)e^{-t}dt}.
	\end{split}
\end{equation}

As $e^{\psi}c(-\psi)\le e^{-t_0}c(t_0)$ on $\{\psi<-t_0\}$, it follows Lemma \ref{lem:A}, equality \eqref{eq:211106d} and inequality \eqref{eq:211106b} that 
\begin{equation}
	\label{eq:211106c}
	\begin{split}&\int_{t_1}^{t_0}c(t)e^{-t}dt\liminf_{B\rightarrow 0+0}\frac{G(t_0)-G(t_0+B)}{\int_{t_0}^{t_0+B}c(t)e^{-t}dt}\\
	=&\int_{t_1}^{t_0}c(t)e^{-t}dt\lim_{j\rightarrow +\infty}\frac{G(t_0)-G(t_0+B_j)}{\int_{t_0}^{t_0+B_j}c(t)e^{-t}dt}\\
	\ge&	\int_{\{\psi<-t_0\}}|F_1-F_{t_0}|e^{-\varphi-\psi-t_0}c(t_0)+\int_{\{-t_0\le\psi<-t_1\}}|F_1|^2e^{-\varphi}c(-\psi)\\
	\ge&\int_{\{\psi<-t_0\}}|F_1-F_{t_0}|e^{-\varphi}c(-\psi)+\int_{\{-t_0\le\psi<-t_1\}}|F_1|^2e^{-\varphi}c(-\psi)\\
	=&\int_{\{\psi<-t_1\}}|F_1|^2e^{-\varphi}c(-\psi)-\int_{\{\psi<-t_0\}}|F_{t_0}|^2e^{-\varphi}c(-\psi)\\
	\ge& G(t_1)-G(t_0).
	\end{split}
\end{equation}
This proves Lemma \ref{lem:C}.
\end{proof}

The following well-known property of concave functions will be used in the proof of Theorem \ref{thm:general_concave}.
\begin{Lemma}
\label{lem:Ea}
Let $a(r)$ be a lower semicontinuous function on $(A,B)$ ( $-\infty\leq A<B\leq +\infty$).
Then $a(r)$ is concave if and only if
\begin{equation}
	\label{eq:concave}
	\frac{a(r_{2})-a(r_{1})}{r_{2}-r_{1}}\geq \limsup_{r\to r_{2}+0}\frac{a(r)-a(r_{2})}{r-r_{2}},
\end{equation}
holds for any $A<r_{1}<r_{2}<B$.
\end{Lemma}
\begin{proof}
	For the convenience of the reader, we recall the proof.
	
	It suffices to prove that inequality \eqref{eq:concave} implies the concavity of $a(r)$. We prove by contradiction: if not, there exists $A< r_3<r_4<r_5<B$ such that
	\begin{equation}
		\label{eq:202142c}
		\frac{a(r_4)-a(r_3)}{r_4-r_3}<\frac{a(r_5)-a(r_3)}{r_5-r_3}<\frac{a(r_5)-a(r_4)}{r_5-r_4}.
	\end{equation}
	Consider $\tilde a(r)=a(r)-a(r_5)-\frac{a(r_5)-a(r_3)}{r_5-r_3}(r-r_5)$ on $(A,B)$. As $a(r)$ is lower semicontinuous  on $(A,B)$, then $\tilde a(r)$ is lower semicontinuous  on $(A,B)$. Note that $\tilde{a}(r_3)=\tilde{a}(r_5)=0$ and $\tilde{a}(r_4)<0$, then it follows from the lower semicontinuity of $\tilde{a}(r)$ that there exists $r_6\in(r_3,r_5)$ such that $\tilde{a}(r_6)=inf_{r\in[r_3,r_5]}\tilde{a}(r)<0$. It clear that $\frac{\tilde{a}(r_6)-\tilde{a}(r_3)}{r_6-r_3}<0$ and $\limsup_{r\rightarrow r_6+0}\frac{\tilde{a}(r)-\tilde{a}(r_6)}{r-r_6}\geq0$. Then we obtain that $\frac{{a}(r_6)-{a}(r_3)}{r_6-r_3}<\frac{a(r_5)-a(r_3)}{r_5-r_3}\leq\limsup_{r\rightarrow r_6+0}\frac{{a}(r)-{a}(r_6)}{r-r_6}$, which contradict inequality \eqref{eq:concave}.
\end{proof}

\subsection{Some results used in the proofs of applications}
In this section, we give some results which will be used in the proofs of applications in Section \ref{sec:app}.

\begin{Lemma}
	\label{l:c'}
	If $c(t)$ is a positive measurable function on $(T,+\infty)$ such that $c(t)e^{-t}$ is  decreasing on $(T,+\infty)$ and $\int_{T_1}^{+\infty}c(t)e^{-t}dt<+\infty$ for some $T_1>T$, then there exists a positive measurable function $\tilde{c}$ on $(T,+\infty)$, satisfying the following statements:
	
	$(1)$ $\tilde{c}\geq c$ on $(T,+\infty)$;
	
	$(2)$ $\tilde{c}(t)e^{-t}$ is strictly decreasing on $(T,+\infty)$ and $\tilde{c}$ is  increasing on $(a,+\infty)$, where $a>T$ is a real number;
	
	$(3)$ $\int_{T_1}^{+\infty}\tilde{c}(t)e^{-t}dt<+\infty$.
	
	Moreover, if $\int_{T}^{+\infty}c(t)e^{-t}dt<+\infty$ and $c\in\mathcal{P}_T$, we can choose $\tilde{c}$ satisfying the above conditions, $\int_{T}^{+\infty}\tilde{c}(t)e^{-t}dt<+\infty$ and $\tilde{c}\in\mathcal{P}_T$.
\end{Lemma}
\begin{proof}
	Without loss of generality, we can assume that $T<0$.
Let $a_n=c(n)e^{-n}$, where $n\in \mathbb{N}^+$. Take $b_1=a_1$, and we can define $b_n=\max{\{\frac{b_{n-1}}{e},a_n\}}$  for $n>1$, inductively. Since $a(n)$ is decreasing with respect to $n$, we have $b_n\geq b_{n+1}\geq \frac{b_{n}}{e}$ and $b_n\geq a_n$ for any $n\in\mathbb{N}^{+}$.

Let \begin{displaymath}
	\tilde a(t)=\left\{\begin{array}{lcl}
	
		eb_n(\frac{b_{n+1}}{b_n})^{t-n} & \mbox{if} & t\in[n,n+1),\\
		c(t)e^{-t+1} &\mbox{if} & t\in(T,1].
		\end{array}
	\right.
\end{displaymath}
It clear that $\tilde a(t)\geq c(t)e^{-t}$,  $\tilde a(t)$ is decreasing and continuous on $(T,+\infty)$. Let $\tilde c(t)=\tilde a(t)e^{t}$. When $t\in[n,n+1)$, as $eb_{n+1}\geq b_n$, we have $\tilde c(t)$ is increasing on $[n,n+1)$, which implies that $\tilde c(t)$ is increasing on $(1,+\infty)$.

As $\int_{0}^{+\infty}c(t)e^{-t}dt<+\infty$, then $\sum_{n=1}^{+\infty}a_n<+\infty$. In the following, we will prove $\int_{0}^{+\infty}\tilde c(t)e^{-t}<+\infty$. By the definition of $\tilde c(t)$, we have
\begin{equation}
\label{eq:210618c}
	\int_{0}^{+\infty}\tilde c(t)e^{-t}=\int_{0}^{1}\tilde a(t)dt+\sum_{n=1}^{+\infty}\int_{n}^{n+1}\tilde a(t)dt \leq c(0)e+e\sum_{n=1}^{+\infty}b_n.\end{equation}
Take $I=\{n_i:n_i$ is the $i-$th positive integer such that $a_{n_i}=b_{n_i}$$\}\in\mathbb{N}^+$. Note that if $a_{n+1}\not=b_{n+1}$, then $b_{n+1}=\frac{b_{n}}{e}$, thus we have
\begin{equation}
	\label{eq:210618d}
	\begin{split}
		\sum_{n=1}^{+\infty}b_n=&\sum_{i=1}\sum_{j=0}^{n_{i+1}-n_i-1}b_{n_i+j}\\
		=&\sum_{i=1}\sum_{j=0}^{n_{i+1}-n_i-1}b_{n_i}e^{-j}\\
		\leq&\sum_{i=1}a_{n_i}\frac{e}{e-1}\\
		<&+\infty ,
	\end{split}
\end{equation}
where if $n_i$ is the largest integer such that $a_{n_i}=b_{n_i}$, take $n_{i+1}=+\infty$. Combining inequality \eqref{eq:210618c} and \eqref{eq:210618d}, we obtain $\int_{0}^{+\infty}\tilde c(t)e^{-t}dt<+\infty$. Let $\tilde{c}(t)=\tilde c(t)+1$,  we have $\tilde{c}\geq c$, $\tilde{c}$ is  increasing on $(1,+\infty)$, $c(t)e^{-t}$ is strictly decreasing on $(T,+\infty)$ and $\int_{0}^{+\infty}\tilde{c}(t)e^{-t}dt<+\infty$.

Moreover, if $\int_{T}^{+\infty}c(t)e^{-t}dt<+\infty$ and $c\in\mathcal{P}_T$, as $\tilde{c}(t)\geq c(t)$ on $(T,+\infty)$ and $\tilde{c}(t)=ec(t)+1$ on $(T,1)$, we have $\int_{T}^{+\infty}\tilde{c}(t)e^{-t}dt<+\infty$ and $c\in\mathcal{P}_T$. Thus, Lemma \ref{l:c'} holds.
\end{proof}

Let $\Omega$ be an open Riemann surface admitted a nontrivial Green function $G_{\Omega}$.
 Let $w$ be a local coordinate on a neighborhood $V_{z_0}$ of $z_0\in\Omega$ satisfying $w(z_0)=0$.

 \begin{Lemma}[see \cite{S-O69}, see also \cite{Tsuji}]
 	\label{l:green}
 	$G_{\Omega}(z,z_0)=\sup_{v\in\Delta_0(z_0)}v(z)$, where $\Delta_0(z_0)$ is the set of negative subharmonic functions on $\Omega$ satisfying that $v-\log|w|$ has a locally finite upper bound near $z_0$.
 \end{Lemma}

\begin{Lemma}\label{l:G-compact}
For any open neighborhood $U$ of $z_0$, there exists $t>0$ such that $\{G_{\Omega}(z,z_0)<-t\}$ is a relatively compact subset of $U$.
\end{Lemma}
\begin{proof}
Let $w$ be a coordinate on a neighborhood $V_{z_0}\subset\subset U$ of $z_0$, such that $w(z_0)=0$ and $G_{\Omega}(z,z_0)=\log|w(z)|+v(w(z))$, where $v$ is a harmonic function on $V_{z_0}$ and  $\sup_{V_{z_0}}|v(w(z))|<+\infty$. Then there exists $t>0$ such that $\{z\in V_{z_0}:\log|w(z)|+v(w(z))<-t\}\subset\subset V_{z_0}$.

We claim that $\{z\in\Omega:G_{\Omega}(z,z_0)<-t\}\subset\subset V_{z_0}$, therefore Lemma \ref{l:G-compact} holds. In fact, set
$$\tilde{G}(z) = \left\{ \begin{array}{lcl}
G_{\Omega}(z,z_0) & \mbox{if}
& z \in V_{z_0}, \\ \max\{G_{\Omega}(z,z_0),-t\} & \mbox{if} &
z\in \Omega\backslash V_{z_0}.
\end{array}\right.
$$
As $\{z\in V_{z_0}:\log|w(z)|+v(w(z))<-t\}\subset\subset V_{z_0}$, we know $\tilde G(z)$ is subharmonic on $\Omega$. Lemma \ref{l:green} tells us $\tilde G(z)\leq G_{\Omega}(z,z_0)$, therefore $\{z\in\Omega:G_{\Omega}(z,z_0)<-t\}=\{z\in V_{z_1}:G_{\Omega}(z,z_0)<-t\}\subset\subset V_{z_0}$.
\end{proof}

\begin{Lemma}\label{l:lo}
 For any
$z_0\in \Omega$ and open subsets $V_1$ and $U_1$ of $\Omega$ satisfying $z_0\in V_1\subset\subset U_1\subset\subset\Omega$,  there exists a constant $N>0$ such that
$$G_{\Omega}(z,z_1)\ge NG_{\Omega}(z,z_0)$$
holds for any $(z,z_1)\in
(\Omega\backslash U_1)\times V_1$.
\end{Lemma}
\begin{proof}
As $V_1\subset\subset U_1\subset\subset\Omega$, fixed $z\in\Omega\backslash U_1$, $G_{\Omega}(z,z_1)$ is harmonic with respect to $z_1$ on a open neighborhood of $\overline {V_1}$. The Harnack inequality shows that there exists a constant $N>0$ such that
\begin{equation}
	\label{eq:210623g}
	\sup_{z_1\in\overline {V_1}}(-G_{\Omega}(z,z_1))\leq N\inf_{z_1\in\overline {V_1}}(-G_{\Omega}(z,z_1))
\end{equation}
holds of any $z\in\Omega\backslash U_1$. As $z_0\in V_1$, it follows from inequality \eqref{eq:210623g} that
$$G_{\Omega}(z,z_1)\ge NG(z,z_0)$$
holds for any $(z,z_1)\in
(\Omega\backslash U_1)\times V_1$.
\end{proof}

The following lemma (proof can be referred to Section \ref{sec:cu}) will be used in the proof of Theorem \ref{thm:e2}.
 \begin{Lemma}
 	\label{l:cu}
 	Let $T$ be a closed positive $(1,1)$ current on $\Omega$. For any open set $U\subset\subset \Omega$ satisfying $U\cap supp T\not=\emptyset$,  there exists a subharmonic function $\Phi<0$ on $\Omega$, which satisfies the following properties:
 	
 	$(1)$ $i\partial\bar\partial\Phi\leq T$ and $i\partial\bar\partial\Phi\not\equiv0$;
 	
 	$(2)$ $\lim_{t\rightarrow0+0}(\inf_{\{G_{\Omega}(z,z_0)\geq-t\}}\Phi(z))=0$;
 	
 	$(3)$ $supp (i\partial\bar\partial\Phi)\subset U$ and $\inf_{\Omega\backslash U}\Phi>-\infty$.
 \end{Lemma}

Now, we recall some notations. Let $c_{\beta}(z)$ be the logarithmic capacity which is locally defined by
 $$c_{\beta}(z_0):=\exp\lim_{z\rightarrow z_0}(G_{\Omega}(z,z_0)-\log|w(z)|)$$
 on $\Omega$ (see \cite{S-O69}).
 The weighted Bergman kernel $\kappa_{\Omega,\rho}$ with weight $\rho$ of holomorphic $(1,0)$ form on $\Omega$ is defined by $\kappa_{\Omega,\rho}:=\sum_{i}e_i\otimes\bar{e}_i$, where $\{e_i\}_{i=1,2,...}$ are holomorphic $(1,0)$ forms on $\Omega$ and satisfy $\sqrt{-1}\int_{\Omega}\rho\frac{e_i}{\sqrt{2}}\wedge\frac{\bar{e}_j}{\sqrt{2}}=\delta_i^{j}$. Let $B_{\Omega,\rho}(z):=\frac{\kappa_{\Omega,\rho}(z)}{|dw|^{2}}$ on $V_{z_0}$.

\begin{Theorem}[\cite{guan-zhou13ap}](A solution of the extended Suita Conjecture)
	\label{thm:suita}Let $u$ be a harmonic function on $\Omega$.
	$c_{\beta}^2(z_0)\leq\pi e^{-2u(z_0)}B_{\Omega,e^{-2u}}(z_0)$ holds, and the equality holds if and only if $\chi_{-u}=\chi_{z_0}.$
\end{Theorem}

\section{proofs of Theorem \ref{thm:general_concave}, Corollary \ref{infty}, Corollary \ref{thm:linear} and Corollary \ref{thm:1}}

In this section, we prove Theorem \ref{thm:general_concave}, Corollary \ref{infty}, Corollary \ref{thm:linear} and Corollary \ref{thm:1}.

\subsection{Proof of Theorem \ref{thm:general_concave}}
Firstly, we prove that if $G(t_0)<+\infty$ for some $t_0>T$, then $G(t_1)<+\infty$ for any $t_1\in(T,t_0)$.  It follows from Lemma \ref{lem:A} that there exists a holomorphic $(n,0)$ form $F_{t_0}$ on $\{\psi<-t_0\}$ satisfying $(F_{t_0}-f)\in H^0(Z_0,(\mathcal{O}(F_M)\otimes\mathcal{F})|_{Z_0})$ and $\int_{\{\psi<-t_0\}}|F_{t_0}|^2e^{-\varphi}c(-\psi)=G(t_0)<+\infty$. Using Lemma \ref{lem:GZ_sharp2}  we get a holomorphic $(n,0)$ form $\tilde F$ on $\{\psi<-t_1\}$, such that
$$(\tilde F-F_{t_0})\in H^0(Z_0,(\mathcal{O}(F_M)\otimes\mathcal{I}(\varphi+\psi))|_{Z_0})\subset H^0(Z_0,(\mathcal{O}(F_M)\otimes\mathcal{F})|_{Z_0})$$
and
\begin{equation}
	\label{eq:202142a}
	\begin{split}
		&\int_{\{\psi<-t_1\}}|\tilde F-(1-b_{t_0,B}(\psi))F_{t_0}|^2e^{-\varphi}c(-\psi)\\
		\leq&\int_{\{\psi<-t_1\}}|\tilde F-(1-b_{t_0,B}(\psi))F_{t_0}|^2e^{-\varphi-\psi+v_{t_0,B}(\psi)}c(-v_{t_0,B}(\psi))\\
		\leq&\left(\int_{t_1}^{t_0+B}c(t)e^{-t}dt\right)\int_{\{\psi<-t_1\}}\frac{1}{B}\mathbb{I}_{\{-t_0+B<\psi<-t_0\}}|F_{t_0}|^2e^{-\varphi-\psi}.
	\end{split}
\end{equation}
Note that
\begin{displaymath}
	\begin{split}
		&\left(\int_{\{\psi<-t_1\}}|\tilde F|^2e^{-\varphi}c(-\psi)\right)^{\frac{1}{2}}-\left(\int_{\{\psi<-t_1\}}|(1-b_{t_0,B}(\psi))F_{t_0}|^2e^{-\varphi}c(-\psi)\right)^{\frac{1}{2}}\\
		\leq&\left(\int_{\{\psi<-t_1\}}|\tilde F-(1-b_{t_0,B}(\psi))F_{t_0}|^2e^{-\varphi}c(-\psi)\right)^{\frac{1}{2}},
	\end{split}
\end{displaymath}
combining with inequality \eqref{eq:202142a}, we obtain
\begin{equation}
	\label{eq:202142b}
	\begin{split}
		&\left(\int_{\{\psi<-t_1\}}|\tilde F|^2e^{-\varphi}c(-\psi)\right)^{\frac{1}{2}}\\
		\leq&\left(\left(\int_{t_1}^{t_0+B}c(t)e^{-t}dt\right)\int_{\{\psi<-t_1\}}\frac{1}{B}\mathbb{I}_{\{-t_0-B<\psi<-t_0\}}|F_{t_0}|^2e^{-\varphi-\psi}\right)^{\frac{1}{2}}\\
		&+\left(\int_{\{\psi<-t_1\}}|(1-b_{t_0,B}(\psi))F_{t_0}|^2e^{-\varphi}c(-\psi)\right)^{\frac{1}{2}}.
	\end{split}
\end{equation}
As $b_{t_0,B}(\psi)=1$ on $\{\psi\geq t_0\}$, $0\leq b_{t_0,B}(\psi)\leq1$, $\int_{\{\psi<-t_0\}}|F_{t_0}|^2e^{-\varphi}c(-\psi)<+\infty$, and $c(t)$ has a positive lower bound on any compact subset of $(T,+\infty)$, then
$$\left(\int_{\{\psi<-t_1\}}|(1-b_{t_0,B}(\psi)) F_{t_0}|^2e^{-\varphi}c(-\psi)\right)^{\frac{1}{2}}<+\infty$$
and
\begin{displaymath}
	\begin{split}
	&\left(\int_{t_1}^{t_0+B}c(t)e^{-t}dt\right)\int_{\{\psi<-t_1\}}\frac{1}{B}\mathbb I_{\{-t_0-B<\psi<-t_0\}}|F_{t_0}|^2e^{-\varphi-\psi}\\
	\leq&\frac{e^{t_0+B}\int_{t_1}^{t_0+B}c(t)e^{-t}dt}{\inf_{t\in(t_0,t_0+B)}c(t)}\int_{\{\psi<-t_1\}}\frac{1}{B}\mathbb I_{\{-t_0-B<\psi<-t_0\}}|F_{t_0}|^2e^{-\varphi}c(-\psi)\\
	<&+\infty,
	\end{split}
\end{displaymath}
 which implies that $$\int_{\{\psi<-t_1\}}|\widetilde F|^2e^{-\varphi}c(-\psi)<+\infty.$$
 Then we obtain $G(t_1)\leq\int_{\{\psi<-t_1\}}|\widetilde F|^2e^{-\varphi}c(-\psi)<+\infty$.

Now, assume that $G(t_0)<+\infty$ for some $t_0\geq T$ (otherwise it is clear that $G(t)\equiv+\infty$).
As $G(h^{-1}(r))$ is lower semicontinuous (Lemma \ref{lem:B}),
then Lemma \ref{lem:C} and Lemma \ref{lem:Ea} imply the concavity of $G(h^{-1}(r))$.
It follows from Lemma \ref{lem:B} that $\lim_{t\rightarrow T+0}G(t)=G(T)$ and $\lim_{t\rightarrow +\infty}G(t)=0$, hence we prove Theorem \ref{thm:general_concave}.
\subsection{Proof of Corollary \ref{infty}}
Note that if there exists a positive decreasing concave function $g(t)$ on $(a,b)\subset\mathbb{R}$ and $g(t)$ is not a constant function, then $b<+\infty$. We prove Corollary \ref{infty} by contradiction: if $G(t)<+\infty$ for some $t\geq T$, as $f\not\in H^0(Z_0,(\mathcal{O}(K_M)\otimes\mathcal{F})|_{Z_0})$, Lemma \ref{lem:0} shows that $G(t)\in(0,+\infty)$. Following from Theorem \ref{thm:general_concave} we know $G({h}^{-1}(r))$ is concave with respect to $r\in(\int_{T_1}^{T}c(t)e^{-t}dt,\int_{T_1}^{+\infty}c(t)e^{-t}dt)$ and $G({h}^{-1}(r))$ is not a constant function, therefore we obtain $\int_{T_1}^{+\infty}c(t)e^{-t}dt<+\infty$, which contradicts to $\int_{T_1}^{+\infty}c(t)e^{-t}dt=+\infty$. Thus Corollary \ref{infty} holds.
\subsection{Proof of Corollary \ref{thm:linear}}

 If $G(t)\in(0,+\infty)$ for some $t\geq T$, Corollary \ref{infty} and Lemma \ref{lem:0} show that $\int_{T_1}^{+\infty}c(t)e^{-t}dt<+\infty$.
	 As $\lim_{t\rightarrow+\infty}G(t)=0$, then $G(h^{-1}(r))$ is concave on $(\int_{T_1}^{T}c(t)e^{-t}dt,\int_{T_1}^{+\infty}c(t)e^{-t}dt]$ by defining $G(+\infty)=0$. Then the concavity of $G(h^{-1}(r))$  implies that the three statements are equivalent.

\subsection{Proof of Corollary \ref{thm:1}}
It follows from Corollary \ref{thm:linear} that $G(t)=\frac{G(T_1)}{\int_{T_1}^{+\infty}c(s)e^{-s}ds}\int_t^{+\infty}c(s)e^{-s}ds$ for any  $t\in[T,+\infty)$.

Firstly, we prove the existence and uniqueness of $F$.

Following the notations in Lemma \ref{lem:C}, as $G(t)=\frac{G(T_1)}{\int_{T_1}^{+\infty}c(s)e^{-s}ds}\int_t^{+\infty}c(s)e^{-s}ds\in(0,+\infty)$ for any  $t\in(T,+\infty)$,  by choosing $t_1\in(T,+\infty)$ and $t_0>t_1$, we know that the inequality  \eqref{eq:211106c} must be equality, which implies that
\begin{equation}
	\label{eq:20210412c}
	\int_{\{\psi<-t_0\}}|F_1-F_{t_0}|^2e^{-\varphi}(e^{-\psi-t_0}c(t_0)-c(-\psi))=0,
\end{equation}
where $F_1$ is a holomorphic $(n,0)$ form on $\{\psi<-t_1\}$ such that $(F_1-f)\in H^{0}(Z_0,(\mathcal{O}(K_{M})\otimes\mathcal{F})|_{Z_0})$  and $F_{t_0}$ is a holomorphic $(n,0)$ form on $\{\psi<-t_0\}$ such that $(F_{t_0}-f)\in H^{0}(Z_0,(\mathcal{O}(K_{M})\otimes\mathcal{F})|_{Z_0})$.
As $\int_{T_1}^{+\infty}c(t)e^{-t}<+\infty$ and $c(t)e^{-t}$ is decreasing, then there exists $t_2>t_0$ such that $c(t)e^{-t}<c(t_0)e^{-t_0}-\delta$ for any $t\geq t_2$, where $\delta$ is a positive constant. Then equality \eqref{eq:20210412c} implies that
\begin{displaymath}
	\begin{split}
		&\delta\int_{\{\psi<-t_2\}}|F_1-F_{t_0}|^2e^{-\varphi}e^{-\psi}\\
		\le&\int_{\{\psi<-t_2\}}|F_1-F_{t_0}|^2e^{-\varphi}(e^{-\psi-t_0}c(t_0)-c(-\psi))\\
		\le&\int_{\{\psi<-t_0\}}|F_1-F_{t_0}|^2e^{-\varphi}(e^{-\psi-t_0}c(t_0)-c(-\psi))\\
		=&0
	\end{split}
\end{displaymath}
It follows from $\varphi+\psi$ is plurisubharmonic function and $F_1$ and $F_{t_0}$ are holomorphic $(n,0)$ forms that $F_1=F_{t_0}$ on $\{\psi<-t_0\}$. As $\int_{\{\psi<-t_0\}}|F_{t_0}|^2e^{-\varphi}c(-\psi)=G(t_0)$ and the inequality  \eqref{eq:211106c} becomes equality, we have
$$\int_{\{\psi<-t_1\}}|F_1|^2e^{-\varphi}c(-\psi)=G(t_1).$$
Following from Lemma \ref{lem:A},  there exists a unique holomorphic $(n,0)$ form $F_t$ on $\{\psi<-t\}$ satisfying $(F_t-f)\in H^{0}(Z_0,(\mathcal{O}(K_{M})\otimes\mathcal{F})|_{Z_0})$ and
$\int_{\{\psi<-t\}}|F_t|^{2}e^{-\varphi}c(-\psi)=G(t)$ for any $t>T$. By discussion in the above, we know $F_{t}=F_{t'}$ on $\{\psi<-\max{\{t,t'\}}\}$ for any $t\in(T,+\infty)$ and $t'\in(T,+\infty)$. Hence combining $\lim_{t\rightarrow T+0}G(t)=G(T)$, we obtain that there  exists a unique holomorphic $(n,0)$ form $F$ on $M$ satisfying $(F-f)\in H^{0}(Z_0,(\mathcal{O}(K_{M})\otimes\mathcal{F})|_{Z_0})$ and
$\int_{\{\psi<-t\}}|F|^{2}e^{-\varphi}c(-\psi)=G(t)$ for any $t\geq T$.

Secondly, we prove equality \eqref{eq:20210412b}.
As $a(t)$ is nonnegative measurable function on $(T,+\infty)$, then there exists a sequence of functions $\{\sum_{j=1}^{n_i}a_{ij}\mathbb I_{E_{ij}}\}_{i\in\mathbb{N}^+}$ ($n_i<+\infty$ for any $i\in\mathbb{N}^+$) satisfying $\sum_{j=1}^{n_i}a_{ij}\mathbb I_{E_{ij}}$ is increasing with respect to $i$ and $\lim_{i\rightarrow +\infty}\sum_{j=1}^{n_i}a_{ij}\mathbb I_{E_{ij}}(t)=a(t)$ for any $t\in(T,+\infty)$, where $E_{ij}$ is a Lebesgue measurable subset
of $(T,+\infty)$ and $a_{ij}\geq 0$ is a constant. It follows from Levi's Theorem that it suffices to prove the case that $a(t)=\mathbb I_{E}(t)$, where $E\subset\subset (T,+\infty)$ is a Lebesgue measurable set .

Note that $G(t)=\int_{\{\psi<-t\}}|F|^2e^{-\varphi}c(-\psi)=\frac{G(T_1)}{\int_{T_1}^{+\infty}c(s)e^{-s}ds}\int_{t}^{+\infty}c(s)e^{-s}ds$, then
\begin{equation}
	\label{eq:20210412d}
	\int_{\{-t_1\leq\psi<-t_2\}}|F|^2e^{-\varphi}c(-\psi)=\frac{G(T_1)}{\int_{T_1}^{+\infty}c(s)e^{-s}ds}\int_{t_2}^{t_1}c(s)e^{-s}ds
\end{equation}
holds for any $T\leq t_2<t_1<+\infty$. It follows from the dominated convergence theorem and inequality \eqref{eq:20210412d} that
\begin{equation}
	\label{eq:20210412g}
	\int_{\{z\in M:-\psi(z)\in N\}}|F|^2e^{-\varphi}=0
\end{equation}
holds for any $N\subset\subset(T,+\infty)$ such that $\mu(N)=0$, where $\mu$ is Lebesgue measure.

As $c(t)e^{-t}$ is decreasing on $(T,+\infty)$, there are at most countable points denoted by $\{s_j\}_{j\in\mathbb{N}^+}$ such that $c(t)$ is not continuous at $s_j$. Then there is a decreasing sequence open sets $\{U_k\}$, such that $\{s_j\}_{j\in\mathbb{N}^+}\subset U_k\subset(T,+\infty)$ for any $j$, and $\lim_{k\rightarrow+\infty}\mu(U_k)=0$.
Choosing any closed interval $[t_2',t_1']\subset (T,+\infty)$. Then we have
\begin{equation}
	\label{eq:20210412e}
	\begin{split}
		&\int_{\{-t_1'\leq\psi<-t_2'\}}|F|^2e^{-\varphi}\\
		=&\int_{\{z\in M:-\psi(z)\in(t_2',t_1']\backslash U_k\}}|F|^2e^{-\varphi}+\int_{\{z\in M:-\psi(z)\in[t_2',t_1']\cap U_k\}}|F|^2e^{-\varphi}		
		\\=&\lim_{n\rightarrow+\infty}\sum_{i=0}^{n-1}\int_{\{z\in M:-\psi(z)\in I_{n,i}\backslash U_k\}}|F|^2e^{-\varphi}+\int_{\{z\in M:-\psi(z)\in[t_2',t_1']\cap U_k\}}|F|^2e^{-\varphi}	,
	\end{split}
\end{equation}
where $I_{n,i}=(t_1'-(i+1)\alpha_{n},t_1'-i\alpha_{n}]$ and $\alpha_n=\frac{t_1'-t_2'}{n}$. Note that
\begin{equation}
	\label{eq:2106a}
	\begin{split}
		&\lim_{n\rightarrow+\infty}\sum_{i=0}^{n-1}\int_{\{z\in M:-\psi(z)\in I_{n,i}\backslash U_k\}}|F|^2e^{-\varphi}\\
		\leq&\limsup_{n\rightarrow+\infty}\sum_{i=0}^{n-1}\frac{1}{\inf_{I_{n,i}\backslash U_k}c(t)}\int_{\{z\in M:-\psi(z)\in I_{n,i}\backslash U_k\}}|F|^2e^{-\varphi}c(-\psi).	\end{split}
\end{equation}
It follows from equality \eqref{eq:20210412d} that inequality \eqref{eq:2106a} becomes
\begin{equation}
	\label{eq:2106b}
	\begin{split}
		&\lim_{n\rightarrow+\infty}\sum_{i=0}^{n-1}\int_{\{z\in M:-\psi(z)\in I_{n,i}\backslash U_k\}}|F|^2e^{-\varphi}\\
		\leq&\frac{G(T_1)}{\int_{T_1}^{+\infty}c(s)e^{-s}ds}\limsup_{n\rightarrow+\infty}\sum_{i=0}^{n-1}\frac{1}{\inf_{I_{n,i}\backslash U_k}c(t)}\int_{I_i\backslash U_k}c(s)e^{-s}ds.	\end{split}
\end{equation}
It is clear that $c(t)$ is uniformly continuous  and has a positive lower bound and upper bound on $[t_2',t_1']\backslash U_k$. Then we have
\begin{equation}
		\label{eq:2106c}
	\begin{split}
		&\limsup_{n\rightarrow+\infty}\sum_{i=0}^{n-1}\frac{1}{\inf_{I_{n,i}\backslash U_k}c(t)}\int_{I_{n,i}\backslash U_k}c(s)e^{-s}ds\\
		\leq&\limsup_{n\rightarrow+\infty}\sum_{i=0}^{n-1}\frac{\sup_{I_{n,i}\backslash U_k}c(t)}{\inf_{I_{n,i}\backslash U_k}c(t)}\int_{I_{n,i}\backslash U_k}e^{-s}ds\\
		=&\int_{(t_2',t_1']\backslash U_k}e^{-s}ds.
		\end{split}
\end{equation}
Combining inequality \eqref{eq:20210412e}, \eqref{eq:2106b} and \eqref{eq:2106c}, we have
\begin{equation}
	\label{eq:2106d}
	\begin{split}
		&\int_{\{-t_1'\leq\psi<-t_2'\}}|F|^2e^{-\varphi}\\
		=&\int_{\{z\in M:-\psi(z)\in(t_2',t_1']\backslash U_k\}}|F|^2e^{-\varphi}+\int_{\{z\in M:-\psi(z)\in[t_2',t_1']\cap U_k\}}|F|^2e^{-\varphi}		
\\\leq&\frac{G(T_1)}{\int_{T_1}^{+\infty}c(s)e^{-s}ds}\int_{(t_2',t_1']\backslash U_k}e^{-s}ds+\int_{\{z\in M:-\psi(z)\in[t_2',t_1']\cap U_k\}}|F|^2e^{-\varphi}.
	\end{split}
\end{equation}
Let $k\rightarrow+\infty$, following from equality \eqref{eq:20210412g} and inequality
\eqref{eq:2106d}, we obtain that
\begin{equation}
	\label{eq:20210414a}
	\int_{\{-t_1'\leq\psi<-t_2'\}}|F|^2e^{-\varphi}\leq\frac{G(T_1)}{\int_{T_1}^{+\infty}c(s)e^{-s}ds}\int_{t_2'}^{t_1'}e^{-s}ds.
	\end{equation}
Following from a similar discussion, we  obtain
$$\int_{\{-t_1'\leq\psi<-t_2'\}}|F|^2e^{-\varphi}\geq\frac{G(T_1)}{\int_{T_1}^{+\infty}c(s)e^{-s}ds}\int_{t_2'}^{t_1'}e^{-s}ds,$$
then combining inequality \eqref{eq:20210414a}, we know that
\begin{equation}
	\label{eq:20210412f}
	\int_{\{-t_1'\leq\psi<-t_2'\}}|F|^2e^{-\varphi}=\frac{G(T_1)}{\int_{T_1}^{+\infty}c(s)e^{-s}ds}\int_{t_2'}^{t_1'}e^{-s}ds.
\end{equation}
Then it is clear that for any open set $U\subset(T,+\infty)$ and compact set $V\subset(T,+\infty)$
$$\int_{\{z\in M:-\psi(z)\in U\}}|F|^2e^{-\varphi}=\frac{G(T_1)}{\int_{T_1}^{+\infty}c(s)e^{-s}ds}\int_{U}e^{-s}ds$$
and
$$\int_{\{z\in M:-\psi(z)\in V\}}|F|^2e^{-\varphi}=\frac{G(T_1)}{\int_{T_1}^{+\infty}c(s)e^{-s}ds}\int_{V}e^{-s}ds.$$

As $E\subset\subset(T,+\infty)$, then $E\cap(t_2,t_1]$ is Lebesgue measurable subset of $(T+\frac{1}{n},n)$ for some large $n$, where $T\leq t_2<t_1\leq+\infty$. Then there exist  a sequence of compact sets $\{V_j\}$ and a sequence of open sets $\{V'_j\}$  satisfying $V_1\subset...\subset V_{j}\subset V_{j+1}\subset...\subset E\cap(t_2,t_1]\subset...\subset V'_{j+1}\subset V'_{j}\subset...\subset V'_1\subset\subset(T,+\infty)$ and $\lim_{j\rightarrow+\infty}\mu(V'_j-V_j)=0$, where $\mu$ is Lebesgue measure. Then we have
\begin{displaymath}
	\begin{split}
		\int_{\{-t_1\leq\psi<-t_2\}}|F|^2e^{-\varphi}\mathbb I_{E}(-\psi)
		=&\int_{\{z\in M:-\psi(z)\in E\cap(t_2,t_1]\}}|F|^2e^{-\varphi}\\
		\leq&\liminf_{j\rightarrow+\infty}\int_{\{z\in M:-\psi(z)\in V'_j\}}|F|^2e^{-\varphi}	\\	\leq&\liminf_{j\rightarrow+\infty}\frac{G(T_1)}{\int_{T_1}^{+\infty}c(s)e^{-s}ds}\int_{V'_j}e^{-s}\\
		=&\frac{G(T_1)}{\int_{T_1}^{+\infty}c(s)e^{-s}ds}\int_{ E\cap(t_2,t_1]}e^{-s}ds\\
		=&\frac{G(T_1)}{\int_{T_1}^{+\infty}c(s)e^{-s}ds}\int_{t_2}^{t_1}e^{-s}\mathbb I_{E}(s)ds	\end{split}
\end{displaymath}
and
\begin{displaymath}
	\begin{split}
		\int_{\{-t_1\leq\psi<-t_2\}}|F|^2e^{-\varphi}\mathbb I_{E}(-\psi)
		\geq&\liminf_{j\rightarrow+\infty}\int_{\{z\in M:-\psi(z)\in V_j\}}|F|^2e^{-\varphi}	\\	\geq&\liminf_{j\rightarrow+\infty}\frac{G(T_1)}{\int_{T_1}^{+\infty}c(s)e^{-s}ds}\int_{V_j}e^{-s}\\
		=&\frac{G(T_1)}{\int_{T_1}^{+\infty}c(s)e^{-s}ds}\int_{t_2}^{t_1}e^{-s}\mathbb I_{E}(s)ds,
	\end{split}
\end{displaymath}
which implies that $\int_{\{-t_1\leq\psi<-t_2\}}|F|^2e^{-\varphi}\mathbb I_{E}(-\psi)=\frac{G(T_1)}{\int_{T_1}^{+\infty}c(s)e^{-s}ds}\int_{t_2}^{t_1}e^{-s}\mathbb I_{E}(s)ds$. Hence we obtain that equality \eqref{eq:20210412b} holds.

Finally, we prove equality \eqref{eq:20210412a}.

By the definition of $G(t_0;\tilde{c})$, we have $G(t_0;\tilde{c})\leq\int_{\{\psi<-t_0\}}|F|^2e^{-\varphi}\tilde{c}(-\psi)$, then we only consider the case $G(t_0;\tilde{c})<+\infty$.

By the definition of $G(t_0;\tilde{c})$, we can choose a  holomorphic $(n,0)$ form $F_{t_0,\tilde{c}}$ on
$\{\psi<-t_0\}$ satisfying $(F_{t_0,\tilde{c}}-f)\in H^{0}(Z_0,(\mathcal{O}(K_{M})\otimes\mathcal{F})|_{Z_0})$ and $\int_{\{\psi<-t_0\}}|F_{t_0,\tilde{c}}|^{2}e^{-\varphi}\tilde{c}(-\psi)<+\infty$. As $\mathcal H^2(\tilde{c},t_0)\subset\mathcal H^2(c,t_0)$, we have $\int_{\{\psi<-t_0\}}|F_{t_0,\tilde{c}}|^2e^{-\varphi}c(-\psi)<+\infty$. By using Lemma \ref{lem:A}, we obtain that
\begin{displaymath}
\begin{split}
\int_{\{\psi<-t\}}|F_{t_0,\tilde{c}}|^2e^{-\varphi}c(-\psi)=&\int_{\{\psi<-t\}}|F|^2e^{-\varphi}c(-\psi)\\
	&+\int_{\{\psi<-t\}}|F_{t_0,\tilde{c}}-F|^2e^{-\varphi}c(-\psi)	
\end{split}	
\end{displaymath}
for any $t\geq t_0$, then
\begin{equation}
	\label{eq:20210413a}
	\begin{split}
	\int_{\{-t_3\leq\psi<-t_4\}}|F_{t_0,\tilde{c}}|^2e^{-\varphi}c(-\psi)=&\int_{\{-t_3\leq\psi<-t_4\}}|F|^2e^{-\varphi}c(-\psi)\\
	&+\int_{\{-t_3\leq\psi<-t_4\}}|F_{t_0,\tilde{c}}-F|^2e^{-\varphi}c(-\psi)
		\end{split}
\end{equation}
holds for any $t_3>t_4\geq t_0$.
It follows from the dominant convergence theorem, equality \eqref{eq:20210413a}, equality \eqref{eq:20210412g} and $c(t)>0$ for any $t>T$, that
\begin{equation}
	\label{eq:20210413d}
	\int_{\{z\in M:-\psi(z)=t\}}|F_{t_0,\tilde{c}}|^2e^{-\varphi}=\int_{\{z\in M:-\psi(z)=t\}}|F_{t_0,\tilde{c}}-F|^2e^{-\varphi}
\end{equation}
holds for any $t>t_0$.

Choosing any closed interval $[t_4',t_3']\subset (t_0,+\infty)\subset(T,+\infty)$. Note that  $c(t)$ is uniformly continuous  and have positive lower bound and upper bound on $[t_4',t_3']\backslash U_k$, where $\{U_k\}_k$ is a decreasing sequence of open subsets of $(T,+\infty)$, such that $c$ is continuous on $(T,+\infty)\backslash U_k$ and $\lim_{k\rightarrow+\infty}\mu(U_k)=0$. Take $N=\cap_{k=1}^{+\infty}U_k$. Note that
\begin{equation}
	\label{eq:20210413b}
	\begin{split}
		&\int_{\{-t_3'\leq\psi<-t_4'\}}|F_{t_0,\tilde{c}}|^2e^{-\varphi}\\
		=&\lim_{n\rightarrow+\infty}\sum_{i=0}^{n-1}\int_{\{z\in M:-\psi(z)\in I_{n,i}\backslash U_k\}}|F_{t_0,\tilde{c}}|^2e^{-\varphi}+\int_{\{z\in M:-\psi(z)\in (t_4',t_3']\cap U_k\}}|F_{t_0,\tilde{c}}|^2e^{-\varphi}\\
		\leq&\limsup_{n\rightarrow+\infty}\sum_{i=0}^{n-1}\frac{1}{\inf_{I_{n,i}\backslash U_k}c(t)}\int_{\{z\in M:-\psi(z)\in I_{n,i}\backslash U_k\}}|F_{t_0,\tilde{c}}|^2e^{-\varphi}c(-\psi)\\
		&+\int_{\{z\in M:-\psi(z)\in (t_4',t_3']\cap U_k\}}|F_{t_0,\tilde{c}}|^2e^{-\varphi},
	\end{split}
\end{equation}
where $I_{n,i}=(t_4'-(i+1)\alpha_{n},t_3'-i\alpha_{n}]$ and $\alpha_n=\frac{t_3'-t_4'}{n}$.
It following from equality \eqref{eq:20210413a}, \eqref{eq:20210413d}, \eqref{eq:20210412g},  and the dominated theorem that
\begin{equation}
	\label{eq:2106e}
	\begin{split}
		&\int_{\{z\in M:-\psi(z)\in I_{n,i}\backslash U_k\}}|F_{t_0,\tilde{c}}|^2e^{-\varphi}c(-\psi)\\
		=&\int_{\{z\in M:-\psi(z)\in I_{n,i}\backslash U_k)\}}|F|^2e^{-\varphi}c(-\psi)+\int_{\{z\in M:-\psi(z)\in I_{n,i}\backslash U_k)\}}|F_{t_0,\tilde{c}}-F|^2e^{-\varphi}c(-\psi).
	\end{split}
\end{equation}
As $c(t)$ is uniformly continuous  and have positive lower bound and upper bound on $[t_4',t_3']\backslash U_k$, combining equality \eqref{eq:2106e}, we have
\begin{equation}
	\label{eq:2106f}\begin{split}
		&\limsup_{n\rightarrow+\infty}\sum_{i=0}^{n-1}\frac{1}{\inf_{I_{n,i}\backslash U_k}c(t)}\int_{\{z\in M:-\psi(z)\in I_{n,i}\backslash U_k\}}|F_{t_0,\tilde{c}}|^2e^{-\varphi}c(-\psi)\\
		=&\limsup_{n\rightarrow+\infty}\sum_{i=0}^{n-1}\frac{1}{\inf_{I_{n,i}\backslash U_k}c(t)}(\int_{\{z\in M:-\psi(z)\in I_{n,i}\backslash U_k)\}}|F|^2e^{-\varphi}c(-\psi)\\
		&+\int_{\{z\in M:-\psi(z)\in I_{n,i}\backslash U_k)\}}|F_{t_0,\tilde{c}}-F|^2e^{-\varphi}c(-\psi))\\
		\leq&\limsup_{n\rightarrow+\infty}\sum_{i=0}^{n-1}\frac{\sup_{I_{n,i}\backslash U_k}c(t)}{\inf_{I_{n,i}\backslash U_k}c(t)}(\int_{\{z\in M:-\psi(z)\in I_{n,i}\backslash U_k\}}|F|^2e^{-\varphi}\\
		&+\int_{\{z\in M:-\psi(z)\in I_{n,i}\backslash U_k\}}|F_{t_0,\tilde{c}}-F|^2e^{-\varphi})\\
		=&\int_{\{z\in M:-\psi(z)\in (t_4',t_3']\backslash U_k\}}|F|^2e^{-\varphi}+\int_{\{z\in M:-\psi(z)\in (t_4',t_3']\backslash U_k\}}|F_{t_0,\tilde{c}}-F|^2e^{-\varphi}.
	\end{split}
\end{equation}
It follows from inequality \eqref{eq:20210413b} and \eqref{eq:2106f}, we obtain that
\begin{equation}
	\label{eq:2106g}
	\begin{split}
		&\int_{\{-t_3'\leq\psi<-t_4'\}}|F_{t_0,\tilde{c}}|^2e^{-\varphi}\\
		\leq&\int_{\{z\in M:-\psi(z)\in (t_4',t_3']\backslash U_k\}}|F|^2e^{-\varphi}+\int_{\{z\in M:-\psi(z)\in (t_4',t_3']\backslash U_k\}}|F_{t_0,\tilde{c}}-F|^2e^{-\varphi}\\&+\int_{\{z\in M:-\psi(z)\in (t_4',t_3']\cap U_k\}}|F_{t_0,\tilde{c}}|^2e^{-\varphi}.
	\end{split}
\end{equation}
It follows from $F_{t_0,\tilde{c}}\in\mathcal{H}^2(c,t_0)$ that $\int_{\{-t_3'\leq\psi<-t_4'\}}|F_{t_0,\tilde{c}}|^2e^{-\varphi}<+\infty$. Let $k\rightarrow+\infty$, following from equality \eqref{eq:20210412g}, inequality \eqref{eq:2106g} and the dominated theorem,  we have
\begin{equation}
\label{eq:20210414b}
	\begin{split}
		\int_{\{-t_3'\leq\psi<-t_4'\}}|F_{t_0,\tilde{c}}|^2e^{-\varphi}\leq&\int_{\{-t_3'\leq\psi<-t_4'\}}|F|^2e^{-\varphi}\\
		&+\int_{\{z\in M:-\psi(z)\in (t_4',t_3']\backslash N\}}|F_{t_0,\tilde{c}}-F|^2e^{-\varphi}\\&+\int_{\{z\in M:-\psi(z)\in (t_4',t_3']\cap N\}}|F_{t_0,\tilde{c}}|^2e^{-\varphi},
	\end{split}
\end{equation}
 Following from a similar discussion, we can obtain that
\begin{displaymath}
	\begin{split}
		\int_{\{-t_3'\leq\psi<-t_4'\}}|F_{t_0,\tilde{c}}|^2e^{-\varphi}\geq&\int_{\{-t_3'\leq\psi<-t_4'\}}|F|^2e^{-\varphi}\\
		&+\int_{\{z\in M:-\psi(z)\in (t_4',t_3']\backslash N\}}|F_{t_0,\tilde{c}}-F|^2e^{-\varphi}\\&+\int_{\{z\in M:-\psi(z)\in (t_4',t_3']\cap N\}}|F_{t_0,\tilde{c}}|^2e^{-\varphi},
	\end{split}
\end{displaymath}
then combining inequality \eqref{eq:20210414b} we have
\begin{equation}
	\label{eq:20210413c}
	\begin{split}
		\int_{\{-t_3'\leq\psi<-t_4'\}}|F_{t_0,\tilde{c}}|^2e^{-\varphi}=&\int_{\{-t_3'\leq\psi<-t_4'\}}|F|^2e^{-\varphi}\\
		&+\int_{\{z\in M:-\psi(z)\in (t_4',t_3']\backslash N\}}|F_{t_0,\tilde{c}}-F|^2e^{-\varphi}\\&+\int_{\{z\in M:-\psi(z)\in (t_4',t_3']\cap N\}}|F_{t_0,\tilde{c}}|^2e^{-\varphi}.
	\end{split}\end{equation}
Using equality \eqref{eq:20210412g}, \eqref{eq:20210413d}, \eqref{eq:20210413c} and  Levi's Theorem, we have
\begin{equation}
\label{eq:2021525a}
	\begin{split}
		\int_{\{z\in M:-\psi(z)\in U\}}|F_{t_0,\tilde{c}}|^2e^{-\varphi}=&\int_{\{z\in M:-\psi(z)\in U\}}|F|^2e^{-\varphi}\\
		&+\int_{\{z\in M:-\psi(z)\in  U\backslash N\}}|F_{t_0,\tilde{c}}-F|^2e^{-\varphi}\\&+\int_{\{z\in M:-\psi(z)\in U\cap N\}}|F_{t_0,\tilde{c}}|^2e^{-\varphi}
	\end{split}
\end{equation}
holds for any open set $U\subset\subset(t_0,+\infty)$, and
\begin{equation}
\label{eq:2021525b}
	\begin{split}
		\int_{\{z\in M:-\psi(z)\in V\}}|F_{t_0,\tilde{c}}|^2e^{-\varphi}=&\int_{\{z\in M:-\psi(z)\in V\}}|F|^2e^{-\varphi}\\
		&+\int_{\{z\in M:-\psi(z)\in  V\backslash N\}}|F_{t_0,\tilde{c}}-F|^2e^{-\varphi}\\&+\int_{\{z\in M:-\psi(z)\in V\cap N\}}|F_{t_0,\tilde{c}}|^2e^{-\varphi}
	\end{split}
\end{equation}
holds for any compact set $V\subset(t_0,+\infty)$. For any  measurable set $E\subset\subset(t_0,+\infty)$, there exists a sequence of compact sets $\{V_l\}$, such that $V_l\subset V_{l+1}\subset E$ for any $l$ and $\lim_{l\rightarrow}\mu(V_l)=\mu(E)$, hence
\begin{equation}
\label{eq:2021525c}
	\begin{split}
		\int_{\{\psi<-t_0\}}|F_{t_0,\tilde{c}}|^2e^{-\varphi}\mathbb I_{E}(-\psi)\geq&\lim_{l\rightarrow+\infty}\int_{\{\psi<-t_0\}}|F_{t_0,\tilde{c}}|^2e^{-\varphi}\mathbb I_{V_{j}}(-\psi)
		\\\geq&\lim_{j\rightarrow+\infty}\int_{\{\psi<-t_0\}}|F|^2e^{-\varphi}\mathbb I_{V_j}(-\psi)
		\\=&\int_{\{\psi<-t_0\}}|F|^2e^{-\varphi}\mathbb I_{E}(-\psi).
	\end{split}
\end{equation}

It is clear that for any $t>t_0$, there exists a sequence of functions $\{\sum_{j=1}^{n_i}\mathbb I_{E_{ij}}\}_{i=1}^{+\infty}$ defined on $(t,+\infty)$, satisfying $E_{ij}\subset\subset(t,+\infty)$, $\sum_{j=1}^{n_{i+1}}\mathbb I_{E_{i+1j}}(s)\geq\sum_{j=1}^{n_i}\mathbb I_{E_{ij}}(s)$, and $\lim_{i\rightarrow+\infty}\sum_{j=1}^{n_i}\mathbb I_{E_{ij}}(s)=\tilde{c}(s)$  for any $s>t$.
Combining Levi's Theorem and inequality \eqref{eq:2021525c}, we have
\begin{equation}
\label{eq:2021525d}
		\int_{\{\psi<-t_0\}}|F_{t_0,\tilde{c}}|^2e^{-\varphi}\tilde{c}(-\psi)\geq\int_{\{\psi<-t_0\}}|F|^2e^{-\varphi}\tilde{c}(-\psi).
\end{equation}
By the definition of $G(t_0,\tilde{c})$, we have $G(t_0,\tilde{c})=\int_{\{\psi<-t_0\}}|F|^2e^{-\varphi}\tilde{c}(-\psi)$. Then equality \eqref{eq:20210412a} holds.

\section{Proofs of Theorem \ref{thm:L2}, Theorem \ref{thm:ll1}, Corollary \ref{c:L2b} and Corollary \ref{thm:ll2}}
In this section, we prove Theorem \ref{thm:L2}, Theorem \ref{thm:ll1}, Corollary \ref{c:L2b} and Corollary \ref{thm:ll2}.

\subsection{Proof of Theorem \ref{thm:L2}}

The following remark  shows that it suffices to consider Theorem \ref{thm:L2} for the case $c(t)$ has a positive lower bound and upper bound on $(t',+\infty)$ for any $t'>T$.

\begin{Remark}\label{r:L2c}
Take $c_j$ is a positive measurable function on $(T,+\infty)$, such that $c_{j}(t)=c(t)$ when $t<T+j$, $c_j(t)=\min\{c(T+j),\frac{1}{j}\}$ when $t\geq T+j$. It is clear that $c_j(t)e^{-t}$ is decreasing with respect to $t$, and $\int_{T}^{+\infty}c_j(t)e^{-t}<+\infty$. As
$$\lim_{j\rightarrow+\infty}\int_{T+j}^{+\infty}c_n(t)e^{-t}=0,$$
we have
$$\lim_{j\rightarrow+\infty}\int_{T}^{+\infty}c_j(t)e^{-t}=\int_{T}^{+\infty}c(t)e^{-t}.$$

	If Theorem \ref{thm:L2} holds in this case, then there exists a holomorphic $(n,0)$ form $F_j$ on $M$ such that $F_j|_S=f$ and
$$\int_{M}|F_j|^2e^{-\varphi}c_j(-\psi)\leq\left(\int_T^{+\infty}c_j(t)e^{-t}dt\right)\sum_{k=1}^{n}\frac{\pi^k}{k!}\int_{S_{n-k}}|f|^2e^{-\varphi}dV_{M}[\psi].$$
Note that $\psi$ has locally lower bound on $M\backslash\psi^{-1}(-\infty)$ and $\psi^{-1}(-\infty)$ is a closed subset of some analytic subset of $M$, it follows from Lemma \ref{l:converge} that there exists a subsequence of $\{F_j\}$, denoted still by $\{F_j\}$, which is uniformly convergent to a holomorphic $(n,0)$ form $F$ on any compact subset of $M$ and
\begin{displaymath}
	\begin{split}
		\int_{M}|F|^2e^{-\varphi}c(-\psi)&\leq\lim_{j\rightarrow+\infty}\left(\int_T^{+\infty}c_j(t)e^{-t}dt\right)\sum_{k=1}^{n}\frac{\pi^k}{k!}\int_{S_{n-k}}|f|^2e^{-\varphi}dV_{M}[\psi]\\
		&=\left(\int_T^{+\infty}c(t)e^{-t}dt\right)\sum_{k=1}^{n}\frac{\pi^k}{k!}\int_{S_{n-k}}|f|^2e^{-\varphi}dV_{M}[\psi].	\end{split}
\end{displaymath}
 Since $F_j|_S=f$ for any $j$, we have $F|_S=f$.

\end{Remark}

By the definition of condition $(ab)$, $\liminf_{t\rightarrow+\infty}c(t)>0$, it suffices to prove the case that $M$ is Stein manifold and $S_{reg}=S$. Without loss of generality, we can assume that $supp(\mathcal{O}_M\slash\mathcal{I}(\psi))=S_{reg}$ (if $supp(\mathcal{O}_M\slash\mathcal{I}(\psi))\not=S_{reg}$, there exists a analytic subset $X$ of $M$ such that $(M,X)$ satisfies condition $(ab)$ and $supp(\mathcal{O}_M\slash\mathcal{I}(\psi))\backslash S_{reg}\in X$).

Since $M$ is Stein, we can find a sequence of Stein manifolds $\{D_m\}_{m=1}^{+\infty}$ satisfying $D_m\subset\subset D_{m+1}$ for any $m$ and $\cup_{m=1}^{+\infty}D_m=M$, and there is a holomorphic $(n,0)$ form $\tilde F$ on $M$ such that $\tilde{F}|_S=f$.

Note that $\int_{D_m}|\tilde F|^2<+\infty$ for any $m$ and
$$\int_{D_m}\mathbb{I}_{\{-t_0-1<\psi<-t_0\}}|\tilde{F}|^2e^{-\varphi-\psi}<+\infty$$
for any $m$ and $t_0>T$.
Using Lemma \ref{lem:GZ_sharp}, for any $D_m$ and $t_0>T$, there exists a holomorphic $(n,0)$ form $F_{m,t_0}$ on $D_m$, such that
\begin{equation}
\label{eq:210622a}
\begin{split}
&\int_{D_m}|F_{m,t_0}-(1-b_{t_0,1}(\psi))\tilde{F}|^{2}e^{-\varphi-\psi+v_{t_0,1}(\psi)}c(-v_{t_0,1}(\psi))\\
\leq& \left(\int_{T}^{t_{0}+1}c(t)e^{-t}dt\right) \int_{D_m}\mathbb{I}_{\{-t_0-1<\psi<-t_0\}}|\tilde{F}|^2e^{-\varphi-\psi}
\end{split}
\end{equation}
where
$b_{t_0,1}(t)=\int_{-\infty}^{t}\mathbb{I}_{\{-t_{0}-1< s<-t_{0}\}}ds$,
$v_{t_0,1}(t)=\int_{0}^{t}b_{t_0,1}(s)ds$. Note that $e^{-\psi}$ is not locally integrable along $S$, and $b_{t_0,1}(t)=0$ for  large enough $t$, then $(F_{m,t_0}-(1-b_{t_0,1}(\psi))\tilde{F})|_{D_m\cap S}=0$, and therefore $F_{m,t_0}|_{D_m\cap S}=f$.

Note that $v_{t_0,1}(\psi)\geq\psi$ and $c(t)e^{-t}$ is decreasing, then the inequality \eqref{eq:210622a} becomes
\begin{equation}
\label{eq:210622c}
\begin{split}
&\int_{D_m}|F_{m,t_0}-(1-b_{t_0,1}(\psi))\tilde{F}|^{2}e^{-\varphi}c(-\psi)\\
\leq& \left(\int_{T}^{t_{0}+1}c(t)e^{-t}dt\right) \int_{D_m}\mathbb{I}_{\{-t_0-1<\psi<-t_0\}}|\tilde{F}|^2e^{-\varphi-\psi}.
\end{split}
\end{equation}

As $\sum_{k=1}^{n}\frac{\pi^k}{k!}\int_{S_{n-k}}\frac{|f|^2}{dV_{M}}e^{-\varphi}dV_{M}[\psi]<+\infty$, by definition of $dV_M[\psi]$ and $supp(\mathcal{O}_M\slash\mathcal{I}(\psi))=S_{reg}$, we have
\begin{equation}
	\label{eq:210622b}
	\begin{split}
	&\limsup_{t_0\rightarrow+\infty}\left(\int_{T}^{t_{0}+1}c(t)e^{-t}dt\right) \int_{D_m}\mathbb{I}_{\{-t_0-1<\psi<-t_0\}}|\tilde{F}|^2e^{-\varphi-\psi}\\
	\leq&\left(\int_{T}^{+\infty}c(t)e^{-t}dt\right) \sum_{k=1}^{n}\frac{\pi^k}{k!}\int_{S_{n-k}\cap D_m}\frac{|f|^2}{dV_{M}}e^{-\varphi}dV_{M}[\psi]\\
	<&+\infty.	
	\end{split}
\end{equation}
Note that $e^{-\varphi}c(-\psi)$ has a positive lower bound on $D_m$, then it follows from inequality \eqref{eq:210622c} and \eqref{eq:210622b} that
$\sup_{t_0}\int_{D_m}|F_{m,t_0}-(1-b_{t_0,1}(\psi))\tilde{F}|^{2}<+\infty.$

Combining with
\begin{equation}
\label{eq:210622d}
\sup_{t_0}\int_{D_m}|(1-b_{t_0,1}(\psi))\tilde{F}|^{2}\leq\sup_{t_0}\int_{D_m}\mathbb{I}_{\{\psi<-t_{0}\}}|\tilde F|^{2}<+\infty,
\end{equation}
one can obtain that $\sup_{t_0}\int_{D_m}|F_{m,t_0}|^{2}<+\infty$,
which implies that
there exists a subsequence of $\{F_{m,t_0}\}_{t_0\rightarrow+\infty}$ (also denoted by $\{F_{m,t_0}\}_{t_0\rightarrow+\infty}$)
compactly convergent to a holomorphic (n,0) form on $D_m$ denoted by $F_{m}$. Then it follows from inequality \eqref{eq:210622c}, inequality \eqref{eq:210622b} and the Fatou's Lemma that
\begin{equation}
\label{eq:210622f}
\begin{split}
\int_{D_m}|F_{m}|^{2}e^{-\varphi}c(-\psi)
=&\int_{D_m}\liminf_{t_0\rightarrow+\infty}|F_{m,t_0}-(1-b_{t_0,1}(\psi))\tilde{F}|^{2}e^{-\varphi}c(-\psi)\\
\leq&\liminf_{t_0\rightarrow+\infty}\int_{D_m}|F_{m,t_0}-(1-b_{t_0,1}(\psi))\tilde{F}|^{2}e^{-\varphi}c(-\psi)\\
\leq&\limsup_{t_0\rightarrow+\infty}\left(\int_{T}^{t_{0}+1}c(t)e^{-t}dt \right)\int_{D_m}\mathbb{I}_{\{-t_0-1<\psi<-t_0\}}|\tilde{F}|^2e^{-\varphi-\psi}\\
	\leq&\left(\int_{T}^{+\infty}c(t)e^{-t}dt\right) \sum_{k=1}^{n}\frac{\pi^k}{k!}\int_{S_{n-k}\cap D_m}\frac{|f|^2}{dV_{M}}e^{-\varphi}dV_{M}[\psi]\\
	\leq&\left(\int_{T}^{+\infty}c(t)e^{-t}dt\right) \sum_{k=1}^{n}\frac{\pi^k}{k!}\int_{S_{n-k}}\frac{|f|^2}{dV_{M}}e^{-\varphi}dV_{M}[\psi],
\end{split}
\end{equation}
and $F_m|_{D_m\cap S}=f$. Inequality \eqref{eq:210622f} implies that
$$\int_{D_m}|F_{m'}|^{2}e^{-\varphi}c(-\psi)\leq \left(\int_{T}^{+\infty}c(t)e^{-t}dt\right) \sum_{k=1}^{n}\frac{\pi^k}{k!}\int_{S_{n-k}}\frac{|f|^2}{dV_{M}}e^{-\varphi}dV_{M}[\psi]$$
holds for any $m'\geq m$.
 As $e^{-\varphi}c(-\psi)$ has a positive lower bound on any $D_m$, by the diagonal method, we obtain a subsequence of $\{F_m\}$, denoted also by $\{F_m\}$, which is uniformly convergent to a holomorphic $(n,0)$ form $F$ on $M$ on any compact subset of $M$ satisfying that $F|_{S}=f$ and
 $$\int_{M}|F|^{2}e^{-\varphi}c(-\psi)\leq \left(\int_{T}^{+\infty}c(t)e^{-t}dt\right) \sum_{k=1}^{n}\frac{\pi^k}{k!}\int_{S_{n-k}}\frac{|f|^2}{dV_{M}}e^{-\varphi}dV_{M}[\psi].$$
Thus Theorem \ref{thm:L2} holds.

\subsection{Proof of Theorem \ref{thm:ll1}}
If $\sum_{k=1}^{n}\frac{\pi^k}{k!}\int_{S_{n-k}}\frac{|f|^2}{dV_{M}}e^{-\varphi}dV_{M}[\psi]=0$, it is clear that $F\equiv0$ satisfying all requirements in Theorem \ref{thm:ll1}. In the following part, we consider the case $\sum_{k=1}^{n}\frac{\pi^k}{k!}\int_{S_{n-k}}\frac{|f|^2}{dV_{M}}e^{-\varphi}dV_{M}[\psi]\in(0,+\infty)$.

Using Theorem \ref{thm:L2}, for any $t>T$, there exists a holomorphic $(n,0)$ form $F_t$ on $\{\psi<-t\}$ such that $F_t|_S=f$ and
$$\int_{\{\psi<-t\}}|F_t|^2e^{-\varphi}c(-\psi)\leq\left(\int_t^{+\infty}c(l)e^{-l}dl\right)\sum_{k=1}^{n}\frac{\pi^k}{k!}\int_{S_{n-k}}\frac{|f|^2}{dV_{M}}e^{-\varphi}dV_{M}[\psi].$$
Then we have inequality
\begin{equation}
	\label{eq:210622j}
	\frac{G(t)}{\int_t^{+\infty}c(l)e^{-l}dl}\leq	\frac{G(T)}{\int_T^{+\infty}c(t)e^{-t}dt}\end{equation}
holds for any $t>T$.
As $(M,S)$ satisfies condition $(ab)$, and $\psi\in A(S)$, Theorem \ref{thm:general_concave} tells us $G(\hat{h}^{-1}(r))$ is concave with respect to $r$.
Combining inequality \eqref{eq:210622j} and Corollary \ref{thm:linear}, we obtain that $G(\hat{h}^{-1}(r))$ is linear with respect to $r$. Note that $\frac{G(T)}{\int_{T}^{+\infty}c(t)e^{-t}dt}=\|f\|_M$, Corollary \ref{thm:1} shows that the rest results of Theorem \ref{thm:ll1} hold.

\subsection{Proof of Corollary \ref{c:L2b}}In this section, we prove Corollary \ref{c:L2b} by using Theorem \ref{thm:L2}.

Since $M$ is Stein, we can find a sequence of Stein manifolds $\{D_l\}_{l=1}^{+\infty}$ satisfying $D_l\subset\subset D_{l+1}$ for any $l$ and $\cup_{l=1}^{+\infty}D_l=M$. Since  $\psi_2$ and $\psi_2+\varphi$ are plurisubharmonic functions on $M$, there exist smooth plurisubharmonic functions $\Psi_m$ and $\Phi_{m'}$, which are decreasingly convergent to $\psi_2$ and $\psi_2+\varphi$, respectively.

Fixed $D_l$, we can choose large enough $m$ such that $\Psi_{m}+\psi_1<-T$ on $D_l$.
Note that $dV_M[\Psi_m+\psi_1]=e^{-\Psi_m}dV_M[\psi_1]$ and
$$\sum_{k=1}^{n}\frac{\pi^k}{k!}\int_{S_{n-k}}\frac{|f|^2}{dV_{M}}e^{-\Phi_{m'}}dV_{M}[\psi_1]\leq\sum_{k=1}^{n}\frac{\pi^k}{k!}\int_{S_{n-k}}\frac{|f|^2}{dV_{M}}e^{-\varphi-\psi_2}dV_{M}[\psi_1]<+\infty,$$
Using Theorem \ref{thm:L2}, for any $D_l$, there exists a holomorphic $(n,0)$ form $F_{l,m'}$ on $D_l$, satisfying $F_{l,m'}|_S=f$ and
\begin{equation}
\label{eq:210622g}
\int_{D_l}|F_{l,m'}|^{2}e^{-\Phi_{m'}+\Psi_m}c(-\Psi_m-\psi_1)\leq \left(\int_{T}^{+\infty}c(t)e^{-t}dt\right) \sum_{k=1}^{n}\frac{\pi^k}{k!}\int_{S_{n-k}}\frac{|f|^2}{dV_{M}}e^{-\Phi_{m'}}dV_{M}[\psi_1].
\end{equation}
As $e^{-\Phi_{m'}+\Psi_m}c(-\Psi_m-\psi_1)$ has locally uniformly positive lower bound for any $m'$ on $D_m\backslash Z$, where $Z$ is some analytic subset of $M$, it follows from Lemma \ref{l:converge} that there exists a subsequence of $\{F_{l,m'}\}_{m'\rightarrow+\infty}$, also denoted by  $\{F_{l,m'}\}_{m'\rightarrow+\infty}$, which satisfies that $\{F_{l,m'}\}_{m'\rightarrow+\infty}$ is uniformly convergent to a holomorphic $(n,0)$ form $F_{l}$ on any compact subset of $D_l$. Following from inequality \eqref{eq:210622g}, Fatou's Lemma and $c(t)e^{-t}$ is decreasing, we have
\begin{equation}
	\label{eq:210622h}
	\begin{split}
	\int_{D_l}|F_{l}|^{2}e^{-\varphi}c(-\psi_2-\psi_1)
\leq&\int_{D_l}|F_{l}|^{2}e^{-\varphi-\psi_2+\Psi_m}c(-\Psi_m-\psi_1)\\
=&\int_{D_l}\lim_{m'\rightarrow+\infty}|F_{l}|^{2}e^{-\Phi_{m'}+\Psi_m}c(-\Psi_m-\psi_1)\\
\leq&\liminf_{m'\rightarrow+\infty}\int_{D_l}|F_{l}|^{2}e^{-\Phi_{m'}+\Psi_m}c(-\Psi_m-\psi_1) \\
\leq&\left(\int_{T}^{+\infty}c(t)e^{-t}dt\right) \sum_{k=1}^{n}\frac{\pi^k}{k!}\int_{S_{n-k}}\frac{|f|^2}{dV_{M}}e^{-\varphi-\psi_2}dV_{M}[\psi_1].	
	\end{split}\end{equation}
Note that $e^{-\varphi}c(-\psi_2-\psi_1)$ has locally a positive lower bound  on $M\backslash Z$, where $Z$ is some analytic subset of $M$, by using Lemma \ref{l:converge} and the diagonal method, we obtain that there exists a subsequence of $\{F_{l}\}$, also denoted by  $\{F_{l}\}$, which satisfies that $\{F_{l}\}$ is uniformly convergent to a holomorphic $(n,0)$ form $F$ on $M$ on any compact subset of $M$. Following from inequality \eqref{eq:210622h} and  Fatou's Lemma, we have
\begin{equation}
	\label{eq:210622i}
	\begin{split}
	\int_{M}|F|^{2}e^{-\varphi}c(-\psi_2-\psi_1)
=&\int_{M}\lim_{l\rightarrow+\infty}\mathbb{I}_{D_l}|F_{l}|^{2}e^{-\varphi}c(-\psi_2-\psi_1)\\
\leq&\liminf_{l\rightarrow+\infty}\int_{D_l}|F_{l}|^{2}e^{-\varphi}c(-\psi_2-\psi_1) \\
\leq&\left(\int_{T}^{+\infty}c(t)e^{-t}dt\right) \sum_{k=1}^{n}\frac{\pi^k}{k!}\int_{S_{n-k}}\frac{|f|^2}{dV_{M}}e^{-\varphi-\psi_2}dV_{M}[\psi_1].	
	\end{split}\end{equation}
This proves Corollary \ref{c:L2b}.

\subsection{Proof of Corollary \ref{thm:ll2}} 	If $\|f\|_S^*=0$, it is clear that $F\equiv0$ satisfying all requirements in Corollary \ref{thm:ll2}. In the following part, we consider the case $\|f\|_S^*\in(0,+\infty)$.

Using Corollary \ref{c:L2b}, for any $t>T$, there exists a holomorphic $(n,0)$ form $F_t$ on $\{\psi<-t\}$ such that $F_t|_S=f$ and
$$\int_{\{\psi<-t\}}|F_t|^2e^{-\varphi}c(-\psi)\leq\left(\int_t^{+\infty}c(l)e^{-l}dl\right)\|f\|_S^*.$$
Then we have inequality
\begin{equation}
	\label{eq:210729a}
	\frac{G(t)}{\int_t^{+\infty}c(l)e^{-l}dl}\leq	\frac{G(T)}{\int_T^{+\infty}c(t)e^{-t}dt}\end{equation}
holds for any $t>T$.
 Theorem \ref{thm:general_concave} tells us $G(\hat{h}^{-1}(r))$ is concave with respect to $r$.
Combining inequality \eqref{eq:210729a} and Corollary \ref{thm:linear}, we obtain that $G(\hat{h}^{-1}(r))$ is linear with respect to $r$. Note that $\frac{G(T)}{\int_{T}^{+\infty}c(t)e^{-t}dt}=\|f\|^*_M$, Corollary \ref{thm:1} shows that the rest results of Theorem \ref{thm:ll2} hold.

\section{Proofs of Theorem \ref{thm:n1}, Theorem \ref{thm:n4} and Theorem \ref{thm:n2}}
In this section, we prove Theorem \ref{thm:n1}, Theorem \ref{thm:n4} and Theorem \ref{thm:n2}.
\subsection{Proof of Theorem \ref{thm:n1}}
We prove the theorem by comparing $G(t;\varphi)$ and $G(t;\tilde\varphi)$. Let us assume that $G(\hat{h}^{-1}(r);\varphi)$ is linear with respect to $r$ to get a contradiction.

As $G(\hat{h}^{-1}(r);\varphi)$ is linear with respect to $r$, it follows from Corollary \ref{thm:1} that there exists a holomorphic $(n,0)$ form $F$ on $M$ such that $(F-f)\in H^0(Z_0,(\mathcal{O}(K_{M})\otimes\mathcal{F})|_{Z_0})$ and $\forall t\geq T$  equality
$$G(t;\varphi)=\int_{\{\psi<-t\}}|F|^2e^{-\varphi}c(-\psi)$$
holds.
 As $\tilde\varphi+\psi$ is plurisubharmonic and $\tilde\varphi-\varphi$ is bounded on $M$, it follows from Theorem \ref{thm:general_concave} that $G(\hat{h}^{-1}(r);\tilde\varphi)$ is concave with respect to $r$.

 As $\tilde\varphi+\psi\geq\varphi+\psi$, $\tilde\varphi+\psi\not=\varphi+\psi$ and both of them are plurisubharmonic functions on $M$, then there exists a subset $U$ of $M$ such that
  $e^{-\tilde\varphi}<e^{-\varphi}$ on a subset $U$ and $\mu(U)>0$, where $\mu$ is Lebesgue measure on $M$. As $F\not\equiv0$, inequality
\begin{equation}
	\label{eq:210615a}
	\frac{G(T_0;\tilde\varphi)}{\int_{T_0}^{+\infty}c(s)e^{-s}ds}\leq\frac{\int_{\{\psi<-T_0\}}|F|^2e^{-\tilde\varphi}c(-\psi)}{\int_{T_0}^{+\infty}c(s)e^{-s}ds}<\frac{G(T_0;\varphi)}{\int_{T_0}^{+\infty}c(s)e^{-s}ds}
\end{equation}
holds for some $T_0>T$.
For $t>T$, there exists a holomorphic $(n,0)$ form $F_t$ on $\{\psi<-t\}$ such that $(F_t-f)\in H^0(Z_0,(\mathcal{O}(K_{M})\otimes\mathcal{F})|_{Z_0})$ and
$$G(t;\tilde\varphi)=\int_{\{\psi<-t\}}|F_t|^2e^{-\tilde\varphi}c(-\psi)<+\infty.$$

As $\tilde\varphi-\varphi$ is bounded on $M$, we have $\int_{\{\psi<-t\}}|F_t|^2e^{-\varphi}c(-\psi)<+\infty.$ It follows from Lemma \ref{lem:A} that
\begin{equation}
	\label{eq:210615b}
	\begin{split}
	G(t_1;\tilde\varphi)-G(t_2;\tilde\varphi)&\geq \int_{\{-t_2\leq\psi<-t_1\}}|F_{t_1}|^2e^{-\tilde\varphi}c(-\psi)\\
		&\geq \left(\inf_{\{-t_2\leq\psi\}}e^{\varphi-\tilde\varphi}\right)\int_{\{-t_2\leq\psi<-t_1\}}|F_{t_1}|^2e^{-\varphi}c(-\psi) \\
		&\geq \left(\inf_{\{-t_2\leq\psi\}}e^{\varphi-\tilde\varphi}\right)\int_{\{-t_2\leq\psi<-t_1\}}|F|^2e^{-\varphi}c(-\psi)
	\end{split}
\end{equation}
holds for $T<t_1<t_2<+\infty$.
As  $\lim_{t\rightarrow T+0}\sup_{z\in\{\psi\geq-t\}}((\tilde\varphi-\varphi)(z))=0$, it follows from inequality \eqref{eq:210615a} and \eqref{eq:210615b} that
\begin{displaymath}
	\begin{split}
	\liminf_{t_2\rightarrow T+0}\frac{G(t_1;\tilde\varphi)-G(t_2;\tilde\varphi)}{\int_{t_1}^{t_2}c(s)e^{-s}ds}
	\geq& \liminf_{t_2\rightarrow T+0}\left(\inf_{z\in\{-t_2\leq\psi\}}e^{\varphi-\tilde\varphi}\right) \frac{\int_{\{-t_2\leq\psi<-t_1\}}|F|^2e^{-\varphi}c(-\psi)}{\int_{t_1}^{t_2}c(s)e^{-s}ds}	\\
	=&\frac{G(T_0;\varphi)}{\int_{T_0}^{+\infty}c(s)e^{-s}ds} \\
	>&\frac{G(T_0;\tilde\varphi)}{\int_{T_0}^{+\infty}c(s)e^{-s}ds},
	\end{split}
\end{displaymath}
which  contradicts the concavity of $G(\hat{h}^{-1}(r);\tilde\varphi)$. Thus the assumption doesn't hold, i.e., $G(\hat{h}^{-1}(r);\varphi)$ is not linear with respect to $r$.

Especially, if $\varphi+\psi$ is strictly plurisubharmonic at $z_1\in M$, we can construct a $\tilde\varphi\geq\varphi$ satisfying the three statements in Theorem \ref{thm:n1}, which implies $G(\hat{h}^{-1}(r);\varphi)$ is not linear with respect to $r$. In fact, there is a small open neighborhood $(U,w)$ of $z_1$ and $w=(w_1,...,w_n)$ is the local coordinate on $U$ such that $i\partial\bar\partial(\varphi+\psi)>\epsilon\omega$ for some $\epsilon>0$, where $\omega=i\sum_{j=1}^{n}dw_j\wedge d\bar w_j$ on $U$. Let $\rho$ be a smooth nonnegative function on $M$ satisfying $\rho\not\equiv0$ and $supp\rho\subset\subset U$. It is clear that there exists a positive number $\delta$ such that
$$i\partial\bar\partial(\varphi+\psi+\delta\rho)>0$$
holds on $U$. Let $\tilde\varphi=\varphi+\delta\rho$, it is clear that $\tilde\varphi$ satisfies the three statements in Theorem \ref{thm:n1}. Thus we complete the proof of Theorem \ref{thm:n1}.

\subsection{Proof of Theorem \ref{thm:n4}}

Let $\tilde\varphi=\varphi+\psi-\tilde\psi$, then $\tilde\varphi+\tilde\psi=\varphi+\psi$ is a plurisubharmonic function on $M$. We prove the theorem by comparing $G(t;\varphi,\psi)$ and $G(t;\tilde\varphi,\tilde\psi)$. Let us assume that $G(\hat{h}^{-1}(r);\varphi,\psi)$ is linear with respect to $r$ to get a contradiction.

 Since $G(T;\varphi,\psi)\in(0,+\infty)$, $G(\hat{h}^{-1}(r))$ is linear and Corollary \ref{thm:1}, we have $\int_{T}^{+\infty}c(t)e^{-t}dt<+\infty$.
As $G(\hat{h}^{-1}(r);\varphi,\psi)$ is linear with respect to $r$, it follows from Corollary \ref{thm:1}, Remark \ref{r:c} and Lemma \ref{l:c'} that we can assume $c(t)e^{-t}$ is strictly decreasing on $(T,+\infty)$ and $c(t)$  is  increasing on $(a,+\infty)$ for some $a>T$.

Using Corollary \ref{thm:1}, there exists a holomorphic $(n,0)$ form $F$ on $M$, such that $(F-f)\in H^0(Z_0,(\mathcal{O}(K_{M})\otimes\mathcal{F})|_{Z_0})$ and $\forall t\geq T$  equality
$$G(t;\varphi,\psi)=\int_{\{\psi<-t\}}|F|^2e^{-\varphi}c(-\psi)$$
holds.

 Since $\lim_{t\rightarrow+\infty}\sup_{\{\psi<-t\}}(\tilde\psi-\psi)=0$, we have $Z_0\subset\{\psi=-\infty\}=\{\tilde\psi=-\infty\}$. As $c(t)e^{-t}$ is decreasing and $\tilde\psi\geq\psi$, we have $e^{-\varphi}c(-\psi)=e^{-\varphi-\psi}e^{\psi}c(-\psi)\leq e^{-\tilde\varphi-\tilde\psi}e^{\tilde\psi}c(-\tilde\psi)=e^{-\tilde\varphi}c(-\tilde\psi)$. It follows from Theorem \ref{thm:general_concave} that $G(\hat{h}^{-1}(r);\tilde\varphi,\tilde\psi)$ is concave with respect to $r$.

We claim that
 \begin{equation}
 	\label{eq:210618a}\lim_{t\rightarrow T+0}\frac{G(t;\tilde\varphi,\tilde\psi)}{\int_{t}^{+\infty}c(s)e^{-s}ds}>\frac{G(T;\varphi,\psi)}{\int_{T}^{+\infty}c(s)e^{-s}ds}. \end{equation}
 In fact, we just need to prove the inequality for the case $G(T;\tilde\varphi,\tilde\psi)<+\infty$.
  It follows from Lemma \ref{lem:A} that there exists a holomorphic $(n,0)$ form $F_T$ on $M$ such that $(F_T-f)\in H^0(Z_0,(\mathcal{O}(K_{M})\otimes\mathcal{F})|_{Z_0})$ and
$$G(T;\tilde\varphi,\tilde\psi)=\int_M|F_T|^2e^{-\varphi}c(-\tilde\psi)\in(0,+\infty),$$
where $G(T;\tilde\varphi,\tilde\psi)>0$ follows from $G(T;\varphi,\psi)>0$.
   As $\tilde\psi\geq\psi$, $\tilde\psi\not=\psi$ and both of them are plurisubharmonic functions on $M$, then there exists a subset $U$ of $M$ such that
  $\tilde\psi>\psi$ on a subset $U$ and $\mu(U)>0$, where $\mu$ is Lebesgue measure on $M$. As  $F_T\not\equiv0$ and $c(t)e^{-t}$ is strictly decreasing on $(T,+\infty)$, we have
\begin{displaymath}
	\begin{split}
	\frac{G(T;\tilde\varphi,\tilde\psi)}{\int_{T}^{+\infty}c(s)e^{-s}ds}=&\frac{\int_{M}|F_T|^2e^{-\tilde\varphi}c(-\tilde\psi)}{\int_{T}^{+\infty}c(s)e^{-s}ds}\\
	>&\frac{\int_{M}|F_T|^2e^{-\varphi}c(-\psi)}{\int_{T}^{+\infty}c(s)e^{-s}ds}\\
	\geq&\frac{G(T;\varphi,\psi)}{\int_{T}^{+\infty}c(s)e^{-s}ds}.		\end{split}	
\end{displaymath}
Then the claim holds.

As $c(t)$ is increasing on $(a,+\infty)$ and $\lim_{t\rightarrow+\infty}\sup_{\{\psi<-t\}}(\tilde\psi-\psi)=0$, we obtain that
\begin{equation}
	\label{eq:210618b}
	\begin{split}
		\lim_{t\rightarrow+\infty}\frac{G(t;\tilde\varphi,\tilde\psi)}{\int_{t}^{+\infty}c(s)e^{-s}ds}\leq&\lim_{t\rightarrow+\infty}\frac{\int_{\{\tilde\psi<-t\}}|F|^2e^{-\tilde\varphi}c(-\tilde\psi)}{\int_{t}^{+\infty}c(s)e^{-s}ds}\\
		\leq&\lim_{t\rightarrow+\infty}\frac{\int_{\{\psi<-t\}}|F|^2e^{-\varphi-\psi}e^{\tilde\psi}c(-\psi)}{\int_{t}^{+\infty}c(s)e^{-s}ds}\\
		\leq&\lim_{t\rightarrow+\infty}\left(\sup_{\{\psi<-t\}}e^{\tilde\psi-\psi}\right)\frac{\int_{\{\psi<-t\}}|F|^2e^{-\varphi}c(-\psi)}{\int_{t}^{+\infty}c(s)e^{-s}ds}\\
		=&\frac{\int_{\{\psi<-T\}}|F|^2e^{-\varphi}c(-\psi)}{\int_{T}^{+\infty}c(s)e^{-s}ds}	.		\end{split}
\end{equation}

Combining inequality \eqref{eq:210618a} and \eqref{eq:210618b}, we have
$$\lim_{t\rightarrow+\infty}\frac{G(t;\tilde\varphi,\tilde\psi)}{\int_{t}^{+\infty}c(s)e^{-s}ds}<\lim_{t\rightarrow T+0}\frac{G(t;\tilde\varphi,\tilde\psi)}{\int_{t}^{+\infty}c(s)e^{-s}ds},$$
which contradicts the concavity of $G(\hat{h}^{-1}(r);\tilde\varphi,\tilde\psi)$. Thus the assumption doesn't hold, i.e., $G(\hat{h}^{-1}(r);\varphi,\psi)$ is not linear with respect to $r$.

Especially, if $\psi$ is strictly plurisubharmonic at $z_1\in M\backslash(\cap_{t}\overline{\{\psi<-t\}})$, we can construct a $\tilde\psi\geq\psi$ satisfying the three statements in Theorem \ref{thm:n4}, which implies $G(\hat{h}^{-1}(r);\varphi,\psi)$ is not linear with respect to $r$. In fact, there is a small open neighborhood $(U,w)$ of $z_1$ and $w=(w_1,...,w_n)$ is the local coordinate on $U$ such that $U\cap(\cap_{t}\overline{\{\psi<-t\}})=\emptyset$ and $i\partial\bar\partial\psi>\epsilon\omega$ for some $\epsilon>0$, where $\omega=i\sum_{j=1}^{n}dw_j\wedge d\bar w_j$ on $U$. Let $\rho$ be a smooth nonnegative function on $M$ satisfying $\rho\not\equiv0$ and $supp\rho\subset\subset U$. It is clear that there exists a positive number $\delta$ such that
$$i\partial\bar\partial(\psi+\delta\rho)>0$$
holds on $U$ and $\psi+\delta\rho<-T$ on $M$. Let $\tilde\psi=\psi+\delta\rho$, it is clear that $\tilde\psi$ satisfies the three statements in Theorem \ref{thm:n4}. Thus we complete the proof of Theorem \ref{thm:n4}.

\subsection{A limiting property of $G(t)$}
The following Proposition gives a limiting property of $G(t)$, which will be used in the proof of Theorem \ref{thm:n2} and Corollary \ref{c:ll3}.
\begin{Proposition}
	\label{p:G}Let $M$ be an $n-$dimensional Stein manifold, and let $S$ be an analytic subset of $M$. Let $c\in P_T$, and let $(\varphi,\psi)\in W$. Let $\mathcal{F}|_{Z_0}=\mathcal{I}(\psi_1)|_{S_{reg}}$. Assume that  $G(T)\in(0,+\infty)$ and  $\psi_2(z)>-\infty$ for almost every $z\in S_{reg}$.
	
	Assume that $c(t)$ is increasing on $(a,+\infty)$ for some $a>T$.
	Then we have
	\begin{equation}
		\label{eq:210626a}
		\lim_{t\rightarrow+\infty}\frac{G(t)}{\int_{t}^{+\infty}c(l)e^{-l}dl}=\sum_{k=1}^{n}\frac{\pi^{k}}{k!}\int_{S_{n-k}}\frac{|f|^2}{dV_M}e^{-\varphi-\psi_2}dV_M[\psi_1].
	\end{equation}
\end{Proposition}
\begin{proof}
	$\lim_{t\rightarrow+\infty}\frac{G(t)}{\int_{t}^{+\infty}c(l)e^{-l}dl}\leq\sum_{k=1}^{n}\frac{\pi^{k}}{k!}\int_{S_{n-k}}\frac{|f|^2}{dV_M}e^{-\varphi-\psi_2}dV_M[\psi_1]$ can be obtained by using Corollary \ref{c:L2b}. Thus, we just need to prove that
	 $$\lim_{t\rightarrow+\infty}\frac{G(t)}{\int_{t}^{+\infty}c(l)e^{-l}dl}\geq\sum_{k=1}^{n}\frac{\pi^{k}}{k!}\int_{S_{n-k}}\frac{|f|^2}{dV_M}e^{-\varphi-\psi_2}dV_M[\psi_1].$$

	For any $t\geq T$, there exists a holomorphic $(n,0)$ form $F_t$ on $\{\psi<-t\}$, such that $F_t|_S=f$ and $\int_{\{\psi<-t\}}|F_t|^{2}e^{-\varphi}c(-\psi)=G(t)$.
	
Let $\{U^{\alpha}\}_{\alpha\in\mathbb{N}}$ be a coordinate patches of $M\backslash S_{sing}$, biholomorphic to polydiscs, and admite the following property: if $U^{\alpha}\cap S_{reg}\not=\emptyset$, and we denote the corresponding coordinates by $(z^{\alpha},w^{\alpha})\in\Delta^{l}\times\Delta^{n-l}$, where $z^{\alpha}=(z_1^{\alpha},...,z_{l}^{\alpha})$ and $w^{\alpha}=(w_1^{\alpha},...,w_{n-l}^{\alpha})$ for some $l\in\{0,1,2...,n-1\}$, then $U^{\alpha}\cap S=U^{\alpha}\cap S_{l}=\{w^{\alpha}=0\}$. Let $\{v^{\alpha}\}$ be a partition of unity subbordinate to $\{U^{\alpha}\}$.

As $\varphi+\psi_2$ is plurisubharmonic, then there exist smooth plurisubharmonic functions $\Phi_n$ on $M$ decreasingly convergent to $\varphi+\psi_2$. Thus, we have
\begin{equation}
	\label{eq:210620a}\int_{\{\psi<-t\}}v^{\alpha}|F_t|^2e^{-\varphi}c(-\psi)\geq\int_{\{\psi<-t\}}v^{\alpha}|F_t|^2e^{-\Phi_n+\psi_2}c(-\psi)\end{equation}
for any $n\in\mathbb{N}$.

Firstly, we consider $\int_{\{\psi<-t\}}v^{\alpha}|F_t|^2e^{-\Phi_n+\psi_2}c(-\psi)$, where $U^{\alpha}\cap S_{l}\not=\emptyset$.

Note that $\psi=\psi_1+\psi_2$ and $\psi_1\in A'(S)$, then for small enough $s>0$, $\psi_1=(n-l)\log(|w^{\alpha}|^2)+h_1$ on $\Delta^{l}\times\{|w^{\alpha}|<s\}$ and $h_1$ is continuos on $\Delta^{l}\times\{|w^{\alpha}|<s\}$. For any $\epsilon>0$, there exists $s>0$ such that $v^{\alpha}(z^{\alpha},w^{\alpha})\geq \max{\{v^{\alpha}(z^{\alpha},0)-\epsilon,0\}}$, $\Phi_{n}(z^{\alpha},w^{\alpha})\leq\Phi_{n}(z^{\alpha},0)+\epsilon$, and $h_1(z^{\alpha},w^{\alpha})\leq h_1(z^{\alpha},0)+\epsilon$ on $\Delta^{l}\times\{|w^{\alpha}|<s\}$. Let $\psi_s(z^{\alpha})=\sup_{|w^{\alpha}|<s}\psi_2(z^{\alpha},w^{\alpha})$. As $\psi_2(z)>-\infty$ for almost every $z\in S_{reg}$, we know $\psi_s(z^{\alpha})>-\infty$ for almost every $z^{\alpha}\in\Delta^{l}$. Let $v^{\alpha}_{\epsilon}:=\max{\{v^{\alpha}(z^{\alpha},0)-\epsilon,0\}}$. As $c(t)$ is increasing for $t>a$, then we have
\begin{equation}
	\label{eq:210620b}
	\begin{split}
		&\int_{\{\psi<-t\}}v^{\alpha}|F_t|^2e^{-\Phi_n+\psi_2}c(-\psi)\\
		\geq&\int_{\{\psi_2+h_1+(n-l)\log(|w^{\alpha}|^2)<-t\}\cap\{|w^{\alpha}|<s\}}v^{\alpha}_{\epsilon}|F_t|^2e^{-\Phi_n(z^{\alpha},0)-\epsilon+\psi_2}c(-\psi)\\
		\geq&\int_{\{\psi_s+h_1(z^{\alpha},0)+\epsilon+(n-l)\log(|w^{\alpha}|^2)<-t\}\cap\{|w^{\alpha}|<s\}}v^{\alpha}_{\epsilon}|F_t|^2e^{-\Phi_n(z^{\alpha},0)-\epsilon+\psi_2}\\
		&\times c(-\psi_s-h_1(z^{\alpha},0)-\epsilon-(n-l)\log(|w^{\alpha}|^2))
	\end{split}
\end{equation}
for $t>a$.

Without loss of generality, assume that $dV_{M}=(\wedge_{k=1}^{l}idz_k^{\alpha}\wedge d\bar{z}_k^{\alpha})\wedge(\wedge_{k=1}^{n-l}idw_k^{\alpha}\wedge d\bar{w}_k^{\alpha})$, $dV_{\alpha}=\wedge_{k=1}^{l}idz_k^{\alpha}\wedge d\bar{z}_k^{\alpha}$ on $U^{\alpha}$ and $dV'_{\alpha}=\wedge_{k=1}^{n-l}idw_k^{\alpha}\wedge d\bar{w}_k^{\alpha}$. Let $h_2(z^{\alpha}):=\psi_s(z^{\alpha})+h_1(z^{\alpha},0)+\epsilon$. As $\frac{|F_t|^2}{dV_M}e^{\psi_2}$ is plurisubharmonic on $\Delta^{l}\times\{|w^{\alpha}|<s\}$, then we obtain that inequality
\begin{equation}
	\label{eq:210629a}
	\begin{split}
	&\int	_{\{h_2+(n-l)\log(|w^{\alpha}|^2)<-t\}\cap\{|w^{\alpha}|<s\}}\frac{|F_t|^{2}}{dV_M}e^{\psi_2}c(-h_2-(n-l)\log(|w^{\alpha}|^2))dV'_{\alpha}\\
	\geq&\frac{|f(z^{\alpha},0)|^2}{dV_M}e^{\psi_2(z^{\alpha},0)}\int_{\{h_2+(n-l)\log(|w^{\alpha}|^2)<-t\}\cap\{|w^{\alpha}|<s\}}c(-h_2-(n-l)\log(|w^{\alpha}|^2))dV'_{\alpha}\\
	=&	2^{n-l}\frac{\sigma_{2n-2l-1}}{2(n-l)}\frac{|f(z^{\alpha},0)|^2}{dV_M}e^{\psi_2(z^{\alpha},0)}\int_{\max{\{t,-h_2(z^{\alpha})-2(n-l)\log(s)\}}}^{+\infty}c(l)e^{-l}dl
	\end{split}
\end{equation}
holds for any $z^{\alpha}\in\Delta^l$.
It follows from inequality \eqref{eq:210620b} and \eqref{eq:210629a} that
\begin{equation}
	\label{eq:210620c}
	\begin{split}
		&\int_{\{\psi<-t\}}v^{\alpha}|F_t|^2e^{-\Phi_n+\psi_2}c(-\psi)\\
		\geq&\int_{\Delta^{l}}v^{\alpha}_{\epsilon}e^{-\Phi_n(z^{\alpha},0)-\epsilon}\\
		&\times\int	_{\{h_2+(n-l)\log(|w^{\alpha}|^2)<-t\}\cap\{|w^{\alpha}|<s\}}\frac{|F_t|^{2}}{dV_M}e^{\psi_2}c(-h_2-(n-l)\log(|w^{\alpha}|^2))dV_M\\
		\geq&2^{n-l}\frac{\sigma_{2n-2l-1}}{2(n-l)}\int_{\Delta^{l}}v^{\alpha}_{\epsilon}e^{-\Phi_n(z^{\alpha},0)-\epsilon}\frac{|f(z^{\alpha},0)|^2}{dV_M}e^{\psi_2(z^{\alpha},0)}e^{-h_2}\\
		&\times\left(\int_{\max{\{t,-h_2(z^{\alpha})-2(n-l)\log(s)\}}}^{+\infty}c(l)e^{-l}dl\right)dV_{\alpha}    \end{split}
	\end{equation}
for $t>a$.

Next, we prove that
\begin{equation}
	\label{eq:210620g}\liminf_{t\rightarrow+\infty}\frac{\int_{\{\psi<-t\}}v^{\alpha}|F_t|^2e^{-\varphi}c(-\psi)}{\int_{t}^{+\infty}c(l)e^{-l}dl}\geq\frac{\pi^{n-l}}{(n-l)!}\int_{S_l}v^{\alpha}e^{-\varphi-\psi_2}\frac{|f(z^{\alpha},0)|^2}{dV_M}dV_M[\psi_1].\end{equation}

It follows from  $\psi_s(z^{\alpha})>-\infty$ for almost every $z^{\alpha}\in\Delta^{l}$ that  $h_2(z^{\alpha})>-\infty$ for almost every $z^{\alpha}\in\Delta^{l}$. Thus we have
\begin{equation}
	\label{eq:210620d}
	\liminf_{t\rightarrow+\infty}\frac{\int_{\max{\{t,-h_2(z^{\alpha})-2(n-l)\log(s)\}}}^{+\infty}c(l)e^{-l}dl}{\int_{t}^{+\infty}c(l)e^{-l}dl}=1
\end{equation}
for almost every $z^{\alpha}\in\Delta^l$. Combining inequality \eqref{eq:210620c}, equality \eqref{eq:210620d} and Fatou's Lemma, we have
\begin{equation}
	\label{eq:210620e}
	\begin{split}
	&\liminf_{t\rightarrow+\infty}\frac{\int_{\{\psi<-t\}}v^{\alpha}|F_t|^2e^{-\Phi_n+\psi_2}c(-\psi)}{\int_{t}^{+\infty}c(l)e^{-l}dl}\\
	\geq&2^{n-l}\frac{\sigma_{2n-2l-1}}{2(n-l)}\int_{\Delta^{l}}v^{\alpha}_{\epsilon}e^{-\Phi_n(z^{\alpha},0)-\epsilon}\frac{|f(z^{\alpha},0)|^2}{dV_M}e^{\psi_2(z^{\alpha},0)}e^{-h_2}\\
		&\times\liminf_{t\rightarrow+\infty}\frac{\int_{\max{\{t,-h_2(z^{\alpha})-2(n-l)\log(s)\}}}^{+\infty}c(l)e^{-l}dl}{\int_{t}^{+\infty}c(l)e^{-l}dl}dV_{\alpha}\\
		=&2^{n-l}\frac{\sigma_{2n-2l-1}}{2(n-l)}\int_{\Delta^{l}}v^{\alpha}_{\epsilon}e^{-\Phi_n(z^{\alpha},0)-\epsilon}\frac{|f(z^{\alpha},0)|^2}{dV_M}e^{\psi_2(z^{\alpha},0)}e^{-h_2}	dV_{\alpha}
	\end{split}	
\end{equation}

 As $dV_{M}=dV_{\alpha}\wedge(\wedge_{k=1}^{n-l}idw_k^{\alpha}\wedge d\bar{w}_k^{\alpha})$ and $\psi_1=(n-l)\log(|w^{\alpha}|^2)+h_1$, by definition of $dV_M[\psi_1]$, we have $dV_M[\psi_1]=2^{n-l}e^{-h_1}dV_{\alpha}$ on $\Delta^{l}\subset S_l$. Then inequality \eqref{eq:210620e} becomes
 \begin{equation}
 	\label{eq:210620f}
 	\begin{split}
 		&\liminf_{t\rightarrow+\infty}\frac{\int_{\{\psi<-t\}}v^{\alpha}|F_t|^2e^{-\Phi_n+\psi_2}c(-\psi)}{\int_{t}^{+\infty}c(l)e^{-l}dl}\\
 		\geq&\frac{\sigma_{2n-2l-1}}{2(n-l)}\int_{S_l}v^{\alpha}_{\epsilon}e^{-\Phi_n(z^{\alpha},0)-\epsilon}\frac{|f(z^{\alpha},0)|^2}{dV_M}e^{\psi_2(z^{\alpha},0)-\psi_s(z^{\alpha})-\epsilon}dV_M[\psi_1]. 	\end{split}
 \end{equation}
When $s\rightarrow0$, $\psi_s(z^{\alpha})$ is decreasing to  $\psi_2(z^{\alpha},0)$ for any $z^{\alpha}\in\Delta^l$. As $\psi_2(z^{\alpha},0)>-\infty$ for almost every $z^{\alpha}\in\Delta^l$,  let $s\rightarrow 0$ and $\epsilon\rightarrow0$, then inequality \eqref{eq:210620f} implies that
\begin{equation}
\label{eq:210620h}
	\begin{split}
 		&\liminf_{t\rightarrow+\infty}\frac{\int_{\{\psi<-t\}}v^{\alpha}|F_t|^2e^{-\Phi_n+\psi_2}c(-\psi)}{\int_{t}^{+\infty}c(l)e^{-l}dl}\\
 		\geq&\frac{\pi^{n-l}}{(n-l)!}\int_{S_l}v^{\alpha}e^{-\Phi_n(z^{\alpha},0)}\frac{|f(z^{\alpha},0)|^2}{dV_M}dV_M[\psi_1]. 	\end{split}
\end{equation}
Note that $\Phi_n$ decreasing to $\varphi+\psi_2$, then inequality \eqref{eq:210620h} implies that inequality \eqref{eq:210620g} holds.

Following from inequality \eqref{eq:210620g} and the concavity of $G(t)$, we have
$$\lim_{t\rightarrow+\infty}\frac{G(t)}{\int_{t}^{+\infty}c(l)e^{-l}dl}\geq\sum_{k=1}^{n}\frac{\pi^{k}}{k!}\int_{S_{n-k}}\frac{|f|^2}{dV_M}e^{-\varphi-\psi_2}dV_M[\psi_1].$$
Thus, Proposition \ref{p:G} holds.
\end{proof}

\subsection{Proof of Theorem \ref{thm:n2}}

 Assume that $G(\hat{h}^{-1}(r))$ is linear with respect to $r$. As $G(T)\in(0,+\infty)$, we have $\int_{T}^{+\infty}c(t)e^{-t}dt<+\infty$. It follows from Corollary \ref{thm:1}, Remark \ref{r:c} and Lemma \ref{l:c'} that we can assume $c(t)e^{-t}$ is strictly decreasing on $(T,+\infty)$ and $c(t)$  is  increasing on $(a,+\infty)$ for some $a>T$. Thus, Proposition \ref{p:G} shows that equality \eqref{eq:2106i} holds.

In the following part,  assume that there exists $\tilde\psi$ satisfying the three statements in Theorem \ref{thm:n2} to get a contradiction. We prove it by comparing $G(t;\varphi,\psi)$ and $G(t;\tilde\varphi,\tilde\psi)$, where $\tilde\varphi=\varphi+\psi-\tilde\psi$.
It follows from Proposition \ref{p:G} and the linearity of $G(\hat{h}^{-1}(r);\varphi,\psi)$ that $\sum_{k=1}^{n}\int_{S_{n-k}}\frac{\pi^{k}}{k!}\frac{|F|^2}{dV_M}e^{-\varphi-\psi_2}dV_M[\psi_1]<+\infty$ and equality
\begin{equation}
	\label{eq:210620i}
	\frac{G(t;\varphi,\psi)}{\int_{t}^{+\infty}c(l)e^{-l}dl}=\sum_{k=1}^{n}\frac{\pi^{k}}{k!}\int_{S_{n-k}}\frac{|F|^2}{dV_M}e^{-\varphi-\psi_2}dV_M[\psi_1]\end{equation}
holds for any $t\geq T$.

Let $\tilde\varphi=\varphi+\psi-\tilde\psi$ As $(\tilde\varphi,\tilde\psi)\in W$, there exist plurisubharmonic functions $\tilde\psi_1$ and $\tilde\psi_2$ such that $\tilde\psi=\tilde\psi_1+\tilde\psi_2$, $\tilde\psi_1\in A'(S)$ and $\tilde\varphi+\tilde\psi_2$ is plurisubharmonic on $M$. $dV_M[\psi_1]=e^{-\psi_1+\tilde\psi_1}dV_M[\tilde\psi_1]$ implies that $\sum_{k=1}^{n}\frac{\pi^{k}}{k!}\int_{S_{n-k}}\frac{|F|^2}{dV_M}e^{-\tilde\varphi-\tilde\psi_2}dV_M[\tilde\psi_1]=\sum_{k=1}^{n}\frac{\pi^{k}}{k!}\int_{S_{n-k}}\frac{|F|^2}{dV_M}e^{-\varphi-\psi_2}dV_M[\psi_1]<+\infty.$ It follows from Corollary \ref{c:L2b} that
\begin{equation}
	\label{eq:210620j}\begin{split}
	\frac{G(T;\tilde\varphi,\tilde\psi)}{\int_{T}^{+\infty}c(l)e^{-l}dl}\leq &\sum_{k=1}^{n}\int_{S_{n-k}}\frac{\pi^{k}}{k!}\frac{|F|^2}{dV_M}e^{-\tilde\varphi-\tilde\psi_2}dV_M[\tilde\psi_1]\\
	=&\sum_{k=1}^{n}\int_{S_{n-k}}\frac{\pi^{k}}{k!}\frac{|F|^2}{dV_M}e^{-\varphi-\psi_2}dV_M[\psi_1].	
	\end{split}	
\end{equation}
Since $\tilde\psi\geq\psi$, $\tilde\psi\not=\psi$, there exists a subset $U$ of $M$ such that $\mu (U)>0$ and $\tilde\psi>\psi$ on $U$, where $\mu$ is Lebesgue measure on $M$. As $c(t)e^{-t}$
is strictly decreasing on $(T,+\infty)$, we have $G(T;\tilde\varphi,\tilde\psi)>G(T;\varphi,\psi)$. Then  inequality \eqref{eq:210620j} implies that
$$\sum_{k=1}^{n}\int_{S_{n-k}}\frac{\pi^{k}}{k!}\frac{|F|^2}{dV_M}e^{-\varphi-\psi_2}dV_M[\psi_1]\geq\frac{G(T;\tilde\varphi,\tilde\psi)}{\int_{T}^{+\infty}c(l)e^{-l}dl}>\frac{G(T;\varphi,\psi)}{\int_{T}^{+\infty}c(l)e^{-l}dl},$$
which contradicts equality \eqref{eq:210620i}. Thus Theorem \ref{thm:n2} holds.

\section{Proofs of Theorem \ref{thm:e2}, Theorem \ref{thm:e1}, Corollary \ref{c:ll3} and Corollary \ref{c:ll4}}

In this section,  we prove Theorem \ref{thm:e2}, Theorem \ref{thm:e1}, Corollary \ref{c:ll3} and Corollary \ref{c:ll4}.

\subsection{A necessary condition of linearity}
The following Proposition give a necessary condition of $G(\hat{h}^{-1}(r))$ is linear, and will be used in the proof of Theorem \ref{thm:e2}.
 \begin{Proposition}
 	\label{l:n} Let $\Omega$ be an open Riemann surface. Let $c\in\mathcal{P}_0$, and assume that there exists $t\geq0$ such that $G(t)\in(0,+\infty)$.
 If $G(\hat{h}^{-1}(r))$ is linear with respect to $r$, then there is no  Lebesgue measurable function $\tilde \varphi\geq\varphi$ such that $\tilde\varphi+\psi$ is subharmonic function on $M$ and satisfies:
	
	$(1)$ $\tilde\varphi\not=\varphi$ and $\mathcal{I}(\tilde\varphi+\psi)=\mathcal{I}(\varphi+\psi)$;
	
	$(2)$ $\lim_{t\rightarrow 0+0}\sup_{\{\psi\geq-t\}}(\tilde\varphi-\varphi)=0$;
	
	$(3)$ there exists an open subset $U\subset\subset\Omega$ such that $\sup_{\Omega\backslash U}(\tilde\varphi-\varphi)<+\infty$, $e^{-\tilde\varphi}c(-\psi)$ has a positive lower bound on $U$ and $\int_{U}|F_1-F_2|^2e^{-\varphi}c(-\psi)<+\infty$ for any $F_1\in\mathcal{H}^2(c,\tilde\varphi,t)$ and $F_2\in\mathcal{H}^2(c,\varphi,t)$, where $U\subset\subset\{\psi<-t\}$. 	
 \end{Proposition}

\begin{proof}
We prove the lemma by comparing $G(t;\varphi)$ and $G(t;\tilde\varphi)$. In the following, let us assume that  there exists a Lebesgue measurable function $\tilde \varphi$ satisfying these properties in Proposition \ref{l:n} to get a contradiction.

As $G(\hat{h}^{-1}(r);\varphi)$ is linear with respect to $r$, it follows from Corollary \ref{thm:1} that there exists a holomorphic $(1,0)$ form $F$ on $\Omega$ such that $(F-f,z_0)\in (\mathcal{O}(K_{\Omega})\otimes\mathcal{F})_{z_0}$  and $\forall t\geq 0$  equality
$$G(t;\varphi)=\int_{\{\psi<-t\}}|F|^2e^{-\varphi}c(-\psi)$$
holds.
 As $\tilde\varphi+\psi$ is subharmonic and there exists a subset $U\subset\subset\Omega$ such that $\sup_{\Omega\backslash U}(\tilde\varphi-\varphi)<+\infty$, $e^{-\tilde\varphi}c(-\psi)$ has a positive lower bound on $U$ and $\mathcal{I}(\tilde\varphi+\psi)=\mathcal{I}(\varphi+\psi)$, it follows from Theorem \ref{thm:general_concave} that $G(\hat{h}^{-1}(r);\tilde\varphi)$ is concave with respect to $r$.

 As $\tilde\varphi+\psi\geq\varphi+\psi$, $\tilde\varphi+\psi\not=\varphi+\psi$ and both of them are subharmonic functions on $\Omega$, then there exists a subset $V$ of $\Omega$ such that
  $e^{-\tilde\varphi}<e^{-\varphi}$ on a subset $V$ and $\mu(V)>0$, where $\mu$ is Lebesgue measure on $\Omega$. As $F\not\equiv0$, inequality
\begin{equation}
	\label{eq:210621a}
	\frac{G(T_0;\tilde\varphi)}{\int_{T_0}^{+\infty}c(s)e^{-s}ds}\leq\frac{\int_{\{\psi<-T_0\}}|F|^2e^{-\tilde\varphi}c(-\psi)}{\int_{T_0}^{+\infty}c(s)e^{-s}ds}<\frac{G(T_0;\varphi)}{\int_{T_0}^{+\infty}c(s)e^{-s}ds}
\end{equation}
holds for some $T_0>0$.
For $t>0$, there exists a holomorphic $(1,0)$ form $F_t$ on $\{\psi<-t\}$ such that $(F_t-f)_{z_0}\in (\mathcal{O}(K_{\Omega})\otimes\mathcal{F})_{z_0}$ and
$$G(t;\tilde\varphi)=\int_{\{\psi<-t\}}|F_t|^2e^{-\tilde\varphi}c(-\psi)<+\infty.$$

As  there exists a subset $U\subset\subset\Omega$ such that $\sup_{\Omega\backslash U}(\tilde\varphi-\varphi)<+\infty$, we get that
\begin{equation}
	\label{eq:210621b}
	\begin{split}
	\int_{\{\psi<-t\}}|F_t|^2e^{-\varphi}c(-\psi)=&	\int_{\{\psi<-t\}\cap U}|F_t|^2e^{-\varphi}c(-\psi)+	\int_{\{\psi<-t\}\backslash U}|F_t|^2e^{-\varphi}c(-\psi)\\
	\leq&2\int_{\{\psi<-t\}\cap U}|F|^2e^{-\varphi}c(-\psi)+2\int_{\{\psi<-t\}\cap U}|F_t-F|^2e^{-\varphi}c(-\psi)\\
	&+e^{\sup_{\Omega\backslash U}(\tilde\varphi-\varphi)}\int_{\{\psi<-t\}\backslash U}|F_t|^2e^{-\tilde\varphi}c(-\psi	)\\
	<&+\infty	\end{split}
\end{equation}
holds for small enough $t>0$.
 It follows from Lemma \ref{lem:A} that
\begin{equation}
	\label{eq:210621c}
	\begin{split}
	G(t_1;\tilde\varphi)-G(t_2;\tilde\varphi)&\geq \int_{\{-t_2\leq\psi<-t_1\}}|F_{t_1}|^2e^{-\tilde\varphi}c(-\psi)\\
		&\geq \left(\inf_{z\in\{-t_2\leq\psi\}}e^{\varphi-\tilde\varphi}\right)\int_{\{-t_2\leq\psi<-t_1\}}|F_{t_1}|^2e^{-\varphi}c(-\psi) \\
		&\geq \left(\inf_{z\in\{-t_2\leq\psi\}}e^{\varphi-\tilde\varphi}\right)\int_{\{-t_2\leq\psi<-t_1\}}|F|^2e^{-\varphi}c(-\psi)
	\end{split}
\end{equation}
holds for small enough $t_1$ and $t_2$ such that $0<t_1<t_2<+\infty$.
As  $\lim_{t\rightarrow 0+0}\sup_{\{\psi\geq-t\}}(\tilde\varphi-\varphi)=0$, it follows from inequality \eqref{eq:210621a} and \eqref{eq:210621c} that
\begin{displaymath}
	\begin{split}
	\liminf_{t_2\rightarrow 0+0}\frac{G(t_1;\tilde\varphi)-G(t_2;\tilde\varphi)}{\int_{t_1}^{t_2}c(s)e^{-s}ds}
	\geq& \liminf_{t_2\rightarrow 0+0}\left(\inf_{z\in\{-t_2\leq\psi\}}e^{\varphi-\tilde\varphi}\right) \frac{\int_{\{-t_2\leq\psi<-t_1\}}|F|^2e^{-\varphi}c(-\psi)}{\int_{t_1}^{t_2}c(s)e^{-s}ds}	\\
	=&\frac{G(T_0;\varphi)}{\int_{T_0}^{+\infty}c(s)e^{-s}ds} \\
	>&\frac{G(T_0;\tilde\varphi)}{\int_{T_0}^{+\infty}c(s)e^{-s}ds},
	\end{split}
\end{displaymath}
which  contradicts the concavity of $G(\hat{h}^{-1}(r);\tilde\varphi)$. Thus the assumption doesn't hold, and we complete the proof of Proposition \ref{l:n}.
\end{proof}
\subsection{Proof of Theorem \ref{thm:e2}}

Firstly, we prove the sufficiency by using Theorem \ref{thm:suita}.
 The following remark shows that it suffices to prove the sufficiency for the case $\psi=2G_{\Omega}(z,z_0)$.
\begin{Remark}
	\label{r:c(p)}
	Let $\tilde\varphi=\varphi+a\psi$, $\tilde c(t)=c(\frac{t}{1-a})e^{-\frac{at}{1-a}}$ and $\tilde\psi=(1-a)\psi$ for some $a\in(-\infty,1)$. It is clear that $e^{-\tilde\varphi}\tilde c(-\tilde\psi)=e^{-\varphi}c(-\psi)$, $(1-a)\int_{t}^{+\infty}c(l)e^{-l}dl=\int_{(1-a)t}^{+\infty}\tilde c(l)e^{-l}dl$ and $G(t;\varphi,\psi,c)=G((1-a)t;\tilde\varphi,\tilde\psi,\tilde{c})$.
\end{Remark}

 Let $\tilde{c}\equiv1$ on $(0,+\infty)$.
 Set $\hat{f}=\frac{f}{g}$, $\hat{\varphi}=\varphi-2\log|g|=2u$, and $\hat{\mathcal{F}}_{z_0}=\mathcal{I}(\hat{\varphi}+\psi)_{z_0}=\mathcal{I}(2G_{\Omega}(z,z_0))_{z_0}$. Denote
$$\inf\{\int_{\{\psi<-t\}}|\tilde f|^{2}e^{-\hat{\varphi}}:(\tilde{f}-\hat{f})_{z_0}\in(\mathcal{O}(K_{\Omega})\otimes\hat{\mathcal{F}})_{z_0}\,\&\,\tilde{f}\in H^0(\{\psi<-t\},\mathcal{O}(K_{\Omega}))\}$$
by $\hat{G}(t;\tilde{c})$. Without loss of generality, we can assume that $\hat{f}(z_0)=dw$, where $w$ is a local coordinate on a neighborhood $V_{z_0}$ of $z_0$ satisfying $w(z_0)=0$. By definition of $G(t;\tilde{c})$ and $B_{\Omega,e^{-2u}}(z_0)$, it is clear that $G(t;\tilde{c})=\hat{G}(t;\tilde{c})$ and $\hat{G}(0;\tilde{c})=\frac{2}{B_{\Omega,e^{-2u}}(z_0)}=\inf\{\int_{\Omega}|\tilde f|^2e^{-2u}:\tilde f$ is a holomorphic extension of $\hat f$ from $z_0$ to $\Omega$$\}$. Theorem \ref{thm:suita} shows that $G(0;\tilde{c})=\hat{G}(0;\tilde{c})=2\pi\frac{e^{-2u(z_0)}}{c_{\beta}^2(z_0)}$. Note that $\|\hat{f}\|_{z_0}=\pi\int_{z_0}\frac{|\hat{f}|^2}{dV_M}e^{-\hat{\varphi}}dV_{\Omega}[2G_{\Omega}(z,z_0)]=2\pi\frac{e^{-2u(z_0)}}{c_{\beta}^2(z_0)}$, therefore Theorem \ref{thm:ll1} tells us that $G(-\log r;\tilde{c})$ and $\hat{G}(-\log r;\tilde{c})$ is linear with respect to $r$.

As $\psi=2G_{\Omega}(z,z_0)$,  Lemma \ref{l:G-compact} shows that,  for any $t_0\geq0,$ there exists  $t>t_0$ such that $\{G_{\Omega}(z,z_0<-t)\}$ is a relatively compact subset of $\Omega$ and $g$ has no zero point in $\{G_{\Omega}(z,z_0<-t)\}\backslash{\{z_0\}}$. Combining Corollary \ref{thm:1}, Remark \ref{r:c} and $G(-\log r;\tilde{c})$ is linear with respect to $r$, we obtain that $G(\hat{h}^{-1}(r))$ is linear with respect to $r$, where $\hat{h}(t)=\int_{t}^{+\infty}c(l)e^{-l}dl$.

In the following part, we prove the necessity in three steps.

By Remark \ref{r:c(p)}, without loss of generality, we can assume that $\varphi$ is subharmonic near $z_0$. As $\varphi+\psi$ is a subharmonic function on $\Omega$, it follows from Weierstrass Theorem on open Riemann surfaces (see \cite{OF81}) and Siu's Decomposition Theorem that
\begin{equation}
	\label{eq:210620l}
	\varphi+\psi=2\log|g|+2G_{\Omega}(z,z_0)+2u,
\end{equation}
where $g$ is a holomorphic function on $\Omega$, and $u$ is a subharmonic function on $\Omega$ such that $v(dd^cu,z)\in[0,1)$ for any $z\in\Omega$.

\

\emph{Step 1: $\mathcal{F}_{z_0}=\mathcal{I}(\varphi+\psi)_{z_0}$,\,$ord_{z_0}(g)=org_{z_0}(f_1)$ and $v(dd^{c}\psi,z_0)>0$.}

\

As $\mathcal{I}(\varphi+\psi)_{z_0}=\mathcal{I}(2\log|g|+2G_{\Omega}(z,z_0))_{z_0}\subset\mathcal{F}_{z_0}$ and $G(0)\not=0$, we have $ord_{z_0}(g)+1>org_{z_0}(f_1)$. Corollary \ref{thm:1} tells us there exists a holomorphic $(1,0)$ form on $\Omega$ such that $(F-f,z_0)\in (\mathcal{O}(K_{\Omega})\otimes\mathcal{F})_{z_0}$ and $G(t)=\int_{\{\psi<-t\}}|F|^{2}e^{-\varphi}c(-\psi)$ for $t\geq0$. Denote that $\tilde{c}(t)=\max{\{c(t),e^{rt}\}}$  on $(0,+\infty)$, where $r\in(0,1)$. Set $F=\tilde Fdw$ on $V_{z_0}$,   and it follows from Corollary \ref{thm:1} and  Remark \ref{r:c} that $|\tilde F|^2e^{-\varphi-r\psi}$ is locally integrable near $z_0$ for any $r\in(0,1)$, which implies that $ord_{z_0}(\tilde F)\geq ord_{z_0}(g)$.

we prove $\mathcal{F}_{z_0}=\mathcal{I}(\varphi+\psi)_{z_0}$ by contradiction: if not,  then $\mathcal{F}_{z_0}\subsetneqq\mathcal{I}(2\log|g|+2G_{\Omega}(z,z_0))_{z_0}$. Since $ord_{z_0}(\tilde F)\geq ord_{z_0}(g)$,  we have $(\tilde F,z_0)\in \mathcal{F}_{z_0}$, which contradicts to $G(0)\not=0$. Thus $\mathcal{F}_{z_0}=\mathcal{I}(\varphi+\psi)_{z_0}$.

 As $ord_{z_0}(\tilde F)\geq ord_{z_0}(g)$, $ord_{z_0}(g)+1>org_{z_0}(f_1)$ and $(\tilde F-f_1,z_0)\in\mathcal{I}(2\log|g|+2G_{\Omega}(z,z_0))_{z_0}$, we have $ord_{z_0}(g)=org_{z_0}(f_1)$.

we prove $v(dd^{c}\psi,z_0)>0$ by contradiction: if not, $v(dd^{c}\psi,z_0)=0$ shows that $\mathcal{I}(\varphi+\psi)_{z_0}=\mathcal{I}(\varphi)_{z_0}$. Without loss of generality, we can assume that $c(t)>1$ for  large enough $t$, then $|\tilde F|^2e^{-\varphi}$ is locally integrable near $z_0$, which contradicts to $(\tilde F,z_0)\not\in\mathcal{F}_{z_0}$. Thus $v(dd^{c}\psi,z_0)>0$.

\

\emph{Step 2: $\psi=2pG_{\Omega}(z,z_0)$ for some $p>0$.}

\

As $\psi$ is  subharmonic function on $\Omega$, it follows from Siu's Decomposition Theorem that $\psi=2pG_{\Omega}(z,z_0)+\psi_1$ such that $v(dd^c\psi_1,z_0)=0$.

 Firstly, we prove $\psi_1$ is harmonic near $z_0$ by contradiction : if not,  there exists a closed  positive $(1,1)$ current $T\not\equiv0$, such that $supp T\subset\subset V_{z_0}$, $T\leq \frac{1}{2}i\partial\bar\partial\psi_1$ on $V_{z_0}$, where $V_{z_0}$ is an open neighborhood of $z_0$, satisfying that $g$ has not zero point on $\overline{V_{z_0}}\backslash\{z_0\}$, $\varphi$ is subharmonic on a neighborhood of $\overline{V_{z_0}}$ and $V_{z_0}\subset\subset\Omega$. Note that $\{z\in \overline{V_{z_0}}:\mathcal{I}(\varphi+\psi)_z\not=\mathcal{O}_z\}=\{z_0\}$.

Using Lemma \ref{l:cu}, there exists a subharmonic function $\Phi<0$ on $\Omega$, which satisfies the following properties: $i\partial\bar\partial\Phi\leq T$ and $i\partial\bar\partial\Phi\not\equiv0$; $\lim_{t\rightarrow0+0}(\inf_{\{G_{\Omega}(z,z_0)\geq-t\}}\Phi(z))=0$; $supp (i\partial\bar\partial\Phi)\subset V_{z_0}$ and $\inf_{\Omega\backslash V_{z_0}}\Phi>-\infty$. It following from Lemma \ref{l:green}, $v(dd^{c}\psi,z_0)>0$ and $\psi<0$ on $\Omega$, that $\lim_{t\rightarrow0+0}(\inf_{\{\psi\geq-t\}}\Phi(z))=0$.

Set $\tilde\varphi=\varphi-\Phi$, then $\tilde\varphi+\psi=\varphi+2pG_{\Omega}(z,z_0)+\psi_1-\Phi$ on $V_{z_0}$, where $\psi_1-\Phi$ is subharmonic on $V_{z_0}$. It is clear that $\tilde\varphi\geq\varphi$ and $\tilde\varphi\not=\varphi$. $supp T\subset\subset V_{z_0}$ and $i\partial\bar\partial\Phi\leq T\leq i\partial\bar\partial\varphi_1$ on $V_{z_0}$ show that $\tilde\varphi+\psi$ is subharmonic on $\Omega$, $\mathcal{I}(\tilde\varphi+\psi)=\mathcal{I}(\varphi+\psi)=\mathcal{I}(2\log|g|+2G_{\Omega}(z,z_0))$.

 Without loss of generality, we can assume that $c(t)>e^{\frac{t}{2}}$ for any $t>0$. $T\leq \frac{1}{2}i\partial\bar\partial\psi_1$ on $V_{z_0}$ and $i\partial\bar\partial\Phi\subset\subset V_{z_0}$ show that $\frac{1}{2}\psi-\Phi$ is subharmonic on $\Omega$, which implies that
$e^{-\tilde\varphi}c(-\psi)\geq e^{-\varphi}e^{\Phi-\frac{1}{2}\psi}$
 has a positive lower bound on $V_{z_0}$. Notice that $\inf_{\Omega\backslash V_{z_0}}(\varphi-\tilde\varphi)=\inf_{\Omega\backslash V_{z_0}}\Phi>-\infty$ and $\int_{V_{z_0}}|F_1-F_2|^2e^{-\varphi}c(-\psi)\leq C\int_{V_{z_0}}|F_1-F_2|^2e^{-\varphi-\psi}<+\infty$ for any $F_1\in\mathcal{H}^2(c,\tilde\varphi,t)$ and $F_2\in\mathcal{H}^2(c,\varphi,t)$, where $V_{z_0}\subset\subset\{\psi<-t\}$, then $\tilde\varphi$ satisfies the conditions in Proposition \ref{l:n}, which contradicts to the result of  Proposition \ref{l:n}. Thus $\psi_1$ is harmonic near $z_0$.

Then, we prove $\psi=2pG_{\Omega}(z,z_0)$. Using Remark \ref{r:c(p)}, it suffices to consider the case $p=1$, where $p=\frac{1}{2}v(dd^{c}\psi,z_0)$. By Siu's Decomposition Theorem and Lemma \ref{l:green}, there exists a subharmonic function $\psi_2\leq0$ on $\Omega$ such that $\psi=2G_{\Omega}(z,z_0)+\psi_2$. Note that $\psi_2(z_0)>-\infty$.

As $\Omega$ is an open Riemann surface, there exists a holomorphic function $f_2$ on $\Omega$, such that $ord_{z_0}(f_2)=ord_{z_0}(f_1)$ and $\{z\in\Omega:f_2(z)=0\}=\{z_0\}$.  Set $\tilde{f}=\frac{f}{f_2}$, $\tilde{\varphi}=\varphi-2\log|f_2|$, and $\tilde{\mathcal{F}}_{z_0}=\mathcal{I}(\tilde{\varphi}+\psi)_{z_0}=\mathcal{I}(2G_{\Omega}(z,z_0))_{z_0}$. Denote
\begin{displaymath}
	\begin{split}
		\inf\Bigg\{\int_{\{\psi<-t\}}|F|^{2}e^{-\tilde{\varphi}}c(-\psi):&(F-\tilde{f})_{z_0}\in(\mathcal{O}(K_{\Omega})\otimes\tilde{\mathcal{F}})_{z_0}\\
		&\&\,F\in H^0(\{\psi<-t\},\mathcal{O}(K_{\Omega}))\Bigg\}
	\end{split}
\end{displaymath}
by $\tilde{G}(t)$.
By the definition of $G(t)$ and $\tilde{G}(t)$, we know $G(t)=\tilde{G}(t)$ for any $t\geq0$, therefore $\tilde{G}(\hat{h}^{-1}(r))$ is linear with respect to $r$.
Note that $(\tilde{\varphi},\psi)\in W$, $(\tilde{\varphi}+\psi-2G_{\Omega}(z,z_0),2G_{\Omega}(z,z_0))\in W$, $\psi_2(z_0)>-\infty$ and $\psi_2\leq0$, then Theorem \ref{thm:n2} shows that $\psi=2G_{\Omega}(z,z_0)$.

\

\emph{Step 3.  $u$ is harmonic on $\Omega$ and $\chi_{-u}=\chi_{z_0}.$}

\

Without loss of generality, we can assume that $\psi=2G_{\Omega}(z,z_0)$.
Lemma \ref{l:G-compact} shows that,  for any $t_0\geq0,$ there exists  $t>t_0$ such that $\{G_{\Omega}(z,z_0)<-t\}$ is a relatively compact subset of $\Omega$ and $g$ has no zero point in $\{G_{\Omega}(z,z_0)<-t\}\backslash{\{z_0\}}$. Combining Corollary \ref{thm:1}, Remark \ref{r:c} and $G(\hat{h}^{-1}(r);c)$ is linear with respect to $r$, we obtain that $G(-\log r;\tilde{c}\equiv1)$ is linear with respect to $r$ and $G(0;\tilde{c})\in(0,+\infty)$.

Now, we assume that $u$ is not harmonic to get a contradiction. There exists a closed  positive $(1,1)$ current $T\not\equiv0$, such that $supp T\subset\subset \Omega$ and $T\leq i\partial\bar\partial u$. There exists an open subset $U\subset\subset \Omega$, such that $supp T\subset U$.

Using Lemma \ref{l:cu}, there exists a subharmonic function $\Phi<0$ on $\Omega$, which satisfies the following properties: $i\partial\bar\partial\Phi\leq T$ and $i\partial\bar\partial\Phi\not\equiv0$; $\lim_{t\rightarrow0+0}(\inf_{\{G_{\Omega}(z,z_0)\geq-t\}}\Phi(z))=0$; $supp (i\partial\bar\partial\Phi)\subset U$ and $\inf_{\Omega\backslash U}\Phi>-\infty$.

Set $\tilde\varphi=\varphi-\Phi$, then $\tilde\varphi=2\log|g|+2u-\Phi$ is subharmonic on $\Omega$. It is clear that $\tilde\varphi\geq\varphi$, $\tilde\varphi\not=\varphi$ and $\tilde\varphi+\psi$ is subharmonic on $\Omega$, $\mathcal{I}(\tilde\varphi+\psi)=\mathcal{I}(\varphi+\psi)=\mathcal{I}(2\log|g|+2G_{\Omega}(z,z_0))$.

 As $\tilde\varphi$ is subharmonic on $\Omega$, we have
$e^{-\tilde\varphi}$
 has a positive lower bound on $U$. Note that $\mathcal{I}(\varphi)=\mathcal{I}(\tilde\varphi)$, then
 $$\int_{U}|F_1-F_2|^2e^{-\varphi}\leq2\int_{U}|F_1|^2e^{-\varphi}+2\int_{U}|F_2|^2e^{-\varphi}<+\infty$$
  for any $F_1\in\mathcal{H}^2(\tilde{c},\tilde\varphi,t)$ and $F_2\in\mathcal{H}^2(\tilde{c},\varphi,t)$, where $U\subset\subset\{\psi<-t\}$ and $\tilde{c}\equiv1$.  Since $\inf_{\Omega\backslash U}(\varphi-\tilde\varphi)=\inf_{\Omega\backslash U}\Phi>-\infty$, then $\tilde\varphi$ satisfies the conditions in Proposition \ref{l:n}, which contradicts to the result of  Proposition \ref{l:n}. Thus $u$ is harmonic on $\Omega$.

Finally, we prove $\chi_{-u}=\chi_{z_0}$ by using  Theorem \ref{thm:suita}.

Recall some notations in the proof of sufficiency. Set $\hat{f}=\frac{f}{g}$, $\hat{\varphi}=\varphi-2\log|g|=2u$, and $\hat{\mathcal{F}}_{z_0}=\mathcal{I}(\hat{\varphi}+\psi)_{z_0}=\mathcal{I}(2G_{\Omega}(z,z_0))_{z_0}$. Denote
\begin{displaymath}
	\begin{split}
		\inf\Bigg\{\int_{\{\psi<-t\}}|\tilde f|^{2}e^{-\hat{\varphi}}:&(\tilde{f}-\hat{f})_{z_0}\in(\mathcal{O}(K_{\Omega})\otimes\hat{\mathcal{F}})_{z_0}
		\\&\&\,\tilde{f}\in H^0(\{\psi<-t\},\mathcal{O}(K_{\Omega}))\Bigg\}
	\end{split}
\end{displaymath}
by $\hat{G}(t;\tilde{c})$. Without loss of generality, we can assume that $\hat{f}(z_0)=dw$, where $w$ is a local coordinate on a neighborhood $V_{z_0}$ of $z_0$ satisfying $w(z_0)=0$. By definition of $G(t;\tilde{c})$ and $B_{\Omega,e^{-2u}}(z_0)$, it is clear that $G(-\log r;\tilde{c})=\hat{G}(-\log r;\tilde{c})$ is linear with respect to $r$ and $\hat{G}(0;\tilde{c})=\frac{2}{B_{\Omega,e^{-2u}}(z_0)}=\inf\{\int_{\Omega}|\tilde f|^2e^{-2u}:\tilde f$ is a holomorphic extension of $\hat f$ from $z_0$ to $\Omega$$\}$.

Note that $\|\hat{f}\|_{z_0}=2\pi\frac{e^{-2u(z_0)}}{c_{\beta}^2(z_0)}$ , then Theorem \ref{thm:n2} shows that
$$\hat{G}(0,\tilde{c})=2\pi\frac{e^{-2u(z_0)}}{c_{\beta}^2(z_0)},$$
i.e., $c_{\beta}^2(z_0)=\pi e^{-2u(z_0)}B_{\Omega,e^{-2u}}(z_0)$.
Therefore Theorem \ref{thm:suita} shows that $\chi_{-u}=\chi_{z_0}$.

Thus, Theorem \ref{thm:e2} holds.

\subsection{Proof of Theorem \ref{thm:e1}}
Theorem \ref{thm:e2} implies the  sufficiency. Thus we just need to prove the necessity.

As $\varphi+\psi$ is a subharmonic function on $\Omega$, it follows from Weierstrass Theorem on open Riemann surfaces (see \cite{OF81}) and Siu's Decomposition Theorem that
\begin{equation}
	\label{eq:210620l}
	\varphi+\psi=2\log|g|+2G_{\Omega}(z,z_0)+2u,
\end{equation}
where $g$ is a holomorphic function on $\Omega$, and $u$ is a subharmonic function on $\Omega$ such that $v(dd^cu,z)\in[0,1)$ for any $z\in\Omega$.

As $\mathcal{I}(\varphi+\psi)_{z_0}=\mathcal{I}(2\log|g|+2G_{\Omega}(z,z_0))_{z_0}\subset\mathcal{F}_{z_0}$ and $G(0)\not=0$, we have $ord_{z_0}(g)+1>org_{z_0}(f_1)$. Corollary \ref{thm:1} tells us there exists a holomorphic $(1,0)$ form on $\Omega$ such that $(F-f,z_0)\in (\mathcal{O}(K_{\Omega})\otimes\mathcal{F})_{z_0}$ and $G(t)=\int_{\{\psi<-t\}}|F|^{2}e^{-\varphi}c(-\psi)$ for $t\geq0$. Let $\tilde{c}(t)=\max{\{c(t),e^{rt}\}}$ defined on $(0,+\infty)$, where $r\in(0,1)$. Set $F=\tilde Fdw$ on $V_{z_0}$,   and it follows from Corollary \ref{thm:1} and  Remark \ref{r:c} that $|\tilde F|^2e^{-\varphi-r\psi}$ is locally integrable near $z_0$ for any $r\in(0,1)$. Note that
$$\int_{U}|\tilde F|^2e^{-\frac{\varphi+\psi}{p}}\leq\left(\int_{U}|\tilde F|^{2p}e^{-\varphi-\psi+ps\psi}\right)^{\frac{1}{p}}\left(\int_Ue^{-qs\psi}\right)^{\frac{1}{q}}$$
holds for any $p>1$, $\frac{1}{p}+\frac{1}{q}=1$, $U$ is a small open neighborhood of $z_0$, and $s\in(0,1)$. For any $p\in(1,+\infty)$, we can choose small enough $U$ and small enough $s\in(0,1)$ such that $\int_{U}|\tilde F|^{2p}e^{-\varphi-\psi+ps\psi}<+\infty$ and $\int_Ue^{-qs\psi}<+\infty$, which implies that $(\tilde F,z_0)\in\mathcal{I}(\frac{\varphi+\psi}{p})_{z_0}\subset\mathcal{I}(\frac{2\log|g|+2G_{\Omega}(z,z_0)}{p})_{z_0}$. Therefore,
we have $ord_{z_0}(\tilde F)\geq ord_{z_0}(g)$.

we prove $\mathcal{F}_{z_0}=\mathcal{I}(\varphi+\psi)_{z_0}$ by contradiction: if not,  then $\mathcal{F}_{z_0}\subsetneqq\mathcal{I}(2\log|g|+2G_{\Omega}(z,z_0))_{z_0}$. Since $ord_{z_0}(\tilde F)\geq ord_{z_0}(g)$,  we have $(\tilde F,z_0)\in \mathcal{F}_{z_0}$, which contradicts to $G(0)\not=0$. Thus $\mathcal{F}_{z_0}=\mathcal{I}(\varphi+\psi)_{z_0}$.

 As $ord_{z_0}(\tilde F)\geq ord_{z_0}(g)$, $ord_{z_0}(g)+1>org_{z_0}(f_1)$ and $(\tilde F-f_1,z_0)\in\mathcal{I}(2\log|g|+2G_{\Omega}(z,z_0))_{z_0}$, we have $ord_{z_0}(g)=org_{z_0}(f_1)$.

we prove $v(dd^{c}\psi,z_0)>0$ by contradiction: if not, as $|\tilde F|^2e^{-\varphi-r\psi}$ is locally integrable near $z_0$ for any $r\in(0,1)$ and $ord_{z_0}(g)=org_{z_0}(\tilde F)$, we have $e^{-2G_{\Omega}(z,z_0)+(1-r)\psi}$ is locally integrable near $z_0$.
 Therefore there exists $s>0$ such that
\begin{displaymath}
	\begin{split}
		\int_{\Delta_s}\frac{e^{(1-r)\psi}}{|w|^2}<+\infty,
	\end{split}
\end{displaymath}
where $w$ is a local coordinate near $z_0$ such that $w(z_0)=0$.
$e^{(1-r)\psi}$ is subharmonic shows that
\begin{displaymath}
	2\pi e^{(1-r)\psi(z_0)}\int_{0}^{s}\frac{1}{t}dt=\int_{\Delta_s}\frac{e^{(1-r)\psi}}{|w|^2}<+\infty,
\end{displaymath}
 which contradicts to $\psi(z_0)>-\infty$. Thus $v(dd^{c}\psi,z_0)>0$ holds.

Using Remark \ref{r:c(p)}, it suffices to consider the case $p=1$, where $p=\frac{1}{2}v(dd^{c}\psi,z_0)$. By Siu's Decomposition Theorem and Lemma \ref{l:green}, there exists a subharmonic function $\psi_2\leq0$ on $\Omega$ such that $\psi=2G_{\Omega}(z,z_0)+\psi_2$. Following the assumption in Theorem \ref{thm:e1}, we know $\psi_2(z_0)>-\infty$.

As $\Omega$ is an open Riemann surface, there exists a holomorphic function $f_2$ on $\Omega$, such that $ord_{z_0}(f_2)=ord_{z_0}(f_1)$ and $\{z\in\Omega:f_2=0\}=\{z_0\}$.  Set $\tilde{f}=\frac{f}{f_2}$, $\tilde{\varphi}=\varphi-2\log|f_2|$, and $\tilde{\mathcal{F}}_{z_0}=\mathcal{I}(\tilde{\varphi}+\psi)_{z_0}=\mathcal{I}(2G_{\Omega}(z,z_0))_{z_0}$. Denote
\begin{displaymath}
	\begin{split}
		\inf\Bigg\{\int_{\{\psi<-t\}}|F|^{2}e^{-\tilde{\varphi}}c(-\psi):&(F-\tilde{f})_{z_0}\in(\mathcal{O}(K_{\Omega})\otimes\tilde{\mathcal{F}})_{z_0}\\
		&\&\,F\in H^0(\{\psi<-t\},\mathcal{O}(K_{\Omega}))\Bigg\}
	\end{split}
\end{displaymath}
by $\tilde{G}(t)$.
By the definition of $G(t)$ and $\tilde{G}(t)$, we know $G(t)=\tilde{G}(t)$ for any $t\geq0$, therefore $\tilde{G}(\tilde{h}^{-1}(r))$ is linear with respect to $r$.
Note that $(\tilde{\varphi},\psi)\in W$, $(\tilde{\varphi}+\psi-2G_{\Omega}(z,z_0),2G_{\Omega}(z,z_0))\in W$, $\psi_2(z_0)>-\infty$ and $\psi_2\leq0$, then Theorem \ref{thm:n2} shows that $\psi=2G_{\Omega}(z,z_0)$.

As $\varphi+\psi$ is subharmonic on $\Omega$ and $\psi=2G_{\Omega}(z,z_0)$, we have $\varphi$ is subharmonic on $\Omega$. Then Theorem \ref{thm:e2} implies that $u$ is harmonic on $\Omega$ and $\chi_{-u}=\chi_{z_0}$.

Thus, Theorem \ref{thm:e1} holds.

\subsection{Proof of Corollary \ref{c:ll3}}
The following remark  shows that it suffices to prove the existence  of holomorphic extension satisfying inequality \eqref{eq:210809} for the case $c(t)$ has a positive lower bound and upper bound on $(t',+\infty)$ for any $t'>0$.

\begin{Remark}\label{r:L2c2}
Take $c_j$ is a positive measurable function on $(0,+\infty)$, such that $c_{j}(t)=c(t)$ when $t<j$, $c_j(t)=\min\{c(j),\frac{1}{j}\}$ when $t\geq j$. It is clear that $c_j(t)e^{-t}$ is decreasing with respect to $t$, and $\int_{0}^{+\infty}c_j(t)e^{-t}<+\infty$. As
$$\lim_{j\rightarrow+\infty}\int_{j}^{+\infty}c_j(t)e^{-t}=0,$$
we have
$$\lim_{j\rightarrow+\infty}\int_{0}^{+\infty}c_j(t)e^{-t}=\int_{0}^{+\infty}c(t)e^{-t}.$$

	If the existence  of holomorphic extension satisfying inequality \eqref{eq:210809} holds in this case, then there exists a holomorphic $(1,0)$ form $F_j$ on $\Omega$ such that $F_j(z_0)=f(z_0)$ and
$$\int_{\Omega}|F_j|^2e^{-\varphi}c_j(-\psi)\leq\left(\int_0^{+\infty}c_j(t)e^{-t}dt\right)\|f\|_{z_0}.$$

Note that $\psi$ has locally lower bound on $\Omega\backslash\psi^{-1}(-\infty)$ and $\psi^{-1}(-\infty)$ is a closed subset of an analytic subset $Z$ of $\Omega$. For any compact subset $K$ of $\Omega\backslash Z$, there exists $s_K>0$ such that $\int_{K}e^{-s_K\psi}dV_{\Omega}<+\infty$, where $dV_{\Omega}$ is a continuous volume form on $\Omega$.  Then  we have
$$\int_{K}\left(\frac{e^{\varphi}}{c_j(-\psi)}\right)^{s_K}dV_{\Omega}=\int_{K}\left(\frac{e^{\varphi+\psi}}{c_j(-\psi)}\right)^{s_K}e^{-s_K\psi}dV_{\Omega}\leq C\int_{K}e^{-s_K\psi}dV_{\Omega}<+\infty,$$
where $C$ is a constant independent of $j$.
 It follows from Lemma \ref{l:converge} ($g_j=e^{-\varphi}c_j(-\psi)$) that there exists a subsequence of $\{F_j\}$, denoted still by $\{F_j\}$, which is uniformly convergent to a holomorphic $(1,0)$ form $F$ on any compact subset of $\Omega$ and
\begin{displaymath}
	\begin{split}
		\int_{\Omega}|F|^2e^{-\varphi}c(-\psi)&\leq\lim_{j\rightarrow+\infty}\left(\int_0^{+\infty}c_j(t)e^{-t}dt\right)\|f\|_{z_0}\\
		&=\left(\int_0^{+\infty}c(t)e^{-t}dt\right)\|f\|_{z_0}.	\end{split}
\end{displaymath}
 Since $F_j(z_0)=f(z_0)$ for any $j$, we have $F(z_0)=f(z_0)$.
\end{Remark}

As $\psi\in A(z_0)$ and $e^{-\varphi-\psi}$ is not $L^1$ on any neighborhood of $z_0$, it follows from Siu's Decomposition Theorem and the following Lemma that  $\psi(z)-2G_{\Omega}(z,z_0)$ and $\varphi(z)+\psi(z)-2G_{\Omega}(z,z_0)$ is subharmonic on $\Omega$ with respect to $z$. Denote that $\psi_2(z)=\psi(z)-2G_{\Omega}(z,z_0)$.

\begin{Lemma}[\cite{skoda1972}]
	Let $u$ is a subharmonic function on $\Omega$. If $v(dd^cu,z_0)<1$, then $e^{-u}$ is $L^1$ on a neighborhood of $z_0$.
\end{Lemma}
As $\Omega$ is a Stein manifold and $\varphi+\psi_2$ is subharmonic on $\Omega$, there exist smooth subharmonic functions $\Phi_l$ on $\Omega$, which are decreasingly convergent to $\varphi+\psi_2$.
We can find a sequence of open Riemann surfaces $\{D_m\}_{m=1}^{+\infty}$ satisfying $z_0\in D_m\subset\subset D_{m+1}$ for any $m$ and $\cup_{m=1}^{+\infty}D_m=\Omega$, and there is a holomorphic $(n,0)$ form $\tilde F$ on $\Omega$ such that $\tilde{F}(z_0)=f(z_0)$.

Note that $\int_{D_m}|\tilde F|^2<+\infty$ for any $m$ and
$$\int_{D_m}\mathbb{I}_{\{-t_0-1<\psi<-t_0\}}|\tilde{F}|^2e^{-\Phi_l-2G_{\Omega}(\cdot,z_0)}\leq e^{t_0+1}\int_{D_m}|\tilde{F}|^2e^{-\Phi_l+\psi_2}<+\infty$$
for any $m$, $l$ and $t_0>T$.
Using Lemma \ref{lem:GZ_sharp} ($\varphi\sim\Phi_l+2G_{\Omega}(\cdot,z_0)$), for any $D_m$, $l\in\mathbb{N}^+$ and $t_0>T$, there exists a holomorphic $(1,0)$ form $F_{l,m,t_0}$ on $D_m$, such that
\begin{equation}
\label{eq:210809d}
\begin{split}
&\int_{D_m}|F_{l,m,t_0}-(1-b_{t_0,1}(\psi))\tilde{F}|^{2}e^{-\Phi_l-2G_{\Omega}(\cdot,z_0)+v_{t_0,1}(\psi)}c(-v_{t_0,1}(\psi))\\
\leq& \left(\int_{0}^{t_{0}+1}c(t)e^{-t}dt\right) \int_{D_m}\mathbb{I}_{\{-t_0-1<\psi<-t_0\}}|\tilde{F}|^2e^{-\Phi_l-2G_{\Omega}(\cdot,z_0)}
\end{split}
\end{equation}
where
$b_{t_0,1}(t)=\int_{-\infty}^{t}\mathbb{I}_{\{-t_{0}-1< s<-t_{0}\}}ds$,
$v_{t_0,1}(t)=\int_{0}^{t}b_{t_0,1}(s)ds$. Note that $e^{-2G_{\Omega}(\cdot,z_0)}$ is not $L^1$ on any neighborhood of $z_0$, and $b_{t_0,1}(t)=0$ for  large enough $t$, then $(F_{l,m,t_0}-(1-b_{t_0,1}(\psi))\tilde{F})(z_0)=0$, and therefore $F_{l,m,t_0}(z_0)=f(z_0)$.

Note that $v_{t_0,1}(\psi)\geq\psi$ and $c(t)e^{-t}$ is decreasing, then the inequality \eqref{eq:210809d} becomes
\begin{equation}
\label{eq:210809e}
\begin{split}
&\int_{D_m}|F_{l,m,t_0}-(1-b_{t_0,1}(\psi))\tilde{F}|^{2}e^{-\Phi_l+\psi_2}c(-\psi)\\
\leq& \left(\int_{0}^{t_{0}+1}c(t)e^{-t}dt\right) \int_{D_m}\mathbb{I}_{\{-t_0-1<\psi<-t_0\}}|\tilde{F}|^2e^{-\Phi_l-2G_{\Omega}(\cdot,z_0)}.
\end{split}
\end{equation}
There exist smooth subharmonic functions $\Psi_k$ on $\Omega$, which are decreasingly convergent to $\psi_2$.
By definition of $dV_{\Omega}[\psi]$,  we have
\begin{equation}
	\label{eq:210809f}
	\begin{split}
	&\limsup_{t_0\rightarrow+\infty}\int_{D_m}\mathbb{I}_{\{-t_0-1<\psi<-t_0\}}|\tilde{F}|^2e^{-\Phi_l+\Psi_k-\psi}\\
	\leq&\pi\int_{z_0}\frac{|f|^2}{dV_{\Omega}}e^{-\Phi_l+\Psi_k}dV_{\Omega}[\psi]\\
	<&+\infty.	
	\end{split}
\end{equation}
Combining inequality \eqref{eq:210809e} and \eqref{eq:210809f}, let $t_0\rightarrow+\infty$, we have
\begin{equation}
	\label{eq:210809g}
	\begin{split}
		&\limsup_{t_0\rightarrow+\infty}\int_{D_m}|F_{l,m,t_0}-(1-b_{t_0,1}(\psi))\tilde{F}|^{2}e^{-\Phi_l+\psi_2}c(-\psi)\\
\leq& \limsup_{t_0\rightarrow+\infty}\left(\int_{0}^{t_{0}+1}c(t)e^{-t}dt\right) \int_{D_m}\mathbb{I}_{\{-t_0-1<\psi<-t_0\}}|\tilde{F}|^2e^{-\Phi_l+\Psi_k-\psi}\\
\leq&\pi\left(\int_{0}^{+\infty}c(t)e^{-t}dt\right) \int_{z_0}\frac{|f|^2}{dV_{\Omega}}e^{-\Phi_l+\Psi_k}dV_{\Omega}[\psi].
	\end{split}
\end{equation}
Let $k\rightarrow+\infty$, inequality \eqref{eq:210809g} implies that
\begin{equation}
	\label{eq:210809h}
	\begin{split}
		&\limsup_{t_0\rightarrow+\infty}\int_{D_m}|F_{l,m,t_0}-(1-b_{t_0,1}(\psi))\tilde{F}|^{2}e^{-\Phi_l+\psi_2}c(-\psi)\\
\leq&\pi\left(\int_{0}^{+\infty}c(t)e^{-t}dt\right) \int_{z_0}\frac{|f|^2}{dV_{\Omega}}e^{-\Phi_l+\psi_2}dV_{\Omega}[\psi].
	\end{split}
\end{equation}
Note that
$$\limsup_{t_0\rightarrow+\infty}\int_{D_m}|(1-b_{t_0,1}(\psi))\tilde{F}|^2e^{-\Phi_l+\psi_2}c(-\psi)<+\infty,$$
then we have
$$\limsup_{t_0\rightarrow+\infty}\int_{D_m}|F_{l,m,t_0}|^2e^{-\Phi_l+\psi_2}c(-\psi)<+\infty.$$
Using Lemma \ref{l:converge},
we obtain that
there exists a subsequence of $\{F_{l,m,t_0}\}_{t_0\rightarrow+\infty}$ (also denoted by $\{F_{l,m,t_0}\}_{t_0\rightarrow+\infty}$)
compactly convergent to a holomorphic (1,0) form on $D_m$ denoted by $F_{l,m}$. Then it follows from inequality \eqref{eq:210809h} and Fatou's Lemma that
\begin{equation}
\label{eq:210809i}
\begin{split}
\int_{D_m}|F_{l,m}|^{2}e^{-\Phi_l+\psi_2}c(-\psi)
=&\int_{D_m}\liminf_{t_0\rightarrow+\infty}|F_{l,m,t_0}-(1-b_{t_0,1}(\psi))\tilde{F}|^{2}e^{-\Phi_l+\psi_2}c(-\psi)\\
\leq&\liminf_{t_0\rightarrow+\infty}\int_{D_m}|F_{l,m,t_0}-(1-b_{t_0,1}(\psi))\tilde{F}|^{2}e^{-\Phi_l+\psi_2}c(-\psi)\\
\leq&\pi\left(\int_{0}^{+\infty}c(t)e^{-t}dt\right) \int_{z_0}\frac{|f|^2}{dV_{\Omega}}e^{-\Phi_l+\psi_2}dV_{\Omega}[\psi].
\end{split}
\end{equation}

As  $\|f\|_{z_0}=\pi\int_{z_0}\frac{|f|^2}{dV_{\Omega}}e^{-\varphi}dV_{\Omega}[\psi]<+\infty$ and $\Phi_l$  are decreasingly convergent to $\varphi+\psi_2$, we have
\begin{equation}
	\label{eq:210809j}
	\lim_{l\rightarrow+\infty}\pi\int_{z_0}\frac{|f|^2}{dV_{\Omega}}e^{-\Phi_l+\psi_2}dV_{\Omega}[\psi]=\|f\|_{z_0}<+\infty.
\end{equation}
It follows from inequality \eqref{eq:210809i} and \eqref{eq:210809j} that
\begin{equation}
	\label{eq:210809k}\limsup_{l\rightarrow+\infty}\int_{D_m}|F_{l,m}|^{2}e^{-\Phi_l+\psi_2}c(-\psi)\leq\left(\int_{0}^{+\infty}c(t)e^{-t}dt\right)\|f\|_{z_0}<+\infty.
\end{equation}
Using Lemma \ref{l:converge} ($g_l=e^{-\Phi_l+\psi_2}c(-\psi)$),
we obtain that
there exists a subsequence of $\{F_{l,m}\}_{l\rightarrow+\infty}$ (also denoted by $\{F_{l,m}\}_{l\rightarrow+\infty}$)
compactly convergent to a holomorphic (1,0) form on $D_m$ denoted by $F_{m}$ and
\begin{equation}
\label{eq:210809l}
\int_{D_m}|F_{m}|^{2}e^{-\varphi}c(-\psi)
\leq\left(\int_{0}^{+\infty}c(t)e^{-t}dt\right) \|f\|_{z_0}.
\end{equation}
 Inequality \eqref{eq:210809l} implies that
$$\int_{D_m}|F_{m'}|^{2}e^{-\varphi}c(-\psi)\leq \pi\left(\int_{0}^{+\infty}c(t)e^{-t}dt\right) \|f\|_{z_0}$$
holds for any $m'\geq m$. Note that $\varphi+\psi$ and $\psi$ are subharmonic on $\Omega$ and $\varphi=(\varphi+\psi)-\psi$.
 Using Lemma \ref{l:converge}, the diagonal method and Levi's Theorem, we obtain a subsequence of $\{F_m\}$, denoted also by $\{F_m\}$, which is uniformly convergent to a holomorphic $(1,0)$ form $F$ on $\Omega$ satisfying that $F(z_0)=f(z_0)$ and
 $$\int_{\Omega}|F|^{2}e^{-\varphi}c(-\psi)\leq \left(\int_{0}^{+\infty}c(t)e^{-t}dt\right)\|f\|_{z_0}$$
Thus the existence  of holomorphic extension satisfying inequality \eqref{eq:210809} holds.

In the following part, we prove the characterization for $\left(\int_{0}^{+\infty}c(t)e^{-t}dt\right)\|f\|_{z_0}=\inf\{\|\tilde F\|_{\Omega}:\tilde F$ is a holomorphic extension of $f$ from $z_0$ to $\Omega$$\}$.

Firstly, we prove the necessity. If $\|f\|_{z_0}=0$, then $F\equiv0$, which contradicts to $F(z_0)=f(z_0)\not=0$. Thus, we only consider the case $\|f\|_{z_0}\in(0,+\infty)$.

As $\{\psi<-t\}$ is an open Riemann surface. Note that $dV_{\Omega}[\psi+t]=e^{-t}dV_{\Omega}[\psi]$. By the above discussion ($\psi\sim\psi+t$, $c(\cdot)\sim c(\cdot+t)$ and $\Omega\sim\{\psi<-t\}$), for any $t>0$, there exists a holomorphic $(n,0)$ form $F_t$ on $\{\psi<-t\}$ such that $F_t(z_0)=f(z_0)$ and
$$\int_{\{\psi<-t\}}|F_t|^2e^{-\varphi}c(-\psi)\leq\left(\int_t^{+\infty}c(l)e^{-l}dl\right)\|f\|_{z_0}.$$
Let $\mathcal{F}|_{Z_0}=\mathcal{I}(\psi)_{z_0}$, by the definition of $G(t)$, we obtain that inequality
\begin{equation}
	\label{eq:210809m}
	\frac{G(t)}{\int_t^{+\infty}c(l)e^{-l}dl}\leq	\frac{G(0)}{\int_0^{+\infty}c(t)e^{-t}dt}\end{equation}
holds for any $t\geq0$.
 Theorem \ref{thm:general_concave} tells us $G(\hat{h}^{-1}(r))$ is concave with respect to $r$.
Combining inequality \eqref{eq:210809m} and Corollary \ref{thm:linear}, we obtain that $G(\hat{h}^{-1}(r))$ is linear with respect to $r$.  As $\psi-2G_{\Omega}(z,z_0)$ is bounded near $z_0$ and $G(0)=(\int_0^{+\infty}c(t)e^{-t}dt)\|f\|_{z_0}\in(0,+\infty)$,  Theorem \ref{thm:e1} shows that  statements $(1)$, $(2)$ and $(3)$  hold.

Now, we prove the sufficiency. Let $\mathcal{F}|_{Z_0}=\mathcal{I}(2G_{\Omega}(z,z_0))_{z_0}$, then Theorem \ref{thm:e2} shows that $G(\hat{h}^{-1}(r))$ is linear with respect to $r$. It follows from Lemma \ref{l:c'} and Corollary \ref{thm:1} that there exists $\tilde{c}\in\mathcal{P}_0$ such that $\tilde{c}(t)$ is increasing on $(a,+\infty)$ for some $a>0$, and $G(\hat{h}_{\tilde c}^{-1}(r);\tilde{c})$ is linear with respect to $r$, where $\hat{h}_{\tilde c}(t)=\int_{t}^{+\infty}\tilde{c}(l)e^{-l}dl$. Using Proposition \ref{p:G}, we have $G(0;\tilde{c})=\|f\|_{z_0}(\int_{0}^{+\infty}\tilde{c}(l)e^{-l}dl)$. Following from Corollary \ref{thm:1} and Remark \ref{r:c}, we obtain that $G(0;c)=\|f\|_{z_0}(\int_{0}^{+\infty}c(l)e^{-l}dl)$, which implies that $\|f\|_{z_0}(\int_{0}^{+\infty}c(l)e^{-l}dl)=\{\|\tilde F\|_{\Omega}:\tilde F$ is a holomorphic extension of $f$ from $z_0$ to $\Omega$$\}$.

Thus, Corollary \ref{c:ll3} holds.
\subsection{Proof of Corollary \ref{c:ll4}}
Note that $2G_{\Omega}(z,z_0)\in A'(z_0)$ and $\psi_1\in A'(z_0)$,  we have
\begin{equation}
	\label{eq:210730a}\|f\|^*_{z_0}=\pi\int_{z_0}\frac{|f|^2}{dV_{\Omega}}e^{-\varphi-\psi_2}dV_{\Omega}[\psi_1]=\pi\int_{z_0}\frac{|f|^2}{dV_{\Omega}}e^{-\varphi-\psi
	+2G_{\Omega}(z,z_0)}dV_{\Omega}[2G_{\Omega}(z,z_0)].\end{equation}
Corollary \ref{c:ll3} implies the  sufficiency. Thus, it suffices to prove the necessity.

Let $\mathcal{F}|_{Z_0}=\mathcal{I}(\psi_1)_{z_0}$. It follows from Lemma \ref{lem:A} that there exists a unique holomorphic extension from $z_0$ to $\Omega$, such that $\|F\|_{\Omega}\leq\|f\|_{z_0}^{*}(\int_{0}^{+\infty}c(l)e^{-l}dl)$. Using Corollary \ref{thm:ll2}, we know that $G(\hat{h}^{-1}(r))$ is linear with respect to $r$, therefore
\begin{equation}
	\label{eq:210627a}
	\frac{G(t)}{\int_{t}^{+\infty}c(l)e^{-l}dl}=\|f\|_{z_0}^{*}
	\end{equation}
holds for any $t\geq0$.

 Let $\tilde\psi=2G_{\Omega}(z,z_0)$. Lemma \ref{l:green} tells us $\psi-\tilde\psi\leq0$ on $\Omega$. Let $\tilde\varphi=\varphi+\psi-\tilde\psi$, then we compare $G(t;\varphi,\psi)$ and $G(t;\tilde\varphi,\tilde\psi)$ to prove $\psi-\tilde\psi\equiv0$. As $\|f\|_{z_0}^{*}<+\infty$ and $e^{-\tilde\varphi}c(-\tilde\psi)=e^{-\varphi-\psi}e^{\tilde\psi}c(-\tilde\psi)\geq e^{-\varphi}c(-\psi)$, it follows from Corollary \ref{c:L2b} and equality \eqref{eq:210730a} that $G(0;\tilde\varphi,\tilde\psi)\leq\|f\|^{*}_{z_0}\left(\int_{0}^{+\infty}c(t)e^{-t}dt\right)$. Without loss of generality, we can assume that $c(t)e^{-t}$ is strictly decreasing on $(0,+\infty)$. We prove $\psi-\tilde\psi\equiv0$ by contradiction: if not, $c(t)e^{-t}$ is strictly decreasing on $(0,+\infty)$ implies that $G(0;\tilde\varphi,\tilde\psi)>G(0;\varphi,\psi)$, which contradicts to $G(0;\tilde\varphi,\tilde\psi)\leq\|f\|^{*}_{z_0}\left(\int_{0}^{+\infty}c(t)e^{-t}dt\right)=G(0;\varphi,\psi)$. Thus, we have $\psi=2G_{\Omega}(z,z_0)$. Combining the linearity of $G(\hat{h}(r);\varphi,\psi)$, $G(0;\varphi,\psi)=\|f\|^{*}_{z_0}\left(\int_{0}^{+\infty}c(t)e^{-t}dt\right)\in(0,+\infty)$ and Theorem \ref{thm:e2}, we obtain that the other two statements in Corollary \ref{c:ll4} hold.

Thus, Corollary \ref{c:ll4} holds.

\section{Appendix}

\subsection{Proof of Lemma \ref{lem:GZ_sharp}}\label{sec:p-GZsharp}
In this section, we prove Lemma \ref{lem:GZ_sharp}.

It follows from Lemma \ref{l:FN1}
that there exist smooth strongly plurisubharmonic functions $\psi_{m}$ and $\varphi_{m}$ on $M$ decreasingly convergent to $\psi$ and $\varphi$ respectively.

The following remark shows that it suffices to consider Lemma \ref{lem:GZ_sharp} for the case that
$M$ is a relatively compact open Stein submanifold of a Stein manifold,
and $F$ is a holomorphic $(n,0)$ form on $\{\psi<-t_{0}\}$ such that $\int_{\{\psi<-t_{0}\}}|F|^{2}<+\infty$,
which implies that
$\sup_{m}\sup_{M}\psi_{m}<-T$ and $\sup_{m}\sup_{M}\varphi_{m}<+\infty$ on $M$.

\begin{Remark}
\label{rem:unify}
It is well-known that there exist open Stein submanifolds $D_{1}\subset\subset\cdots\subset\subset D_{j}\subset\subset D_{j+1}\subset\subset\cdots$
such that $\cup_{j=1}^{+\infty}D_{j}=M$.

If inequality \eqref{equ:3.4} holds on any $D_{j}$ and inequality \eqref{equ:20171122a} holds on $M$,
then for any $B>0$, we obtain a sequence of holomorphic (n,0) forms $\tilde{F}_{j}$ on $D_{j}$ such that
\begin{equation}
\label{eq:j1}
\begin{split}
&\int_{D_{j}}|\tilde{F}_{j}-(1-b_{t_0,B}(\psi))F|^{2}e^{-\varphi+v_{t_0,B}(\psi)}c(-v_{t_0,B}(\psi))
\\\leq&\int_{T}^{t_{0}+B}c(t)e^{-t}dt\int_{D_{j}}\frac{1}{B}\mathbb{I}_{\{-t_{0}-B<\psi<-t_{0}\}}|F|^{2}e^{-\varphi}\leq C\int_{T}^{t_{0}+B}c(t)e^{-t}dt
\end{split}
\end{equation}
is bounded with respect to $j$.
Note that for any given $j$, $e^{-\varphi+v_{t_0,B}(\psi)}c(-v_{t_0,B}(\psi))$ has a positive lower bound,
then it follows that for any any given $j$, $\int_{ D_{j}}|\tilde{F}_{j'}-(1-b_{t_0,B}(\psi))F|^{2}$ is bounded with respect to $j'\geq j$.
Combining with
\begin{equation}
\label{equ:20171123a}
\int_{ D_{j}}|(1-b_{t_0,B}(\psi))F|^{2}\leq
\int_{D_{j}\cap\{\psi<-t_{0}\}}|F|^{2}<+\infty
\end{equation}
and inequality \eqref{equ:3.4},
one can obtain that $\int_{ D_{j}}|\tilde{F}_{j'}|^{2}$ is bounded with respect to $j'\geq j$.

By the diagonal method, there exists a subsequence $F_{j''}$ uniformly convergent on any $D_{j}$ to a holomorphic $(n,0)$ form on $M$ denoted by $\tilde{F}$.
Then it follows from inequality \eqref{eq:j1} and Fatou's Lemma that

$$\int_{ D_{j}}|\tilde{F}-(1-b_{t_0,B}(\psi))F|^{2}e^{-(\varphi-v_{t_0,B}(\psi))}c(-v_{t_0,B}(\psi))\leq C\int_{T}^{t_{0}+B}c(t)e^{-t}dt,$$
then one can obtain Lemma \ref{lem:GZ_sharp} when $j$ goes to $+\infty$.
\end{Remark}

Next, we recall some lemmas on $L^{2}$ estimates for
some $\bar\partial$ equations.

\begin{Lemma}\label{l:vector7}(see \cite{demailly99,demailly2010})
Let $X$ be a complete K\"{a}hler manifold equipped with a (non
necessarily complete) K\"{a}hler metric $\omega$, and let $E$ be a
Hermitian vector bundle over $X$. Assume that there are smooth and
bounded functions $\eta$, $g>0$ on $X$ such that the (Hermitian)
curvature operator

$$\textbf{B}:=[\eta\sqrt{-1}\Theta_{E}-\sqrt{-1}\partial\bar\partial\eta-
\sqrt{-1}g\partial\eta\wedge\bar\partial\eta,\Lambda_{\omega}]$$
is positive definite everywhere on $\Lambda^{n,q}T^{*}_{X}\otimes E$, for some $q\geq 1$.
Then for every form $\lambda\in L^{2}(X,\Lambda^{n,q}T^{*}_{X}\otimes E)$ such that $D''\lambda=0$ and
$\int_{X}\langle\textbf{B}^{-1}\lambda,\lambda\rangle dV_{M}<\infty$,
there exists $u\in L^{2}(X,\Lambda^{n,q-1}T^{*}_{X}\otimes E)$ such that $D''u=\lambda$ and
$$\int_{X}(\eta+g^{-1})^{-1}|u|^{2}dV_{M}\leq\int_{X}\langle\textbf{B}^{-1}\lambda,\lambda\rangle dV_{M}.$$
\end{Lemma}

\begin{Lemma}
\label{l:positve}(see \cite{guan-zhou13ap})Let $X$ and $E$ be as in the above lemma and
$\theta$ be a continuous $(1,0)$ form on $X$. Then we have
$$[\sqrt{-1}\theta\wedge\bar\theta,\Lambda_{\omega}]\alpha=\bar\theta\wedge(\alpha\llcorner(\bar\theta)^\sharp\big),$$
for any $(n,1)$ form $\alpha$ with value in $E$. Moreover, for any
positive $(1,1)$ form $\beta$, we have $[\beta,\Lambda_{\omega}]$ is
semipositive.
\end{Lemma}

The following Lemma belongs to Fornaess
and Narasimhan on approximation property of plurisubharmonic
functions of Stein manifolds.

\begin{Lemma}
\label{l:FN1}\cite{FN1980} Let $X$ be a Stein manifold and $\varphi \in PSH(X)$. Then there exists a sequence
$\{\varphi_{n}\}_{n=1,2,\cdots}$ of smooth strongly plurisubharmonic functions such that
$\varphi_{n} \downarrow \varphi$.
\end{Lemma}

For the sake of completeness, let's recall some steps in the proof in
\cite{guan_sharp} (see also \cite{guan-zhou13p,guan-zhou13ap,GZopen-effect}) with some slight
modifications in order to prove Lemma \ref{lem:GZ_sharp}.

It follows from Remark \ref{rem:unify} that
it suffices to consider that $M$ is a Stein manifold, and $F$ is holomorphic $(n,0)$ form on $\{\psi<-t_{0}\}$ and
\begin{equation}
\label{equ:20171122e}
\int_{\{\psi<-t_{0}\}}|F|^{2}<+\infty,
\end{equation}
and there exist smooth plurisubharmonic functions $\psi_{m}$ and $\varphi_{m}$ on $M$ decreasingly convergent to $\psi$ and $\varphi$ respectively,
satisfying $\sup_{m}\sup_{M}\psi_{m}<-T$ and $\sup_{m}\sup_{M}\varphi_{m}<+\infty$.

\

\emph{Step 1: construct some functions}

\

Let $\varepsilon\in(0,\frac{1}{8}B)$.
Let $\{v_{\varepsilon}\}_{\varepsilon\in(0,\frac{1}{8}B)}$ be a family of smooth increasing convex functions on $\mathbb{R}$,
which are continuous functions on $\mathbb{R}\cup\{-\infty\}$, such that:

 $(1)$ $v_{\varepsilon}(t)=t$ for $t\geq-t_{0}-\varepsilon$, $v_{\varepsilon}(t)=constant$ for $t<-t_{0}-B+\varepsilon$ and are pointwise convergent to $v_{t_0,B}(t)$;

 $(2)$ $v''_{\varepsilon}(t)$ are pointwise convergent to $\frac{1}{B}\mathbb{I}_{(-t_{0}-B,-t_{0})}$, when $\varepsilon\to 0$,
 and $0\leq v''_{\varepsilon}(t)\leq \frac{2}{B}\mathbb{I}_{(-t_{0}-B+\varepsilon,-t_{0}-\varepsilon)}$ for any $t\in \mathbb{R}$;

 $(3)$ $v'_{\varepsilon}(t)$ are pointwise convergent to $b_{t_0,B}(t)$ which is a continuous function on $\mathbb{R}$, when $\varepsilon\to 0$, and $0\leq v'_{\varepsilon}(t)\leq1$ for any $t\in \mathbb{R}$.

One can construct the family $\{v_{\varepsilon}\}_{\varepsilon\in(0,\frac{1}{8}B)}$ by the setting
\begin{equation}
\label{equ:20140101}
\begin{split}
v_{\varepsilon}(t):=&\int_{-\infty}^{t}\left(\int_{-\infty}^{t_{1}}\left(\frac{1}{B-4\varepsilon}
\mathbb{I}_{(-t_{0}-B+2\varepsilon,-t_{0}-2\varepsilon)}*\rho_{\frac{1}{4}\varepsilon}\right)(s)ds\right)dt_{1}
\\&-\int_{-\infty}^{-t_0}\left(\int_{-\infty}^{t_{1}}\left(\frac{1}{B-4\varepsilon}\mathbb{I}_{(-t_{0}-B+2\varepsilon,
-t_{0}-2\varepsilon)}*\rho_{\frac{1}{4}\varepsilon}\right)(s)ds\right)dt_{1}-t_0,
\end{split}
\end{equation}
where $\rho_{\frac{1}{4}\varepsilon}$ is the kernel of convolution satisfying $supp(\rho_{\frac{1}{4}\varepsilon})\subset (-\frac{1}{4}\varepsilon,\frac{1}{4}\varepsilon)$.
Then it follows that
$$v''_{\varepsilon}(t)=\frac{1}{B-4\varepsilon}\mathbb{I}_{(-t_{0}-B+2\varepsilon,-t_{0}-2\varepsilon)}*\rho_{\frac{1}{4}\varepsilon}(t),$$
and
$$v'_{\varepsilon}(t)=\int_{-\infty}^{t}\left(\frac{1}{B-4\varepsilon}\mathbb{I}_{(-t_{0}-B+2\varepsilon,-t_{0}-2\varepsilon)}
*\rho_{\frac{1}{4}\varepsilon}\right)(s)ds.$$

Let $\eta=s(-v_{\varepsilon}(\psi_{m}))$ and $\phi=u(-v_{\varepsilon}(\psi_{m}))$,
where $s\in C^{\infty}((S,+\infty))$ satisfies $s>0$ and $s'>0$, and
$u\in C^{\infty}((S,+\infty))$, such that $u''s-s''>0$, and $s'-u's=1$.
It follows from $\sup_{m}\sup_{M}\psi_{m}<-S$ and $\max{\{t,-t_0-B\}}\leq v_{\epsilon}(t)\leq\max{\{t,t_0\}}$ that $\phi=u(-v_{\varepsilon}(\psi_{m}))$ are uniformly bounded
on $M$ with respect to $m$ and $\varepsilon$,
and $u(-v_{\varepsilon}(\psi))$ are uniformly bounded
on $M$ with respect to $\varepsilon$.
Let $\Phi=\phi+\varphi_{m'}$,
and let $\tilde{h}=e^{-\Phi}$.

Let $f(x)=2\mathbb{I}_{(-\frac{1}{2},\frac{1}{2})}*\rho(x)$ be a smooth function on $\mathbb{R}$, where $\rho$ is is the kernel of convolution satisfying $supp(\rho)\subset (-\frac{1}{3},\frac{1}{3})$ and $\rho\geq0$.

Let $g_{l}(x)=\left\{ \begin{array}{lcl}
	lf(lx) & \mbox{if}& x\leq0\\
	lf(l^2x) & \mbox{if}& x>0
    \end{array} \right.$, then $\{g_{l}\}_{l\in\mathbb{N}^+}$ be a family of smooth functions on $\mathbb R$ satisfying that:

$(1)$ $supp(g_l)\subset[-\frac{1}{l},\frac{1}{l}],$ $g_{l}(x)\geq 0$ for any $x\in\mathbb R$;

$(2)$ $\int_{-\frac{1}{l}}^0g_{l}(x)dx=1$, $\int_{0}^{\frac{1}{l}}g_{l}(x)dx\leq\frac{1}{l}$ for any $l\in\mathbb{N}^+$.

Set $c_{l}(t)=e^t\int_\mathbb{R}h(e^y(t-S)+S)g_{l}(y)dy$, where $h(t)=e^{-t}c(t)$ and $c\in\tilde{\mathcal P}_S$. It is easy to get
\begin{displaymath}
	c_l(t)-c(t)\geq e^t\int_{-\frac{1}{l}}^0(h(e^y(t-S)+S)-h(t))g_l(y)dy\geq0.
\end{displaymath}
 Set $\tilde h(t)=h(e^t+S)$ and $\tilde g_l(t)=g_l(-t)$, then $c_l(t)=e^t\tilde h*\tilde g_l(\ln(t-S))\in C^{\infty}(S,+\infty)$. Because $h(t)$ is decreasing with respect to $t$, so is $c_{l}(t)e^{-t}$. And
\begin{displaymath}
	\begin{split}
		\int_S^{s}c_l(t)e^{-t}dt &=	\int_S^{s}\int_{\mathbb{R}}h(e^y(t-S)+S)g_{l}(y)dydt\\
&= \int_{\mathbb{R}}e^{-y}g_l(y)\int_S^{e^y(s-S)+S}h(t)dtdy\\
&\leq\int_{\mathbb{R}}e^{-y}g_l(y)dy\int_S^{e(s-S)+S}h(t)dt\\
&<+\infty,
	\end{split}	
\end{displaymath}
then $c_l(t)\in\tilde{\mathcal P}_S$ for any $l\in\mathbb{N}^+$.

As $h(t)$ is decreasing with respect to $t$, then set $h^-(t)=\lim_{s\rightarrow t-0}h(s)\geq h(t)$ and $c^-(t)=\lim_{s\rightarrow t-0}c(s)\geq c(t)$, then we  claim that $\lim_{l\rightarrow +\infty}c_l(t)=c^-(t)$. In fact, we have
\begin{equation}
	\label{eq:c1}\begin{split}
		|c_l(t)-c^-(t)|\leq& e^t\int_{-\frac{1}{l}}^0|h(e^y(t-S)+S)-h^-(t)|g_l(y)dy\\&+e^t\int_0^{\frac{1}{l}}h(e^y(t-S)+S)g_l(y)dy.		
	\end{split}\end{equation}

$\forall \varepsilon>0$, $\exists \delta>0$ and $|h(t-\delta)-h^-(t)|<\varepsilon$. Then $\exists N>0$, $\forall l>N$, such that $e^y(t-S)+S>t-\delta$ for all $y\in[-\frac{1}{l},0)$ and $\frac{1}{l}<\varepsilon$. It following from \eqref{eq:c1} that
\begin{displaymath}
	|c_l(t)-c^-(t)|\leq \varepsilon e^t+\varepsilon h(t)e^t,
	\end{displaymath}
hence $\lim_{l\rightarrow +\infty}c_l(t)=c^-(t)$ for any $t>S$.

\

\emph{Step 2: Solving $\bar\partial-$equation with smooth polar function and smooth weight}

\

Now let $\alpha\in \Lambda_x^{n,1}T_{M}^{*}$, for any $x\in M$. Using inequality $s>0$ and the fact that $\varphi_{m}$ is
plurisubharmonic on $M$, we get
\begin{equation}
\label{equ:10.1}
\begin{split}
\langle\textbf{B}\alpha,\alpha\rangle_{\tilde{h}}
=&\langle[\eta\sqrt{-1}\Theta_{\tilde{h}}-\sqrt{-1}\partial\bar\partial\eta-
\sqrt{-1}g\partial\eta\wedge\bar\partial\eta,\Lambda_{\omega}]\alpha,\alpha\rangle_{\tilde{h}}
\\\geq&\langle[\eta\sqrt{-1}\partial\bar\partial\phi-\sqrt{-1}\partial\bar\partial\eta-
\sqrt{-1}g\partial\eta\wedge\bar\partial\eta,\Lambda_{\omega}]\alpha,\alpha\rangle_{\tilde{h}}.
\end{split}
\end{equation}
where $g>0$ is a smooth and bounded function on $M$.
We need the following calculations to determine $g$.
\begin{equation}
\label{}
\begin{split}
&\partial\bar{\partial}\eta=-s'(-v_{\varepsilon}(\psi_{m}))\partial\bar{\partial}(v_{\varepsilon}(\psi_{m}))
+s''(-v_{\varepsilon}(\psi_{m}))\partial v_{\varepsilon}(\psi_{m})\wedge
\bar{\partial}v_{\varepsilon}(\psi_{m}),
\end{split}
\end{equation}

and
\begin{equation}
\label{}
\begin{split}
&\partial\bar{\partial}\phi=-u'(-v_{\varepsilon}(\psi_{m}))\partial\bar{\partial}v_{\varepsilon}(\psi_{m})
+
u''(-v_{\varepsilon}(\psi_{m}))\partial v_{\varepsilon}(\psi_{m})\wedge\bar{\partial}v_{\varepsilon}(\psi_{m}).
\end{split}
\end{equation}
Then we have
\begin{equation}
\label{equ:vector1}
\begin{split}
&-\partial\bar{\partial}\eta+\eta\partial\bar{\partial}\phi-g(\partial\eta)\wedge\bar\partial\eta
\\=&(s'-su')\partial\bar{\partial}v_{\varepsilon}(\psi_{m})+((u''s-s'')-gs'^{2})\partial
(-v_{\varepsilon}(\psi_{m}))\bar{\partial}(-v_{\varepsilon}(\psi_{m}))
\\=&(s'-su')(v'_{\varepsilon}(\psi_{m})\partial\bar{\partial}\psi_{m}+v''_{\varepsilon}(\psi_{m})
\partial(\psi_{m})\wedge\bar{\partial}(\psi_{m}))
\\&+((u''s-s'')-gs'^{2})\partial
(-v_{\varepsilon}(\psi_{m}))\wedge\bar{\partial}(-v_{\varepsilon}(\psi_{m})).
\end{split}
\end{equation}
We omit composite item $-v_{\varepsilon}(\psi_{m})$ after $s'-su'$ and $(u''s-s'')-gs'^{2}$ in the above equalities. Let $g=\frac{u''s-s''}{s'^{2}}(-v_{\varepsilon}(\psi_{m}))$.
It follows that $\eta+g^{-1}=(s+\frac{s'^{2}}{u''s-s''})(-v_{\varepsilon}(\psi_{m}))$.

As $v'_{\varepsilon}\geq 0$  and $s'-su'=1$, using Lemma
\ref{l:positve}, equality \eqref{equ:vector1} and inequality
\eqref{equ:10.1}, we obtain
\begin{equation}
\label{equ:semi.vector3}
\begin{split}
\langle\textbf{B}\alpha, \alpha\rangle_{\tilde{h}}=&\langle[\eta\sqrt{-1}\Theta_{\tilde{h}}-\sqrt{-1}\partial\bar\partial
\eta-\sqrt{-1}g\partial\eta\wedge\bar\partial\eta,\Lambda_{\omega}]
\alpha,\alpha\rangle_{\tilde{h}}
\\\geq&
\langle[(v''_{\varepsilon}\circ\psi_{m})
\sqrt{-1}\partial\psi_{m}\wedge\bar{\partial}\psi_{m},\Lambda_{\omega}]\alpha,\alpha\rangle_{\tilde{h}}
\\=&\langle (v''_{\varepsilon}\circ\psi_{m}) \bar\partial\psi_{m}\wedge
(\alpha\llcorner(\bar\partial\psi_{m})^\sharp\big ),\alpha\rangle_{\tilde{h}}.
\end{split}
\end{equation}

Using the definition of contraction, Cauchy-Schwarz inequality and
the inequality \eqref{equ:semi.vector3}, we have
\begin{equation}
\label{equ:20181106}
\begin{split}
|\langle (v''_{\varepsilon}\circ\psi_{m})\bar\partial\psi_{m}\wedge \gamma,\tilde\alpha\rangle_{\tilde{h}}|^{2}
=&|\langle (v''_{\varepsilon}\circ\psi_{m}) \gamma,\tilde\alpha\llcorner(\bar\partial\psi_{m})^\sharp\big
\rangle_{\tilde{h}}|^{2}
\\\leq&\langle( v''_{\varepsilon}\circ\psi_{m}) \gamma,\gamma\rangle_{\tilde{h}}
(v''_{\varepsilon}\circ\psi_{m})|\tilde\alpha\llcorner(\bar\partial\psi_{m})^\sharp\big|_{\tilde{h}}^{2}
\\=&\langle (v''_{\varepsilon}\circ\psi_{m}) \gamma,\gamma\rangle_{\tilde{h}}
\langle (v''_{\varepsilon}\circ\psi_{m}) \bar\partial\psi_{m}\wedge
(\tilde\alpha\llcorner(\bar\partial\psi_{m})^\sharp\big ),\alpha\rangle_{\tilde{h}}
\\\leq&\langle (v''_{\varepsilon}\circ\psi_{m})\gamma,\gamma\rangle_{\tilde{h}}
\langle\textbf{B}\tilde\alpha,\tilde\alpha\rangle_{\tilde{h}},
\end{split}
\end{equation}
for any $(n,0)$ form $\gamma$.

It follows from $s>0$ and $\varphi_{m'}$ is strongly plurisubharmonic that $\textbf{B}$ is positive definite everywhere on $\Lambda^{n,1}T_{M}^{*}$.
As $F$ is holomorphic on $\{\psi<-t_{0}\}$ and $ Supp(v''_{\varepsilon}(\psi_{m}))\subset\{\psi<-t_0\}$,
then $\lambda:=\bar{\partial}[(1-v'_{\varepsilon}(\psi_{m})){F}]$
is well-defined and smooth on $M$.

Taking $\gamma=F$, and
$\tilde\alpha=\textbf{B}^{-1}\lambda$,
note that $\tilde{h}=e^{-\Phi}$,
using inequality \eqref{equ:20181106},
we have
$$\langle \textbf{B}^{-1}\lambda,\lambda\rangle_{\tilde{h}} \leq v''_{t_0,\varepsilon}(\psi_{m})| \tilde{F}|^{2}e^{-\Phi}.$$
Then it follows that
$$\int_{M}\langle \textbf{B}^{-1}\lambda,\lambda\rangle_{\tilde{h}}
 \leq \int_{M}v''_{t_0,\varepsilon}(\psi_{m})| \tilde{F}|^{2}e^{-\Phi}.$$
Assume that we can choose $\eta$ and $\phi$ such that $e^{v_{\varepsilon}\circ\psi_{m}}e^{\phi}c_l(-v_{\varepsilon}\circ\psi_{m})=(\eta+g^{-1})^{-1}$. Using Lemma \ref{l:vector7},
we have locally $L^{1}$ function $u_{m,m',\varepsilon,l}$ on $M$ such that $\bar{\partial}u_{m,m',\varepsilon,l}=\lambda$,
and
\begin{equation}
 \label{equ:3.2}
 \begin{split}
 &\int_{M}|u_{m,m',\varepsilon,l}|^{2}e^{v_{\varepsilon}(\psi_{m})-\varphi_{m'}}c_l(-v_{\varepsilon}\circ\psi_{m})\\=&\int_{M}|u_{m,m',\varepsilon,l}|^{2}(\eta+g^{-1})^{-1} e^{-\Phi}\\ \leq&\int_{M}\langle \textbf{B}^{-1}\lambda,\lambda\rangle_{\tilde{h}}\\
  \leq&\int_{M}v''_{\varepsilon}(\psi_{m})| F|^2e^{-\Phi}\\=&\int_{M}v''_{\varepsilon}(\psi_{m})| F|^2e^{-\phi-\varphi_{m'}}.
  \end{split}
\end{equation}

Let $F_{m,m',\varepsilon,l}:=-u_{m,m',\varepsilon,l}+(1-v'_{\varepsilon}(\psi_{m})){F}$.
Then inequality \eqref{equ:3.2} becomes
\begin{equation}
 \label{equ:20171122c}
 \begin{split}
 &\int_{M}|F_{m,m',\varepsilon,l}-(1-v'_{\varepsilon}(\psi_{m})){F}|^{2}e^{v_{\varepsilon}(\psi_{m})-\varphi_{m'}}c_l(-v_{\varepsilon}\circ\psi_{m})
  \\&\leq\int_{M}(v''_{\varepsilon}(\psi_{m}))| F|^2e^{-\phi-\varphi_{m'}}.
  \end{split}
\end{equation}

\

\emph{Step 3: Singular polar function and smooth weight}

\

As $\sup_{m,\varepsilon}\sup_{M}|\phi|=\sup_{m,\varepsilon}|u(-v_{\varepsilon}(\psi_{m}))|<+\infty$ and $\sup_{M}\varphi_{m'}<+\infty$,
note that
$$v''_{\varepsilon}(\psi_{m})| F|^2e^{-\phi-\varphi_{m'}}\leq\frac{2}{B}\mathbb{I}_{\{\psi<-t_{0}\}}| F|^{2}\sup_{m,\varepsilon}e^{-\phi-\varphi_{m'}}$$
on $M$,
then it follows from inequality \eqref{equ:20171122e} and the dominated convergence theorem that
\begin{equation}
\label{equ:20171122f}
 \lim_{m\to+\infty}\int_{M}v''_{\varepsilon}(\psi_{m})| F|^2e^{-\phi-\varphi_{m'}}=
 \int_{M}v''_{\varepsilon}(\psi)| F|^2e^{-u(-v_{\varepsilon}(\psi))-\varphi_{m'}}
\end{equation}

Note that $\inf_{m}\inf_{M}e^{v_{\varepsilon}(\psi_{m})-\varphi_{m'}}c_l(-v_{\varepsilon}\circ\psi_{m})>0$,
then it follows from inequality \eqref{equ:20171122c} and \eqref{equ:20171122f}
that $\sup_{m}\int_{M}|F_{m,m',\varepsilon,l}-(1-v'_{\varepsilon}(\psi_{m})){F}|^{2}<+\infty$.
Note that
\begin{equation}
\label{equ:20171122g}
|(1-v'_{\varepsilon}(\psi_{m}))F|\leq |\mathbb{I}_{\{\psi<-t_{0}\}}F|,
\end{equation}
then it follows from inequality \eqref{equ:20171122e}
that $\sup_{m}\int_{M}|F_{m,m',\varepsilon,l}|^{2}<+\infty$,
which implies that there exists a subsequence of $\{F_{m,m',\varepsilon,l}\}_{m}$
(also denoted by $\{F_{m,m',\varepsilon,l}\}_m$) compactly convergent to a holomorphic $F_{m',\varepsilon,l}$ on $M$.

Note that $e^{v_{\varepsilon}(\psi_{m})-\varphi_{m'}}c_l(-v_{\varepsilon}\circ\psi_{m})$ are uniformly bounded on $M$ with respect to $m$,
then it follows from
$|F_{m,m',\varepsilon,l}-(1-v'_{\varepsilon}(\psi_{m})){F}|^{2}\leq
 2(|F_{m,m',\varepsilon,l}|^{2}+ |(1-v'_{\varepsilon}(\psi_{m})){F}|^{2})
\leq  2(|F_{m,m',\varepsilon,l}|^{2}+  |\mathbb{I}_{\{\psi<-t_{0}\}}F^{2}|)$
and the dominated convergence theorem that
\begin{equation}
 \label{equ:20171122d}
 \begin{split}
 \lim_{m\to+\infty}&\int_{K}|F_{m,m',\varepsilon,l}-(1-v'_{\varepsilon}(\psi_{m})){F}|^{2}e^{v_{\varepsilon}(\psi_{m})-\varphi_{m'}}c_l(-v_{\varepsilon}\circ\psi_{m})
  \\=&\int_{K}|F_{m',\varepsilon,l}-(1-v'_{\varepsilon}(\psi)){F}|^{2}e^{v_{\varepsilon}(\psi)-\varphi_{m'}}c_l(-v_{\varepsilon}\circ\psi)
  \end{split}
\end{equation}
holds for any compact subset $K$ on $M$.
Combining with inequality \eqref{equ:20171122c} and \eqref{equ:20171122f},
one can obtain that
\begin{equation}
 \label{equ:20171122h}
 \begin{split}
&\int_{K}|F_{m',\varepsilon,l}-(1-v'_{\varepsilon}(\psi)){F}|^{2}e^{v_{\varepsilon}(\psi)-\varphi_{m'}}c_l(-v_{\varepsilon}\circ\psi)
\\&\leq
\int_{M}v''_{\varepsilon}(\psi)| F|^2e^{-u(-v_{\varepsilon}(\psi))-\varphi_{m'}},
\end{split}
\end{equation}
which implies
\begin{equation}
 \label{equ:20171122i}
 \begin{split}
&\int_{M}|F_{m',\varepsilon,l}-(1-v'_{\varepsilon}(\psi)){F}|^{2}e^{v_{\varepsilon}(\psi)-\varphi_{m'}}c_l(-v_{\varepsilon}\circ\psi)
\\&\leq
\int_{M}v''_{\varepsilon}(\psi)| F|^2e^{-u(-v_{\varepsilon}(\psi))-\varphi_{m'}}.
\end{split}
\end{equation}

\

\emph{Step 4: Nonsmooth cut-off function}

\

Note that
$\sup_{\varepsilon}\sup_{M}e^{-u(-v_{\varepsilon}(\psi))-\varphi_{m'}}<+\infty,$
and
$$v''_{\varepsilon}(\psi)| F|^2e^{-u(-v_{\varepsilon}(\psi))-\varphi_{m'}}\leq
\frac{2}{B}\mathbb{I}_{\{-t_{0}-B<\psi<-t_{0}\}}| F|^2\sup_{\varepsilon}\sup_{M}e^{-u(-v_{\varepsilon}(\psi))-\varphi_{m'}},$$
then it follows from inequality \eqref{equ:20171122e} and the dominated convergence theorem that
\begin{equation}
\label{equ:20171122j}
\begin{split}
&\lim_{\varepsilon\to0}\int_{M}v''_{\varepsilon}(\psi)| F|^2e^{-u(-v_{\varepsilon}(\psi))-\varphi_{m'}}
\\=&\int_{M}\frac{1}{B}\mathbb{I}_{\{-t_{0}-B<\psi<-t_{0}\}}|F|^2e^{-u(-v_{t_0,B}(\psi))-\varphi_{m'}}
\\\leq&(\sup_{M}e^{-u(-v_{t_0,B}(\psi))})\int_{M}\frac{1}{B}\mathbb{I}_{\{-t_{0}-B<\psi<-t_{0}\}}|F|^2e^{-\varphi_{m'}}<+\infty.
\end{split}
\end{equation}

Note that $\inf_{\varepsilon}\inf_{M}e^{v_{\varepsilon}(\psi)-\varphi_{m'}}c_l(-v_{\varepsilon}\circ\psi)>0$,
then it follows from inequality \eqref{equ:20171122i} and \eqref{equ:20171122j} that
$\sup_{\varepsilon}\int_{M}|F_{m',\varepsilon,l}-(1-v'_{\varepsilon}(\psi)){F}|^{2}<+\infty.$
Combining with
\begin{equation}
\label{equ:20171122k}
\sup_{\varepsilon}\int_{M}|(1-v'_{\varepsilon}(\psi)){F}|^{2}\leq\int_{M}\mathbb{I}_{\{\psi<-t_{0}\}}|F|^{2}<+\infty,
\end{equation}
one can obtain that $\sup_{\varepsilon}\int_{M}|F_{m',\varepsilon,l}|^{2}<+\infty$,
which implies that
there exists a subsequence of $\{F_{m',\varepsilon,l}\}_{\varepsilon\to0}$ (also denoted by $\{F_{m',\varepsilon,l}\}_{\varepsilon\to0}$)
compactly convergent to a holomorphic (n,0) form on $M$ denoted by $F_{m',l}$. Then it follows from inequality \eqref{equ:20171122i}, inequality \eqref{equ:20171122j} and Fatou's Lemma that
\begin{equation}
\label{equ:20171122n}
\begin{split}
&\int_{M}|F_{m',l}-(1-b_{t_0,B}(\psi)){F}|^{2}e^{v_{t_0,B}(\psi)-\varphi_{m'}}c_l(-v\circ\psi)\\
=&\int_{M}\liminf_{\varepsilon\rightarrow0}|F_{m',\varepsilon,l}-(1-v'_{\varepsilon}(\psi)){F}|^{2}e^{v_{\varepsilon}(\psi)-\varphi_{m'}}c_l(-v_{\varepsilon}\circ\psi)\\
\leq&\liminf_{\varepsilon\rightarrow0}\int_{M}|F_{m',\varepsilon,l}-(1-v'_{\varepsilon}(\psi)){F}|^{2}e^{v_{\varepsilon}(\psi)-\varphi_{m'}}c_l(-v_{\varepsilon}\circ\psi)\\
\leq&\liminf_{\varepsilon\rightarrow0}\int_{M}v''_{\varepsilon}(\psi)| F|^2e^{-u(-v_{\varepsilon}(\psi))-\varphi_{m'}}
\\\leq&(\sup_{M}e^{-u(-v_{t_0,B}(\psi))})\int_{M}\frac{1}{B}\mathbb{I}_{\{-t_{0}-B<\psi<-t_{0}\}}|F|^2e^{-\varphi_{m'}}.
\end{split}
\end{equation}

\

\emph{Step 5: Singular weight}

\

Note that
\begin{equation}
\label{equ:20171122o}
\int_{M}\frac{1}{B}\mathbb{I}_{\{-t_{0}-B<\psi<-t_{0}\}}|F|^2e^{-\varphi_{m'}}\leq\int_{M}\frac{1}{B}\mathbb{I}_{\{-t_{0}-B<\psi<-t_{0}\}}|F|^{2}e^{-\varphi}<+\infty,
\end{equation}
and $\sup_{M}e^{-u(-v_{t_0,B}(\psi))}<+\infty$,
then it from \eqref{equ:20171122n} that
$$\sup_{m'}\int_{M}|F_{m',l}-(1-b(\psi)){F}|^{2}e^{v(\psi)-\varphi_{m'}}c_l(-v\circ\psi)<+\infty.$$
Combining with $\inf_{m'}\inf_{M}e^{v(\psi)-\varphi_{m'}}c_l(-v(\psi))>0$,
one can obtain that
$$\sup_{m'}\int_{M}|F_{m',l}-(1-b(\psi)){F}|^{2}<+\infty.$$
Note that
\begin{equation}
\label{equ:20171122p}
\int_{M}|(1-b(\psi)){F}|^{2}\leq\int_{M}|\mathbb{I}_{\{\psi<-t_{0}\}}F|^{2} <+\infty.
\end{equation}
Then $\sup_{m'}\int_{M}|F_{m',l}|^{2}<+\infty$,
which implies that there exists a compactly convergent subsequence of $\{F_{m',l}\}_{m'}$ (also denoted by $\{F_{m',l}\}_{m'}$),
which converges to a holomorphic (n,0) form $\tilde{F_l}$ on $M$.
Then it follows from inequality \eqref{equ:20171122n}, inequality \eqref{equ:20171122o} and Fatou's Lemma that
\begin{equation}
\label{equ:20171122s}
\begin{split}
&\int_{M}|F_{l}-(1-b_{t_0,B}(\psi)){F}|^{2}e^{v_{t_0,B}(\psi)-\varphi}c_l(-v_{t_0,B}\circ\psi)\\
=&\int_{M}\liminf_{m'\rightarrow+\infty}|F_{m',l}-(1-b_{t_0,B}(\psi)){F}|^{2}e^{v_{t_0,B}(\psi)-\varphi_{m'}}c_l(-v_{t_0,B}\circ\psi)\\
\leq&\liminf_{m'\rightarrow+\infty}\int_{M}|F_{m',l}-(1-b_{t_0,B}(\psi)){F}|^{2}e^{v_{t_0,B}(\psi)-\varphi_{m'}}c_l(-v_{t_0,B}\circ\psi)\\
\leq&\liminf_{m'\rightarrow+\infty}(\sup_{M}e^{-u(-v_{t_0,B}(\psi))})\int_{M}\frac{1}{B}\mathbb{I}_{\{-t_{0}-B<\psi<-t_{0}\}}|F|^2e^{-\varphi_{m'}}
\\\leq&(\sup_{M}e^{-u(-v_{t_0,B}(\psi))})\int_{M}\frac{1}{B}\mathbb{I}_{\{-t_{0}-B<\psi<-t_{0}\}}|F|^2e^{-\varphi}.
\end{split}
\end{equation}

\

\emph{Step 6: ODE system}

\

we need to find $\eta$ and $\phi$ such that
$(\eta+g^{-1})=e^{-\psi_{m}}e^{-\phi}\frac{1}{c_l(-v_{\varepsilon}(\psi_{m}))}$ on $M$ and $s'-u's=1$.
As $\eta=s(-v_{\varepsilon}(\psi_{m}))$ and $\phi=u(-v_{\varepsilon}(\psi_{m}))$,
we have $(\eta+g^{-1}) e^{v_{\varepsilon}(\psi_{m})}e^{\phi}=(s+\frac{s'^{2}}{u''s-s''})e^{-t}e^{u}\circ(-v_{\varepsilon}(\psi_{m}))$.

Summarizing the above discussion about $s$ and $u$, we are naturally led to a
system of ODEs (see \cite{guan-zhou12,guan-zhou13p,guan-zhou13ap,GZopen-effect}):
\begin{equation}
\label{GZ}
\begin{split}
&1).\,\,\left(s+\frac{s'^{2}}{u''s-s''}\right)e^{u-t}=\frac{1}{c_l(t)}, \\
&2).\,\,s'-su'=1,
\end{split}
\end{equation}
where $t\in(T,+\infty)$.

It is not hard to solve the ODE system \eqref{GZ} and get $u(t)=-\log(\int_{S}^{t}c_l(t_{1})e^{-t_{1}}dt_{1})$ and
$s(t)=\frac{\int_{S}^{t}(\int_{S}^{t_{2}}c_l(t_{1})e^{-t_{1}}dt_{1})dt_{2}}{\int_{S}^{t}c_l(t_{1})e^{-t_{1}}dt_{1}}$
(see \cite{guan-zhou13ap}).
It follows that $s\in C^{\infty}((S,+\infty))$ satisfies $s>0$ and $s'>0$,
$u\in C^{\infty}((S,+\infty))$ satisfies $u''s-s''>0$.

As $u(t)=-\log(\int_{S}^{t}c_l(t_{1})e^{-t_{1}}dt_{1})$ is decreasing with respect to $t$,
then it follows from $-S\geq v(t)\geq\max\{t,-t_{0}-B_{0}\}\geq -t_{0}-B_{0}$ for any $t\leq0$
that
\begin{equation}
\begin{split}
\sup_{M}e^{-u(-v(\psi))}
\leq\sup_{t\in(S,t_{0}+B]}e^{-u(t)}
=\int_{S}^{t_{0}+B}c_l(t_{1})e^{-t_{1}}dt_{1},
\end{split}
\end{equation}
then it follows from inequality \eqref{equ:20171122a} and inequality \eqref{equ:20171122s} that
\begin{equation}
	\label{eq:2021329a}
	\int_{M}|F_{l}-(1-b_{t_0,B}(\psi)){F}|^{2}e^{v_{t_0,B}(\psi)-\varphi}c_l(-v_{t_0,B}(\psi))
	\leq C\int_{S}^{t_{0}+B}c_l(t_{1})e^{-t_{1}}dt_{1}.
\end{equation}

\

\emph{Step 7: non-smooth function $c$}

\

By the construction of $c_l$ in Step $1$, we have
\begin{equation}
	\label{eq:2021329b}
	\begin{split}
		&\int_{S}^{t_0+B}c_l(t_1)e^{-t_1}dt_1\\
		=&\int_{S}^{t_0+B}\int_{\mathbb{R}}h((t_1-S)e^y+S)g_l(y)dydt_1\\
		=&\int_{\mathbb{R}}e^{-y}g_l(y)\int_S^{(t_0+B-S)e^y+S}h(s)dsdy\\
		=&\int_{\mathbb{R}}e^{-y}g_l(y)dy\int_S^{t_0+B}h(s)ds+\int_{\mathbb{R}}e^{-y}g_l(y)\int_{t_0+B}^{(t_0+B-S)e^y+S}h(s)dsdy	.		
	\end{split}
\end{equation}
As
\begin{displaymath}
\begin{split}
	&\lim_{l\rightarrow+\infty}\left|\int_{\mathbb{R}}e^{-y}g_l(y)dy-1\right|\\\leq &\lim_{l\rightarrow+\infty}\left|\int_{-\frac{1}{l}}^0(e^{-y}-1)g_l(y)dy\right|+\lim_{l\rightarrow+\infty}\left|\int_0^{\frac{1}{l}}e^{-y}g_l(y)dy\right|\\
	=&0	
\end{split}
\end{displaymath}
and
\begin{displaymath}
	\begin{split}
	&\left|\int_{\mathbb{R}}e^{-y}g_l(y)\int_{t_0+B}^{(t_0+B-S)e^y+S}h(s)dsdy\right|\\
	\leq&e^{\frac{1}{l}}\left(1+\frac{1}{l}\right)h((t_0+B-S)e^{-1}+S)(t_0+B-S)(e^{\frac{1}{l}}-e^{-\frac{1}{l}}),		
	\end{split}
\end{displaymath}
then it follows from inequality \eqref{eq:2021329b} that
\begin{equation}
	\label{eq:2021329c}
	\lim_{l\rightarrow+\infty}\int_{S}^{t_0+B}c_l(t_1)e^{-t_1}dt_1=\int_{S}^{t_0+B}c(t_1)e^{-t_1}dt_1.
\end{equation}
Combining with $\inf_l\inf_{M}e^{v_{t_0,B}(\psi)-\varphi}c_l(-v(\psi))\geq\inf_{M}e^{v_{t_0,B}(\psi)-\varphi}c(-v(\psi))>0$, we obtain that
$$\sup_{l}\int_{M}|F_{l}-(1-b_{t_0,B}(\psi)){F}|^{2}<+\infty.$$
Note that
\begin{equation}
\label{eq:2021329d}
\int_{M}|(1-b_{t_0,B}(\psi)){F}|^{2}\leq\int_{M}|\mathbb{I}_{\{\psi<-t_{0}\}}F|^{2} <+\infty,
\end{equation}
then $\sup_{l}\int_{M}|F_{l}|^{2}<+\infty$,
which implies that there exists a compactly convergent subsequence of $\{F_{l}\}$ (also denoted by $\{F_{l}\}$),
which converges to a holomorphic (n,0) form $\tilde{F}$ on $M$.
Then it follows from inequality \eqref{eq:2021329a} and the Fatou's Lemma that
\begin{displaymath}
\begin{split}
&\int_{M}|\tilde F-(1-b_{t_0,B}(\psi)){F}|^{2}e^{v_{t_0,B}(\psi)-\varphi}c(-v_{t_0,B}(\psi))\\
\leq&\int_{M}|\tilde F-(1-b_{t_0,B}(\psi)){F}|^{2}e^{v_{t_0,B}(\psi)-\varphi}c^-(-v_{t_0,B}(\psi))\\
=&\int_{M}\liminf_{l\rightarrow+\infty}|F_{l}-(1-b_{t_0,B}(\psi)){F}|^{2}e^{v_{t_0,B}(\psi)-\varphi}c_l(-v_{t_0,B}(\psi))
\\\leq&\liminf_{l\rightarrow+\infty}\int_{M}|F_{l}-(1-b_{t_0,B}(\psi)){F}|^{2}e^{v_{t_0,B}(\psi)-\varphi}c_l(-v_{t_0,B}(\psi))\\
\leq&C\liminf_{l\rightarrow+\infty}\int_{S}^{t_0+B}c_l(t_1)e^{-t_1}dt_1
\\
=&C\int_{S}^{t_0+B}c(t_1)e^{-t_1}dt_1.
\end{split}
\end{displaymath}
Thus we prove Lemma \ref{lem:GZ_sharp}.

\subsection{Proof of Lemma \ref{l:cu}}
\label{sec:cu}
The proof is from \cite{GM-concave} with a few minor modifications.

Choose $p\in supp T\cap U$. By Lemma \ref{l:G-compact}, there exist  a real number $t>0$ and a coordinate $(V,w)$, such that $w(p)=0$, $w(V)\cong B(0,1)$ and $V\subset\subset\{G_{\Omega}(z,p)<-t\}\subset\subset U$. There exists a cut-off function $\theta$ on $\Omega$, such that $\theta\equiv1$ on $w^{-1}(B(0,\frac{1}{4}))$ and $supp\theta\subset\subset w^{-1}(B(0,\frac{1}{2}))$.

Let $\tilde{T}=\theta T$, then $\tilde{T}$ is a closed positive $(1,1)$ current on $\Omega$ with
 $supp \tilde{T}\subset\subset w^{-1}(B(0,\frac{1}{2}))$ and $\tilde{T}\not\equiv0$. Now, we prove that exists a subharmonic function $\Phi<0$ on $\Omega$, which satisfies the following properties: $i\partial\bar\partial\Phi= \tilde{T}$; $\lim_{t\rightarrow0+0}(\inf_{\{G_{\Omega}(z,z_0)\geq-t\}}\Phi(z))=0$; $\inf_{\Omega\backslash U}\Phi>-\infty$. Then $\Phi$ satisfies the requirements in Lemma \ref{l:cu}.

\

\emph{Step 1: Construct $\Phi$.}

\

Let $\rho\in C^{\infty}(\mathbb{C})$ be a function with
$supp \rho \subset B(0,\frac{1}{2})$ and $\rho(z)$ depends only on
$|z|$, $\rho \ge 0$ and
$\int_{\mathbb{C}}\rho(z)d\lambda_z=1$. Let
$\rho_{n}(z)=n\rho(nz)$, $\rho_n$ is a
family of smoothing kernels.

As $w(V)\cong B(0,1)$, without misunderstanding we see $(V,z_1)$ and $(B(0,1),w)$ the same.
As $supp \tilde{T}\subset\subset w^{-1}(B(0,\frac{1}{2}))$ and $supp \rho \subset B(0,\frac{1}{2})$, denote that
$T_n=\tilde{T}\ast \rho_n$ be the convolution of $\tilde{T}$. In fact, for
  any test function $h\in C_c^{\infty}(\Omega)$, $((h\circ w^{-1})\ast
\rho_n)(w)$ ($h\ast\rho_n(w)$ for short) is well defined on $w^{-1}(B(0,\frac{1}{2}))$, and
$\langle T_n(z_1),h(z_1)\rangle=<\tilde{T}(w), h\ast
\rho_n(w)>$. Then $T_n$ is a smooth closed positive $(1,1)$ current on $\Omega$ with
 $supp T_n\subset\subset w^{-1}(B(0,\frac{1}{2}+\frac{1}{2n}))$.

Let $u_n(z)=\langle
T_n(z_1),\frac{1}{\pi}G_{\Omega}(z,z_1)\rangle$. $G_{\Omega}(z,z_1)$ is locally integrable with respect $z_1\in\Omega$ for any fixed $z\in\Omega$ implies that $u_n(z)>-\infty$ for any $z\in\Omega$. For fixed $z$ and fixed $n$, we
will prove $\langle T_n(z_1),\frac{1}{\pi}G_{\Omega}(z,z_1)\rangle=\langle
\tilde{T}(w),(\frac{1}{\pi}G_{\Omega}(z,\cdot)\ast\rho_n)(w) \rangle$.
For fixed $z$, $G_{\Omega}(z,z_1)$ is a subharmonic function on $\Omega$.
There exists a sequence of smooth subharmonic functions
$G_m(z_1)$ decreasingly converge to $G_{\Omega}(z,z_1)$ with respect
to $m$. As
$G_m(z_1)$ is smooth, we have
\begin{equation}\label{eq:210623a}
\langle T_n(z_1),\frac{1}{\pi}G_m(z_1)\rangle=<\tilde{T}(w),
\frac{1}{\pi}G_m\ast \rho_n(w)>.
\end{equation}
As $\tilde{T}$ and $T_n$ are closed positive $(1,1)$ current
on $\Omega$ with $supp T_n \subset\subset V$ and $supp \tilde{T} \subset\subset V$, and $G_m(z_1)$
decreasingly converge to $G_{\Omega}(z,z_1)<0$ with respect to $m$ on $\Omega$, it
follows from Levi's Theorem and equality \eqref{eq:210623a} that
\begin{displaymath}
	\begin{split}
		\langle T_n(z_1),\frac{1}{\pi}G_{\Omega}(z,z_1)\rangle=&\lim _{m \to +\infty}\langle
T_n(z_1),\frac{1}{\pi}G_m(z_!)\rangle\\
=&\lim _{m \to +\infty}<\tilde{T}(w),
\frac{1}{\pi}G_m\ast \rho_n(w)>\\
=&<\tilde{T}(w),
\frac{1}{\pi}G_{\Omega}(z,\cdot)\ast \rho_n(w)>.
	\end{split}
\end{displaymath}

Fixed $z\in\Omega$, as $\frac{1}{\pi}G_{\Omega}(z,z_1)$ is subharmonic, then $\frac{1}{\pi}G_{\Omega}(z,\cdot)\ast\rho_n$ is
decreasingly  convergent  to $\frac{1}{\pi}G_{\Omega}(z,z_1)$ with respect to $n$.
Note that $\tilde{T}$ is a positive $(1,1)$ current on $\Omega$, then
$u_n(z)$ is decreasing with respect to $n$.
Let $\Phi(z)=\lim_{n\to +\infty}u_n(z)$. $G_{\Omega}(z,z_1)<0$ on $\Omega\times\Omega$ shows that $u_n(z)<0$ and $\Phi(z)<0$ on $\Omega$.

\

\emph{Step 2: $i\partial\bar{\partial}\Phi=\tilde{T}$.}

\

Firstly, we show that both $\{u_n\}$ and $\Phi$ is $L^1_{loc}$
function on $\Omega$. As $u_n\leq0$ on $\Omega$ and $u_n$ is decreasingly convergent to $\Phi$ with respect to $n$ on $\Omega$,  it suffices to prove that for any $q\in\Omega$ there exists an open subset $K\subset\subset\Omega$, such that $q\in K$ and $\int_{K}|u_n|dV_{\Omega}\leq C$, where $dV_{\Omega}$ is some continuous volume form and $C$ is a constant which independent of $n$.

It is clear that there exists a compact subset $D$ of $V$ such that $supp \tilde{T}\subset D$ and $supp T_n\subset D$ for any $n$. When $q\not\in V$, where exists a coordinate $w_1$ on a neighborhood $V'$ of $q$, such that $w_1(q)=0$, $V'\subset\subset \Omega$, $w_1(V')\cong B(3,1)$ and $\overline{V'}\cap D =\emptyset$. Note that for any $(z,z_1)\in V'\times D$, $G(z,z_1)<0$ on $\Omega\times\Omega$, $G(z,z_1)$ is harmonic with respect to  $z$ or $z_1$ when fixed another one and $\sum_{z_1\in V}|G(q,z_1)|<+\infty$.
Without loss of generality, we see $(V',z)$ and $(B(3,1),w_1)$ the same and assume that $dV_{\Omega}=d\lambda_z$ on $V'$, where $d\lambda_z$ is the Lebesgue measure on $\mathbb{C}$. Then we have
\begin{equation}
	\label{eq:210623c}
	\begin{split}
		\int_{V'}|u_n|dV_{\Omega}&=\frac{1}{\pi}\int_{z\in V'}\int_{z_1\in V}|G_{\Omega}(z,z_1)|T_n(z_1)d\lambda_z\\
		&=\frac{1}{\pi}\int_{z_1\in V}\int_{z\in V'}|G_{\Omega}(z,z_1)|d\lambda_zT_n	(z_1)\\
		&=\frac{1}{\pi}\int_{z_1\in V}\pi|G_{\Omega}(q,z_1)|T_n(z_1)\\
		&\leq\|T_n\|\sup_{z_1\in V}|G_{\Omega}(q,z_1)|\\
		&=\|\tilde{T}\|\sup_{z_1\in V}|G_{\Omega}(q,z_1)|.
		\end{split}
\end{equation}

When $q\in V$, $G_{\Omega}(w,\tilde{w})=\log|w-\tilde{w}|+v(w,\tilde{w})$ on $V\times V$, where $v(w,\tilde{w})$ is harmonic with respect to  $w$ or $\tilde{w}$ when fixed another one. Without loss of generality, we see $(V,z)$ and $(B(0,1),w)$ the same and assume that $dV_{\Omega}=d\lambda_w$ on $V$, where $d\lambda_w$ is the Lebesgue measure on $\mathbb{C}$. Then we have
\begin{equation}
	\label{eq:210623d}
	\begin{split}
		\int_{V}u_ndV_{\Omega}&=\frac{1}{\pi}\int_{w_\in V}\int_{\tilde{w}\in V}G_{\Omega}(w,\tilde{w})T_n(\tilde{w})d\lambda_w\\
		&=\frac{1}{\pi}\int_{\tilde{w}\in V}\int_{w\in V}G_{\Omega}(w,\tilde{w})d\lambda_wT_n(\tilde{w})	\\
		&=\frac{1}{\pi}\int_{\tilde{w}\in V}\int_{w\in V}\log|w-\tilde{w}|d\lambda_wT_n(\tilde{w})+\frac{1}{\pi}\int_{\tilde{w}\in V}\int_{w\in V}v(w,\tilde{w})d\lambda_wT_n(\tilde{w}).	\end{split}
\end{equation}
Note that
$$\int_{w\in V}\log|w-\tilde{w}|d\lambda_w\geq-\int_{w\in B(0,2)}|\log|w||d\lambda_w>-\infty$$
holds for any $\tilde{w}\in V$,
$$\int_{w\in V}v(w,\tilde{w})d\lambda_w=\pi v(q,\tilde{w})$$
holds for any $\tilde{w}\in V$ and $\inf_{\tilde{w}\in V}v(q,\tilde{w})>-\infty$, then equality \eqref{eq:210623d} implies that there exists a constant $N>0$ such that
\begin{equation}
	\label{eq:210623e}
	\int_{V}u_ndV_{\Omega}\geq N\|T_n\|.
\end{equation}
By the definition of $T_n$, we know $\|T_n\|=\|\tilde{T}\|<+\infty$. As $u_n\leq0$, combining inequality \eqref{eq:210623c} and \eqref{eq:210623e}, we obtain that any $q\in\Omega$ there exists an open subset $K\subset\subset\Omega$, such that $q\in K$ and $\int_{K}|u_n|dV_{\Omega}\leq C$, where $dV_{\Omega}$ is some continuous volume form and $C$ is a constant which independent of $n$. Hence, we know $\{u_n\} \in L^1_{loc}(\Omega)$ and $\Phi\in L^1_{loc}(\Omega)$.

Now we consider $i\partial\bar{\partial}\Phi$. Let $g\in
C^{\infty}_c(X)$ be a test function. It follows from $\Phi\in L_{loc}^1(\Omega)$ and the dominated convergence Theorem that
\begin{equation}
	\label{eq:2106j}
	\begin{split}
	\langle i\partial\bar{\partial}\Phi,g\rangle
&=\langle \Phi(z),i\partial\bar{\partial}g(z)\rangle\\
&=\lim_{n\to+\infty}\langle
u_n(z),i\partial\bar{\partial}g(z)\rangle. 	
	\end{split}
\end{equation}
As $u_n(z)=\langle
T_n(z_1),\frac{1}{\pi}G_{\Omega}(z,z_1)\rangle$, using Fubini's Theorem, equality \eqref{eq:2106j} becomes
\begin{equation}
	\label{eq:2106k}
	\begin{split}
	\langle i\partial\bar{\partial}\Phi,g\rangle
&=\lim_{n\to+\infty}\langle \langle
T_n(z_1),\frac{1}{\pi}G_{\Omega}(z,z_1)\rangle,i\partial\bar{\partial}g(z)\rangle\\
&=\lim_{n\to+\infty}\langle T_n(z_1),\langle
\frac{1}{\pi}G_{\Omega}(z,z_1),i\partial\bar{\partial}g(z)\rangle\rangle. 	
	\end{split}
\end{equation}
Since $T_n$ is positive $(1,1)$ current on $\Omega$, $T_n$ converge weakly to $\tilde{T}$ and $\frac{i}{\pi}\partial_z\bar\partial_z G_{\Omega}(z,z_1)=\delta_{z_1}$, it follows from equality \eqref{eq:2106k} that
\begin{equation}\label{eq:210623f}
\begin{split}
\langle i\partial\bar{\partial}\Phi,g\rangle
&=\lim_{n\to+\infty}\langle T_n(z_1),\langle
\frac{1}{\pi}G_{\Omega}(z,z_1),i\partial\bar{\partial}g(z)\rangle\rangle\\
&=\lim_{n\to+\infty}\langle T_n(z_1),g(z_1)\rangle\\
&=\langle \tilde{T},g\rangle,\\
\end{split}
\end{equation}
which implies that $i\partial\bar\partial\Phi=\tilde{T}$.

\

\emph{Step 3: $\lim_{t\rightarrow0+0}(\inf_{\{G_{\Omega}(z,z_0)\geq-t\}}\Phi(z))=0$ and $\inf_{\Omega\backslash U}\Phi>-\infty$}

\

Let $W\subset\subset \Omega$ be an open set
of $\Omega$ which satisfies $\overline{V}\cup\{z_0\} \subset W$
and $\overline{W}\cap \{-t\leq G_{\Omega}(z,z_0)\}=\emptyset$, where $t$ is a small enough positive number.
Then for every fixed $z\in\{-t\leq G_{\Omega}(z,z_0)\}$,
$G_{\Omega}(z,z_1)$ is harmonic function on a neighborhood of $\overline{W}$ with respect to
$z_1$.
By the Harnack inequality of harmonic function, there
exists a $M>0$ such that
$$\sup _{z_1\in \overline{W}}(-G_{\Omega}(z,z_1))\le
M\inf _{z_1\in \overline{W}}(-G_{\Omega}(z,z_1)).$$
As $z\in\{-t\leq G_{\Omega}(z,z_0)\}$, we have
$$Mt > -MG_{\Omega}(z,z_0)\ge M\inf _{z_1\in
\overline{W}}(-G_{\Omega}(z,z_1)\ge\sup _{z_1\in
\overline{W}}(-G_{\Omega}(z,z_1)\ge 0,$$
which means that $\lim_{t\rightarrow0+0}(\inf_{{\{G_{\Omega}(z,z_0)\geq-t\}}\times \overline{W}}G_{\Omega}(z,z_1))=0$.

Note that $0\geq u_n(z)=\langle T_n(z_1),\frac{1}{\pi}G_{\Omega}(z,z_1)\rangle\geq\frac{1}{\pi}\inf_{{\{G_{\Omega}(z,z_0)\geq-t\}}\times \overline{W}}G_{\Omega}(z,z_1)\|T_n\|$ holds for any $n$ and $z\in\{-t\leq G_{\Omega}(z,z_0)\}$, as $\|T_n\|=\|\tilde{T}\|<+\infty$ and $u_n$ is decreasingly convergent to $\Phi$, then we have
\begin{displaymath}
		\lim_{t\rightarrow0+0}(\inf_{\{G_{\Omega}(z,z_0)\geq-t\}}\Phi(z))\geq\lim_{t\rightarrow0+0}\frac{1}{\pi}\inf_{{\{G_{\Omega}(z,z_0)\geq-t\}}\times \overline{W}}G_{\Omega}(z,z_1)\|\tilde{T}\|=0	.
\end{displaymath}

Next, we prove $\inf_{\Omega\backslash U}\Phi>-\infty$.
Note that $p\in V\subset\subset\{G_{\Omega}(z,p)<-t\}\subset\subset U\subset\subset \Omega$, it follows from Lemma \ref{l:lo} that there exists a constant $N>0$, such that
\begin{equation}
	\label{eq:210623h}
	G_{\Omega}(z,z_1)\geq NG_{\Omega}(z,p)\geq-Nt
\end{equation}
holds for any $(z,z_1)\in(\Omega\backslash U,V)$. As $u_n(z)=\langle T_n(z_1),\frac{1}{\pi} G_{\Omega}(z,z_1)\rangle$ and $supp T_n\subset\subset V$ for any $n$, then we have $u_n(z)\geq-\frac{Nt}{\pi}\|T_n\|$ hold on $z\in\Omega\backslash U$. Note that $\|T_n\|=\|\tilde{T}\|$ and $u_n$ is decreasingly convergent to $\Phi$, then we have $\inf_{\Omega\backslash U}\Phi>-\infty$.

Thus, Lemma \ref{l:cu} holds.

\vspace{.1in} {\em Acknowledgements}.
The authors would like to thank Dr. Shijie Bao and Dr. Zhitong Mi for checking the manuscript and  pointing out some typos. The first named author was supported by National Key R\&D Program of China 2021YFA1003103, NSFC-11825101, NSFC-11522101 and NSFC-11431013.

\bibliographystyle{references}
\bibliography{xbib}

\end{document}